\documentclass[12pt]{amsart}%
\usepackage{amsfonts}
\usepackage{amsmath}
\usepackage{amssymb}
\usepackage{graphicx}%
\providecommand{\U}[1]{\protect\rule{.1in}{.1in}}
\newtheorem{theorem}{Theorem}
\theoremstyle{plain}

\newtheorem{corollary}{Corollary}

\newtheorem{lemma}{Lemma}

\newtheorem{proposition}{Proposition}
\newtheorem{remark}{Remark}

\numberwithin{equation}{section}
\numberwithin{equation}{section}
\numberwithin{theorem}{section}
\numberwithin{lemma}{section}
\numberwithin{remark}{section}
\numberwithin{example}{section}
\numberwithin{proposition}{section}
\numberwithin{definition}{section}
\numberwithin{corollary}{section}
\begin{document}
\title[Projective positivity]{Projective positivity of the function systems}
\author{Anar Dosi}
\address[Dosi=Dosiev]{ Middle East Technical University NCC, Guzelyurt, KKTC, Mersin
10, Turkey }
\email{dosiev@yahoo.com, dosiev@metu.edu.tr}
\urladdr{http://www.mathnet.ru/php/person.phtml?option\_lang=eng\&personid=23380}
\date{March 3, 2023}

\begin{abstract}
The present paper is devoted to the projective positivity in the category of
function systems, which plays a key role in the quantization problems of the
operator systems. The main result of the paper asserts that every unital
$\ast$-normed space can be equipped with the projective positivity. The
geometry of the related state spaces is described in the case of $L^{p}%
$-spaces, Schatten matrix spaces, and $L^{p}$-spaces of a finite von Neumann algebra.

\end{abstract}
\maketitle

\section{\textbf{Introduction\label{Sec1}}}

A concrete function system is a unital self-adjoint subspace of a unital
commutative $C^{\ast}$-algebra whereas \textit{an abstract function system} is
defined as an ordered $\ast$-vector space $\left(  V,e\right)  $ with an
Archimedian order unit $e$. A function system turns out to be a unital $\ast
$-normed space equipped with an order (semi)norm, and all order norms in $V$
are equivalent by running within min and max order norms \cite[Theorems 4.5,
4.7]{PT}. In the theory of Banach lattices, the function systems are referred
as the ordered vector spaces with their strong order units \cite[page 18]%
{PMN}. Every abstract function system can be realized up to an isometry as a
concrete one on its state space equipped with the $w^{\ast}$-topology. That is
the complex version of Kadison's representation theorem for the real function
systems \cite[Theorem 5.2]{PT}. Recall that the state space $S\left(
V_{+}\right)  $ of an ordered $\ast$-vector space $V$ with an Archimedian
order unit $e$ and a unital cone $V_{+}$ of its positive elements is defined
as the intersection $V_{he}^{\ast}\cap\operatorname{ball}V^{\ast}$ of the real
hyperplane $V_{he}^{\ast}=\left\{  \varphi\in V^{\ast}:\varphi=\varphi^{\ast
},\left\langle e,\varphi\right\rangle =1\right\}  $ in the norm dual space
$V^{\ast}$ given by the functional $\left\langle e,\cdot\right\rangle
:V^{\ast}\rightarrow\mathbb{C}$, and the unit ball of $V^{\ast}$. It turns out
that the cone $V_{+}$ is precisely defined by means of its state space
$S\left(  V_{+}\right)  $, namely $v\geq0$ iff $\left\langle v,S\left(
V_{+}\right)  \right\rangle \geq0$ \cite[Proposition 2.20]{PT}. Kadison's
representation is given by the unital mapping
\[
\Phi:V\rightarrow C\left(  S\left(  V_{+}\right)  \right)  ,\quad\Phi\left(
v\right)  \left(  \varphi\right)  =\left\langle v,\varphi\right\rangle ,
\]
where $C\left(  S\left(  V_{+}\right)  \right)  $ is the $C^{\ast}$-algebra of
all complex continuous functions on the Hausdorff compact topological space
$S\left(  V_{+}\right)  $. A detailed investigation of the function systems
(by introducing the concept of a function system itself) was carried out in
\cite{PT} by Paulsen and Tomforde.

In quantum functional analysis \cite{ER}, \cite{Hel} the function systems are
treated as the commutative operator (or quantum) systems, whereas the operator
systems equipped with various unital quantum cones are quantizations of the
underlying function system. An operator system is defined as a unital
self-adjoint subspace of the operator algebra $\mathcal{B}\left(  H\right)  $
of all bounded linear operators on a complex Hilbert space $H$, whereas an
operator space is a subspace of $\mathcal{B}\left(  H\right)  $. The
quantization principles of a unital cone in a unital $\ast$-vector space were
developed in \cite{PaulTT} by Paulsen, Todorov and Tomforde. These
constructions are the categorical background of the abstract operator systems
followed by the key result of Choi and Effros \cite{ECh} on an abstract
characterization of the operator systems (see also \cite[Ch. 13]{Pa}). A
classification problem of the operator systems among the operator space
structures on a unital $\ast$-vector space was investigated in \cite{DSbM}. It
turns out that every operator system construction can be pushed through an
operator space one that results in a suitable unital quantum cone to be
related to. That unital quantum cone in turn generates a new function system
structure on the underlying unital $\ast$-vector space. But practically it is
not clear how to detect the unital cone(s) related to a certain operator space
construction. For example, the projective tensor product or even $1$-sum of
the operator systems does not exist in its proper sense (see \cite{Kav}). The
following simple example could illustrate the case. There is a well defined
$1$-sum $V=\mathbb{C}\overset{1}{\oplus}\mathbb{C}$ of the trivial operator
spaces $\mathbb{C}$ with its operator dual $V^{\ast}=\mathbb{C}\overset{\infty
}{\oplus}\mathbb{C}$. Note that $V$ possesses the natural involution
$z=\left(  z_{0},z_{1}\right)  \mapsto z^{\ast}=\left(  z_{0}^{\ast}%
,z_{1}^{\ast}\right)  $ such that $\left\Vert z^{\ast}\right\Vert
_{1}=\left\Vert z\right\Vert _{1}$, and the unit $e=\left(  2^{-1}%
,2^{-1}\right)  $. In this case, $V_{he}^{\ast}=\left\{  w\in\mathbb{R}%
\overset{\infty}{\oplus}\mathbb{R}:w_{0}+w_{1}=2\right\}  $ and $V_{he}^{\ast
}\cap\operatorname{ball}V^{\ast}=\left\{  \left(  1,1\right)  \right\}  $,
that is, the related state space has been degenerated and it could only define
a non-separated cone in $V$. Thus how to define the operator system $1$-sum
$\mathbb{C}\overset{1}{\oplus}\mathbb{C}$ in this trivial case remains unclear?

The construction of the projective objects within the category of the operator
systems causes a serious challenge to tackle with. An operator system quotient
does not always preserve the related operator space quotient (see discussion
in \cite[Section 4]{Kav2}). Actually, the problem remains open on the level of
the function systems. In the previous example $\mathbb{C}$ is a function
system. What are the projective function systems that is the main target of
the present paper.

We provide a complete solution to the problem by developing the general
concept of \textit{the projective positivity}. Recall that a norm on a $\ast
$-vector space $V$ is said to be a $\ast$\textit{-norm} if the unit
$\operatorname{ball}V$ is hermitian, that is, $\left(  \operatorname{ball}%
V\right)  ^{\ast}=\operatorname{ball}V$. In this case, $\left(  V,V^{\ast
}\right)  $ is a dual $\ast$-pair \cite{Dhjm}, and $V_{h}^{\ast}$ denotes the
real vector space of all bounded hermitian functionals on $V$. Fix a nonzero
$e\in V_{h}$ and as above put $V_{he}^{\ast}=\left\{  y\in V_{h}^{\ast
}:\left\langle e,y\right\rangle =1\right\}  $. Then $S=V_{he}^{\ast}%
\cap\operatorname{ball}V^{\ast}$ is a convex subset of $V_{h}^{\ast}$. We say
that $e$ is \textit{a unit element }(or just \textit{unit})\textit{ for }%
$V$\textit{ }if $S\neq\varnothing$. Every hermitian unit vector $e$ of a
$\ast$-normed space $V$ is a unit automatically. A $\ast$-normed vector space
$V$ with a fixed unit $e$ is called \textit{a unital }$\ast$\textit{-normed
space }$\left(  V,e\right)  $. The set $S$ can be poor enough to define an
order in $V$ as we have seen above. That is the case in many other basic
examples of analysis (see below Sections \ref{secLP}, \ref{secVNM}).

If $\operatorname{ball}V$ is equivalent to the absolute polar $S^{\circ}$ of
$S$, we say that $\operatorname{ball}V$ is \textit{a unital ball}, and
$\left(  V,e\right)  $ is again called\textit{ an abstract function system}
\textit{with its state space }$S$\textit{. }In this case, $\mathfrak{c=}%
\left\{  v\in V_{h}:\left\langle v,S\right\rangle \geq0\right\}  $ is a
separated, closed, unital cone in $V$ such that $S\left(  \mathfrak{c}\right)
=S$ and $\left\Vert \cdot\right\Vert \sim\left\Vert \cdot\right\Vert _{e}$,
where $\left\Vert v\right\Vert _{e}=\sup\left\vert \left\langle
v,S\right\rangle \right\vert $, $v\in V$ (the min-norm). Thus our function
system is a bit more general concept than one considered above. Using again
Kadison's representation, one can identify $V$ with a unital self-adjoint
subspace of $C\left(  S\left(  \mathfrak{c}\right)  \right)  $ up to an order
and topological isomorphism, that is, $\left(  V,\mathfrak{c}\right)  $ is a
concrete function system on $S\left(  \mathfrak{c}\right)  $. If
$\operatorname{ball}V=S^{\circ}$ then $\left\Vert \cdot\right\Vert =\left\Vert
\cdot\right\Vert _{e}$ and we come up with the isometric embedding
$\Phi:V\rightarrow C\left(  S\left(  \mathfrak{c}\right)  \right)  $.

In the case of a non-unital $\operatorname{ball}V$ we consider the following
convex sets $S_{\varepsilon}=V_{he}^{\ast}\cap\varepsilon\operatorname{ball}%
V^{\ast}$ for $\varepsilon>0$. Thus $\operatorname{ball}V$ is unital whenever
$\operatorname{ball}V\sim S_{1}^{\circ}$. But if $\operatorname{ball}V$ is not
unital then we prove that%
\[
\operatorname{ball}V\sim S_{\varepsilon}^{\circ}\text{\quad for every
}\varepsilon>1/\left\Vert e\right\Vert
\]
(see Theorem \ref{propUNS1}). The main result of the paper asserts that
\[
\mathfrak{c}_{\varepsilon}=\left\{  v\in V_{h}:\left\langle v,S_{\varepsilon
}\right\rangle \geq0\right\}
\]
is a closed, unital, cone in $V$ whose state space $S\left(  \mathfrak{c}%
_{\varepsilon}\right)  =S_{\varepsilon}$, and $\left(  V,\mathfrak{c}%
_{\varepsilon}\right)  $ is a concrete function system on $S_{\varepsilon}$
whose $\varepsilon$-norm $\left\Vert v\right\Vert _{\varepsilon}%
=\vee\left\vert \left\langle v,S_{\varepsilon}\right\rangle \right\vert $,
$v\in V$ is equivalent to the original one of $V$ for every $\varepsilon
>1/\left\Vert e\right\Vert $. Thus without changing the original topology of a
unital $\ast$-normed space $\left(  V,e\right)  $ one can endow it with a
compatible function system structure called \textit{a projective function
system}.

As a practical application of the main result we consider a various examples
including $L^{p}$-spaces, the matrix Schatten-spaces and noncommutative
$L^{p}$-spaces of a finite von Neumann algebra. In Section \ref{secLP} we
discuss the case of classical $L_{\mu}^{p}\left(  \mathcal{T}\right)  $,
$1\leq p\leq\infty$ spaces. The unit vectors $e_{p}=\mu\left(  \mathcal{T}%
\right)  ^{-1/p}1$ are indeed units of the related Banach $\ast$-spaces
$L_{\mu}^{p}\left(  \mathcal{T}\right)  $. But their balls
$\operatorname{ball}L_{\mu}^{p}\left(  \mathcal{T}\right)  $ fail to be unital
for all $1\leq p<\infty$. For $p=\infty$ we come up with the abelian von
Neumann algebra $L_{\mu}^{\infty}\left(  \mathcal{T}\right)  $ with the unit
$e_{\infty}=1$, and $\left(  L_{\mu}^{\infty}\left(  \mathcal{T}\right)
,e_{\infty}\right)  $ is an operator system, which in turn is a function
system too. For a conjugate couple $\left(  p,q\right)  $ with $p<\infty$ and
$\varepsilon\geq1$ we define the following convex subsets
\[
S_{\varepsilon}^{p}=\left\{  f\in L^{q}\left(  \mathcal{T}\right)
_{h}:\left\Vert f\right\Vert _{q}\leq\varepsilon,\int fd\mu=\mu\left(
\mathcal{T}\right)  ^{1/p}\right\}  \subseteq\partial\left(  \varepsilon
\operatorname{ball}L^{q}\left(  \mathcal{T}\right)  \right)  .
\]
The state space $S_{1}^{p}$ in $\operatorname{ball}L_{\mu}^{q}\left(
\mathcal{T}\right)  $ is again degenerated to a singleton $\left\{
e_{q}\right\}  $. Nonetheless the unital $\ast$-normed space $L_{\mu}%
^{p}\left(  \mathcal{T}\right)  $ has the projective positivity given by the
separated, closed, unital cone $L_{\mu}^{p}\left(  \mathcal{T}\right)
_{\varepsilon}^{+}$ whose state space is $S_{\varepsilon}^{p}$ for every
$\varepsilon>1$. It turns out that the original cone $L^{p}\left(
\mathcal{T}\right)  _{+}$ of the Banach lattice $L^{p}\left(  \mathcal{T}%
\right)  _{h}$ and the unital cones $L^{p}\left(  \mathcal{T}\right)
_{\varepsilon}^{+}$ are totally different cones, they are not comparable
unless $\operatorname{supp}\left(  \mu\right)  $ is finite (see Theorem
\ref{thmainComm}). Moreover, in the finite dimensional nontrivial case we have
$\ell^{p}\left(  n\right)  _{+}=\ell^{p}\left(  n\right)  _{\varepsilon}^{+}$
only if $n=2$ and $\varepsilon=2^{1/p}$.

In Section \ref{secVNM} we first consider the full matrix algebra
$\mathbb{M}_{n}$ with the unit matrix $e$ for $n>1$. The algebra
$\mathbb{M}_{n}$ equipped with the Schatten $p$-norm is denoted by
$\mathfrak{S}_{p}$, $1\leq p\leq\infty$. In this case, $\mathfrak{S}_{\infty
}^{\ast}=\mathfrak{S}_{1}$, and $\mathfrak{S}_{\infty}^{+}$ is the standard
unital cone of the positive elements with its state space $S_{\tau}=\left\{
y\in\mathfrak{S}_{1}:y\geq0,\tau\left(  y\right)  =1\right\}  $ of the density
matrices from $\mathbb{M}_{n}$, that is, $\left(  \mathfrak{S}_{\infty
},e\right)  $ is a function system. But that is not the case for the rest
$\left(  \mathfrak{S}_{p},e\right)  $, which are unital $\ast$-normed spaces
with their common unit $e$. Nonetheless $\left(  \mathfrak{S}_{p},e\right)  $
is a function system equipped with the projective cone $\mathfrak{S}_{p}^{+}$
given by $\varepsilon=1>1/\left\Vert e\right\Vert _{p}$ such that
$\mathfrak{S}_{\infty,\varepsilon}^{+}\subseteq\mathfrak{S}_{1}^{+}%
\subseteq\mathfrak{S}_{p}^{+}\subseteq\mathfrak{S}_{\infty}^{+}$,
$\varepsilon\geq n$ is an increasing family of the unital cones with their
state spaces
\[
S\left(  \mathfrak{S}_{p}^{+}\right)  =\left\{  \left(  1+r\right)
z-rw:z,w\in S_{\tau},zw=0,\left\Vert \left(  1+r\right)  z+rw\right\Vert
_{q}\leq1\right\}
\]
for all conjugate couples $\left(  p,q\right)  $ (see Theorem \ref{thMP1}).
Moreover, $\mathfrak{S}_{p}^{+}$ is reduced to the original cone of
$\mathfrak{S}_{p}$ only if $n=2$ and $\varepsilon=2^{1/p}$. Thus only for the
two-level quantum system (qubit) \cite[Ch. 2]{Hol}, the projective positivity
is reduced to the original one precisely. The (finite dimensional) geometry of
the state spaces $S\left(  \mathfrak{S}_{p}^{+}\right)  $ is described in the
proof of Theorem \ref{thMP1}.

Finally, the obtained result for $\mathbb{M}_{n}$ is generalized in Theorem
\ref{thFVNA1} to the case of a finite von Neumann algebra $\mathcal{M}$ with a
faithful, finite, normal trace $\tau$, $\tau\left(  e\right)  >1$. It is also
proved (see Proposition \ref{propFVNA1}) that the unital cone $L^{2}\left(
\mathcal{M},\tau\right)  _{\varepsilon}^{+}$ for $\varepsilon=\left(
2/\tau\left(  e\right)  \right)  ^{1/2}>1/\sqrt{\tau\left(  e\right)  }$ of
the unital Hilbert $\ast$-space $\left(  L^{2}\left(  \mathcal{M},\tau\right)
,e\right)  $ is reduced to one from \cite{Dfaa19}. Thus $L^{2}\left(
\mathcal{M},\tau\right)  _{\varepsilon}^{+}$ is the positivity that stands
behind of Pisier's operator Hilbert space \cite[3.5]{ER}, \cite[7.4]{Hel}.

\section{The preliminary material on operator systems}

To avoid a possible discomfort during the reading of the new material we
present in this section some key results on operators systems by adding up
some new facts too. For some assertions it is really hard to indicate the
related direct references, that is why we support the presentation with
Appendix Section \ref{SecApp1}. For the basic lattice operations $\inf$ and
$\sup$ we use the notations $\wedge$ and $\vee$, respectively. The convex hull
of a subset $S$ in a vector space is denoted by $\operatorname{co}\left(
S\right)  $ whereas $\operatorname{abc}\left(  S\right)  $ denotes the
absolutely convex hull of $S$. The characteristic function of a set $S$ is
denoted by $\left[  S\right]  $.

\subsection{Operator systems\label{subsecOS22}}

By an operator system (or noncommutative function system) $V$ on a Hilbert
space $H$ we mean a unital self-adjoint subspace of $\mathcal{B}\left(
H\right)  $. Thus $V\subseteq\mathcal{B}\left(  H\right)  $ is a subspace such
that $e\in V$ and $V^{\ast}=V$, where $e$ is the unit of $\mathcal{B}\left(
H\right)  $. A unital $C^{\ast}$-algebra is an example of an operator system.
The real subspace of all hermitian elements of an operator system $V$ is
denoted by $V_{h}$, whereas $V_{+}=V\mathcal{\cap}\mathcal{B}\left(  H\right)
_{+}$ is the (algebraically or topologically (see \cite{Dhjm})) closed cone of
its positive elements. The cone $V_{+}$ is separated, that is, $V_{+}%
\cap-V_{+}=\left\{  0\right\}  $. Moreover, it is a \textit{unital cone }in
the sense of that $V_{+}-e$ absorbs all $V_{h}$. Namely, for every $x\in
V_{h}$ we have $x\in\left\Vert x\right\Vert \left(  V_{+}-e\right)  $. Note
also that $x=\left(  x+\left\Vert x\right\Vert e\right)  -\left\Vert
x\right\Vert e\in V_{+}-V_{+}$ for every $x\in V_{h}$, that is, $V_{h}%
=V_{+}-V_{+}$, which in turn implies that $V=\left(  V_{+}-V_{+}\right)
+i\left(  V_{+}-V_{+}\right)  $.

A linear functional $\mu:V\rightarrow\mathbb{C}$ defines its conjugate
$\mu^{\ast}:V\rightarrow\mathbb{C}$ to be $\left\langle x,\mu^{\ast
}\right\rangle =\left\langle x^{\ast},\mu\right\rangle ^{\ast}$, $x\in V$. If
$\mu=\mu^{\ast}$ then $\mu$ is called \textit{a hermitian functional on} $V$,
and the real vector space of all hermitian functionals on $V$ is denoted by
$V_{h}^{\ast}$. Notice that $\mu\in V_{h}^{\ast}$ iff $\left\langle V_{h}%
,\mu\right\rangle \subseteq\mathbb{R}$. Indeed, if $\mu\in V_{h}^{\ast}$ and
$x\in V_{h}$, then $\left\langle x,\mu\right\rangle ^{\ast}=\left\langle
x^{\ast},\mu\right\rangle ^{\ast}=\left\langle x,\mu^{\ast}\right\rangle
=\left\langle x,\mu\right\rangle $, that is, $\left\langle x,\mu\right\rangle
\in\mathbb{R}$. Conversely, $\left\langle V_{h},\mu\right\rangle
\subseteq\mathbb{R}$ implies that
\[
\left\langle x,\mu^{\ast}\right\rangle =\left\langle \operatorname{Re}%
x+i\operatorname{Im}x,\mu^{\ast}\right\rangle =\left\langle \operatorname{Re}%
x-i\operatorname{Im}x,\mu\right\rangle ^{\ast}=\left\langle \operatorname{Re}%
x,\mu\right\rangle +i\left\langle \operatorname{Im}x,\mu\right\rangle
=\left\langle x,\mu\right\rangle
\]
for all $x\in V$, that is, $\mu=\mu^{\ast}$. Moreover,
\begin{equation}
\left\Vert \mu\right\Vert =\vee\left\vert \left\langle \operatorname{ball}%
V_{h},\mu\right\rangle \right\vert \text{ for all }\mu\in V_{h}^{\ast}.
\label{nHf}%
\end{equation}
Indeed, for every $x\in\operatorname{ball}V$ with nonzero $\left\langle
x,\mu\right\rangle $ we have $\left\vert \left\langle x,\mu\right\rangle
\right\vert =\left\langle \theta x,\mu\right\rangle =\operatorname{Re}%
\left\langle \theta x,\mu\right\rangle =\left\langle \operatorname{Re}\left(
\theta x\right)  ,\mu\right\rangle $ for suitable $\theta\in\mathbb{C}$,
$\left\vert \theta\right\vert =1$. But $\left\Vert \operatorname{Re}\left(
\theta x\right)  \right\Vert \leq\left\Vert x\right\Vert \leq1$, therefore
$\left\vert \left\langle x,\mu\right\rangle \right\vert \leq\left\vert
\left\langle \operatorname{ball}V_{h},\mu\right\rangle \right\vert $. Hence
$\left\Vert \mu\right\Vert =\vee\left\vert \left\langle \operatorname{ball}%
V,\mu\right\rangle \right\vert \leq\left\vert \left\langle \operatorname{ball}%
V_{h},\mu\right\rangle \right\vert \leq\left\Vert \mu\right\Vert $.

A linear functional $\mu:V\rightarrow\mathbb{C}$ is said to be
\textit{positive} if $\left\langle V_{+},\mu\right\rangle \geq0$. In this case
$\mu=\mu^{\ast}$ automatically. Indeed, $\left\langle V_{h},\mu\right\rangle
=\left\langle V_{+},\mu\right\rangle -\left\langle V_{+},\mu\right\rangle
\subseteq\mathbb{R}$, which in turn implies $\mu\in V_{h}^{\ast}$. The set of
all positive functionals on $V$ is denoted by $V_{+}^{\ast}$. It turns out to
be a subcone of $V_{h}^{\ast}$. If $V=C\left(  \mathcal{T}\right)  $ is the
abelian $C^{\ast}$-algebra of all continuous functions on a compact Hausdorff
topological space $\mathcal{T}$, then $C\left(  \mathcal{T}\right)  ^{\ast
}=M\left(  \mathcal{T}\right)  $ is the space of all Radon charges on
$\mathcal{T}$. In this case, $M\left(  \mathcal{T}\right)  _{+}$ consists of
all positive Radon measures on $\mathcal{T}$. A Radon charge $\mu\in M\left(
\mathcal{T}\right)  $ is positive (that is, $\mu\geq0$) iff $\left\Vert
\mu\right\Vert =\left\langle 1,\mu\right\rangle =\mu\left(  \mathcal{T}%
\right)  $. The fact can be proven for a linear functional over an operator
system (see below Lemma \ref{lemOS0} from Appendix Section \ref{SecApp1}).

\subsection{The states of an operator system\label{subsecSOS11}}

A unital positive functional on $V$ is called \textit{a state of the cone
}$V_{+}$, and we put
\[
S\left(  V_{+}\right)  =\left\{  \varphi\in V_{+}^{\ast}:\left\langle
e,\varphi\right\rangle =1\right\}
\]
to be \textit{the state space of the cone} $V_{+}$. In the case of $V=C\left(
\mathcal{T}\right)  $ we come up with the convex set $P\left(  \mathcal{T}%
\right)  $ of probability measures on $\mathcal{T}$, that is, $P\left(
\mathcal{T}\right)  =S\left(  C\left(  \mathcal{T}\right)  _{+}\right)  $ (see
Subsection \ref{subsecEP}). We also put $V_{he}^{\ast}=\left\{  \varphi\in
V_{h}^{\ast}:\left\langle e,\varphi\right\rangle =1\right\}  $ to be the real
hyperplane of all unital, hermitian functionals on $V$. Using Lemma
\ref{lemOS0}, we conclude that%
\[
S\left(  V_{+}\right)  =\left\{  \varphi\in\operatorname{ball}V^{\ast
}:\left\Vert \varphi\right\Vert =\left\langle e,\varphi\right\rangle
=1\right\}  =V_{he}^{\ast}\cap\operatorname{ball}V^{\ast}.
\]
In particular, $S\left(  V_{+}\right)  $ is a $w^{\ast}$-compact subset of
$\operatorname{ball}V^{\ast}$. For every unit vector $\zeta\in H$ the
functional $\omega_{\zeta}:V\mathcal{\rightarrow}\mathbb{C}$, $\omega_{\zeta
}\left(  x\right)  =\left(  x\zeta,\zeta\right)  $ belongs to $S\left(
V_{+}\right)  $. In particular,
\begin{equation}
x\in V_{+}\quad\text{iff\quad}\left\langle x,S\left(  V_{+}\right)
\right\rangle \geq0, \label{sp}%
\end{equation}
which means that the operator system positivity is\textit{ a state positivity}.

Using the state space $S\left(  V_{+}\right)  $, one can define a new (state)
norm $\left\Vert \cdot\right\Vert _{e}$ on $V$ in the following way
$\left\Vert x\right\Vert _{e}=\vee\left\vert \left\langle x,S\left(
V_{+}\right)  \right\rangle \right\vert $.

\begin{proposition}
\label{propAp1}The norm $\left\Vert \cdot\right\Vert _{e}$ on $V$ is a unital
$\ast$-norm, which is equivalent to the original norm of $V$. Moreover,
$\left\Vert x\right\Vert _{e}=\left\Vert x\right\Vert $, $x\in V_{h}$, and
$\left\Vert \cdot\right\Vert _{e}$ is an order norm in the sense of
\[
\left\Vert x\right\Vert _{e}=\wedge\left\{  r>0:-re\leq x\leq re\right\}
\text{ for all }x\in V_{h}.
\]
In particular, $\left\Vert \mu\right\Vert _{e}=\left\Vert \mu\right\Vert $ for
every $\mu\in V_{h}^{\ast}$.
\end{proposition}

\begin{proof}
For each $x\in V$ we have
\[
\left\Vert x\right\Vert _{e}\leq\left\Vert x\right\Vert \left(  \vee\left\Vert
S\left(  V_{+}\right)  \right\Vert \right)  =\left\Vert x\right\Vert \left(
\vee\left\langle e,S\left(  V_{+}\right)  \right\rangle \right)  =\left\Vert
x\right\Vert
\]
and $\left\Vert e\right\Vert _{e}=\vee\left\vert \left\langle e,S\left(
V_{+}\right)  \right\rangle \right\vert =1$. Moreover, $\left\Vert x^{\ast
}\right\Vert _{e}=\vee\left\vert \left\langle x^{\ast},S\left(  V_{+}\right)
\right\rangle \right\vert =\vee\left\vert \left\langle x,S\left(
V_{+}\right)  \right\rangle ^{\ast}\right\vert =\left\Vert x\right\Vert _{e}$
and
\[
\left\Vert \operatorname{Re}x\right\Vert _{e}=\vee\left\vert \left\langle
\operatorname{Re}x,S\left(  V_{+}\right)  \right\rangle \right\vert
=\vee\left\vert \operatorname{Re}\left\langle x,S\left(  V_{+}\right)
\right\rangle \right\vert \leq\left\Vert x\right\Vert _{e}.
\]
Similarly, $\left\Vert \operatorname{Im}x\right\Vert _{e}\leq\left\Vert
x\right\Vert _{e}$. If $x\in V_{h}$ then (see \cite[3.2.25]{Ped})%
\[
\left\Vert x\right\Vert =\vee\left\{  \left\vert \left(  x\zeta,\zeta\right)
\right\vert :\left\Vert \zeta\right\Vert =1\right\}  =\vee\left\{  \left\vert
\left\langle x,\omega_{\zeta}|V\right\rangle \right\vert :\left\Vert
\zeta\right\Vert =1\right\}  \leq\vee\left\vert \left\langle x,S\left(
V_{+}\right)  \right\rangle \right\vert =\left\Vert x\right\Vert _{e},
\]
that is, $\left\Vert x\right\Vert _{e}=\left\Vert x\right\Vert $ for all $x\in
V_{h}$. In particular, for each $x\in V$ we have $\left\Vert x\right\Vert
\leq\left\Vert \operatorname{Re}x\right\Vert +\left\Vert \operatorname{Im}%
x\right\Vert \leq2\left\Vert x\right\Vert _{e}\leq2\left\Vert x\right\Vert $.
Thus $\left\Vert \cdot\right\Vert $ and $\left\Vert \cdot\right\Vert _{e}$ are
equivalent norms on $V$. If $\alpha=\wedge\left\{  r>0:-re\leq x\leq
re\right\}  $ for $x\in V_{h}$, then
\begin{align*}
\alpha &  =\wedge\left\{  r>0:re\pm x\geq0\right\}  =\wedge\left\{
r>0:\left\langle re\pm x,\varphi\right\rangle \geq0,\varphi\in S\left(
V_{+}\right)  \right\} \\
&  =\wedge\left\{  r>0:r\pm\left\langle x,\varphi\right\rangle \geq
0,\varphi\in S\left(  V_{+}\right)  \right\}  =\wedge\left\{  r>0:\vee
\left\vert \left\langle x,S\left(  V_{+}\right)  \right\rangle \right\vert
\leq r\right\} \\
&  =\vee\left\vert \left\langle x,S\left(  V_{+}\right)  \right\rangle
\right\vert =\left\Vert x\right\Vert _{e}\text{,}%
\end{align*}
that is, $\left\Vert \cdot\right\Vert _{e}$ is an order norm on $V$. Finally,
based on (\ref{nHf}), for every $\mu\in V_{h}^{\ast}$ we have
\[
\left\Vert \mu\right\Vert _{e}=\vee\left\vert \left\langle \operatorname{ball}%
\left\Vert \cdot\right\Vert _{e},\mu\right\rangle \right\vert =\vee\left\{
\left\vert \left\langle x,\mu\right\rangle \right\vert :x\in V_{h},\left\Vert
x\right\Vert _{e}\leq1\right\}  =\vee\left\vert \left\langle
\operatorname{ball}V_{h},\mu\right\rangle \right\vert =\left\Vert
\mu\right\Vert ,
\]
that is, $\left\Vert \mu\right\Vert _{e}=\left\Vert \mu\right\Vert $.
\end{proof}

There is a well defined unital $\ast$-linear mapping $\Phi:V\rightarrow
C\left(  S\left(  V_{+}\right)  \right)  $, $\Phi\left(  x\right)  \left(
\varphi\right)  =\left\langle x,\varphi\right\rangle $, $x\in V$, $\varphi\in
S\left(  V_{+}\right)  $ (Kadison's representation). Notice that $\Phi\left(
V_{+}\right)  \subseteq C\left(  S\left(  V_{+}\right)  \right)  _{+}$.

\begin{corollary}
\label{corOKE31}The $\ast$-representation $\Phi:\left(  V,\left\Vert
\cdot\right\Vert _{e}\right)  \rightarrow C\left(  S\left(  V_{+}\right)
\right)  $ is an order isomorphism and isometric mapping onto its range.
\end{corollary}

\begin{proof}
Based on Proposition \ref{propAp1}, we have
\[
\left\Vert \Phi\left(  x\right)  \right\Vert _{\infty}=\vee\left\vert
\Phi\left(  x\right)  \left(  S\left(  V_{+}\right)  \right)  \right\vert
=\vee\left\vert \left\langle x,S\left(  V_{+}\right)  \right\rangle
\right\vert =\left\Vert x\right\Vert _{e}%
\]
for all $x\in V$, that is, $\Phi$ is an isometry. Furthermore%
\[
\Phi\left(  x\right)  \in C\left(  S\left(  V_{+}\right)  \right)
_{+}\Leftrightarrow\left\langle x,S\left(  V_{+}\right)  \right\rangle
\geq0\Leftrightarrow x\in V_{+}%
\]
by (\ref{sp}). Thus $\Phi\left(  V_{+}\right)  =C\left(  S\left(
V_{+}\right)  \right)  _{+}\cap\Phi\left(  V_{h}\right)  $, which means $\Phi$
is an order isomorphism onto its range.
\end{proof}

Thus if the original norm of $V$ equals to $\left\Vert \cdot\right\Vert _{e}$
then $V$ turns out to be a unital $\ast$-subspace of $C\left(  S\left(
V_{+}\right)  \right)  $, that is, $V$ is \textit{a concrete function system
}up to an isometry.

\begin{corollary}
\label{corOKE44}Let $V$ be an operator system. Then $V$ is a function system
iff $\left\Vert \cdot\right\Vert =\left\Vert \cdot\right\Vert _{e}$.
\end{corollary}

\begin{proof}
Suppose that $V$ is a function system, that is, $V\subseteq C\left(
\mathcal{T}\right)  $ is a unital self-adjoint subspace for some Hausdorff
compact topological space $\mathcal{T}$. If $\varphi\in S\left(  V_{+}\right)
$ then $\varphi\in\operatorname{ball}V^{\ast}$ and it admits a bounded linear
extension up to some $\mu\in M\left(  \mathcal{T}\right)  $ with $\left\Vert
\mu\right\Vert =1=\left\langle e,\varphi\right\rangle =\left\langle
e,\mu\right\rangle $ thanks to Hahn-Banach extension Theorem. By Lemma
\ref{lemOS0}, we obtain that $\mu\in P\left(  \mathcal{T}\right)  $. Since
$P\left(  \mathcal{T}\right)  |V\subseteq S\left(  V_{+}\right)  $, it follows
that (see Proposition \ref{propBPM})%
\begin{align*}
\left\Vert f\right\Vert _{e}  &  =\vee\left\vert \left\langle f,S\left(
V_{+}\right)  \right\rangle \right\vert \leq\vee\left\vert \left\langle
f,P\left(  \mathcal{T}\right)  |V\right\rangle \right\vert \leq\vee\left\vert
\left\langle f,\operatorname{ball}M\left(  \mathcal{T}\right)  \right\rangle
\right\vert =\left\Vert f\right\Vert _{\infty}\\
&  =\vee\left\vert f\left(  \mathcal{T}\right)  \right\vert =\vee\left\{
\left\vert \left\langle f,\delta_{t}|V\right\rangle \right\vert :t\in
\mathcal{T}\right\}  \leq\vee\left\vert \left\langle f,S\left(  V_{+}\right)
\right\rangle \right\vert =\left\Vert f\right\Vert _{e},
\end{align*}
that is, $\left\Vert f\right\Vert _{e}=\left\Vert f\right\Vert _{\infty}$ for
all $f\in V$. The reverse implication follows from Corollary \ref{corOKE31}.
\end{proof}

\begin{corollary}
\label{corabcSV}If $V$ is an operator system then $\operatorname{ball}%
V_{h}^{\ast}\subseteq\operatorname{abc}S\left(  V_{+}\right)  $. In
particular, $2^{-1}\operatorname{ball}V^{\ast}\subseteq\operatorname{abc}%
S\left(  V_{+}\right)  \subseteq\operatorname{ball}V^{\ast}$.
\end{corollary}

\begin{proof}
The inclusion $\operatorname{abc}S\left(  V_{+}\right)  \subseteq
\operatorname{ball}V^{\ast}$ is immediate. Take $\varphi\in\operatorname{ball}%
V_{h}^{\ast}$. By Theorem \ref{thOS1}, $\varphi$ admits an orthogonal
expansion $\varphi=\varphi_{+}-\varphi_{-}$. In the case of a nontrivial
expansion, we put $\phi_{+}=r_{+}^{-1}\varphi_{+}$, $\phi_{-}=r_{-}%
^{-1}\varphi_{-}$ with $r_{+}=\left\langle e,\varphi_{+}\right\rangle $,
$r_{-}=\left\langle e,\varphi_{-}\right\rangle $. Then $\phi_{+}$, $\phi
_{-}\in S\left(  V_{+}\right)  $ thanks to Lemma \ref{lemOS0}, and
$\varphi=r_{+}\phi_{+}-r_{-}\phi_{-}$. Since $r_{+}+r_{-}=\left\Vert
\varphi\right\Vert \leq1$, it follows that $\varphi=r_{+}\phi_{+}-r_{-}%
\phi_{-}\in\operatorname{abc}S\left(  V_{+}\right)  $. Hence
$\operatorname{ball}V_{h}^{\ast}\subseteq\operatorname{abc}S\left(
V_{+}\right)  $. Finally, every $\varphi\in\operatorname{ball}V^{\ast}$ admits
the expansion $\varphi=\operatorname{Re}\varphi+i\operatorname{Im}\varphi$
with $\operatorname{Re}\varphi,\operatorname{Im}\varphi\in\operatorname{ball}%
V_{h}^{\ast}$. Whence $\operatorname{ball}V^{\ast}\subseteq2\operatorname{abc}%
S\left(  V_{+}\right)  $.
\end{proof}

Based on Theorem \ref{thOS1}, we also derive that every element of the unit
sphere of $V_{h}^{\ast}$ is a convex combination of $S\left(  V_{+}\right)
\cup-S\left(  V_{+}\right)  $. The combination is unique in the case of a
$C^{\ast}$-algebra $\mathcal{A}$.

\section{$\varepsilon$-positivity in operator systems\label{SecDPOS}}

In this section we analyze a new concept of the $\varepsilon$-positivity in an
operator system $V$ by deforming the original positivity of $V$.

\subsection{$\varepsilon$-postivity}

Let $V$ be an operator system with its unit $e$, and the separated, closed,
unital cone $V_{+}$ of its positive elements, whose state space is denoted by
$S\left(  V_{+}\right)  $ (equipped with the $w^{\ast}$-topology). We put
$S_{\varepsilon}=V_{he}^{\ast}\cap\varepsilon\operatorname{ball}V^{\ast}$ for
$\varepsilon\geq1$. In this case, $S_{1}=V_{he}^{\ast}\cap\operatorname{ball}%
V^{\ast}=S\left(  V_{+}\right)  $. An element $v\in V_{h}$ is said to be
$\varepsilon$\textit{-positive} if $v\geq\left\Vert v\right\Vert
\dfrac{\varepsilon-1}{\varepsilon+1}e$ in $V$. The set of all $\varepsilon
$-positive elements in $V$ is denoted by $V_{\varepsilon}^{+}$. Note that
$V_{\varepsilon}^{+}\subseteq V_{+}$, $V_{1}^{+}=V_{+}$, and $V_{\varepsilon
}^{+}$ is a separated, (norm) closed cone.

\begin{lemma}
\label{lemEp1}The cone $V_{\varepsilon}^{+}$ is unital and%
\[
A_{\varepsilon}=S\left(  V_{\varepsilon}^{+}\right)  =S_{\varepsilon},
\]
where $A_{\varepsilon}=\left\{  \left(  1+s\right)  \phi-s\psi:\phi,\psi\in
S\left(  V_{+}\right)  ,0\leq s\leq2^{-1}\left(  \varepsilon-1\right)
\right\}  $.
\end{lemma}

\begin{proof}
Since the assertion is trivial for $\varepsilon=1$, we assume that
$\varepsilon>1$. Take $v\in V_{h}$ and put $r=\left\Vert v\right\Vert
\varepsilon$. Then $v+re\in V_{\varepsilon}^{+}$. Indeed,
\[
\left\Vert v+re\right\Vert \dfrac{\varepsilon-1}{\varepsilon+1}e\leq\left(
\left\Vert v\right\Vert +r\right)  \dfrac{\varepsilon-1}{\varepsilon
+1}e=\left(  \left\Vert v\right\Vert \varepsilon-\left\Vert v\right\Vert
\right)  e=re-\left\Vert v\right\Vert e\leq re+v,
\]
which means that $v+re\in V_{\varepsilon}^{+}$, that is, $V_{\varepsilon}^{+}$
is a unital cone.

Now prove that $S\left(  V_{+}\right)  \subseteq S\left(  V_{\varepsilon}%
^{+}\right)  \subseteq S_{\varepsilon}$. Since $V_{\varepsilon}^{+}\subseteq
V_{+}$, it is immediate that $S\left(  V_{+}\right)  \subseteq S\left(
V_{\varepsilon}^{+}\right)  $. If $\varphi\in S\left(  V_{\varepsilon}%
^{+}\right)  $ and $v\in V_{h}$ then $\pm v+\left\Vert v\right\Vert
\varepsilon e\in V_{\varepsilon}^{+}$ and $\left\langle \pm v+\left\Vert
v\right\Vert \varepsilon e,\varphi\right\rangle \geq0$, which means that
$\left\vert \left\langle v,\varphi\right\rangle \right\vert \leq
\varepsilon\left\Vert v\right\Vert $. By (\ref{nHf}), we conclude that
$\left\Vert \varphi\right\Vert =\vee\left\vert \left\langle
\operatorname{ball}V_{h},\varphi\right\rangle \right\vert \leq\varepsilon$.
Thus $S\left(  V_{\varepsilon}^{+}\right)  \subseteq V_{he}^{\ast}%
\cap\varepsilon\operatorname{ball}V^{\ast}=S_{\varepsilon}$.

Further, take $\varphi\in A_{\varepsilon}$, that is, $\varphi=\left(
1+s\right)  \phi-s\psi$ with $\phi,\psi\in S\left(  V_{+}\right)  $ and $0\leq
s\leq2^{-1}\left(  \varepsilon-1\right)  $. Prove that $\varphi\in S\left(
V_{\varepsilon}^{+}\right)  $. As we have just noticed above the statement is
true if $s=0$. Suppose $s>0$. Then $\varphi\in V_{h}^{\ast}$ and $\left\langle
e,\varphi\right\rangle =\left(  1+s\right)  \left\langle e,\phi\right\rangle
-s\left\langle e,\psi\right\rangle =1$, that is, $\varphi\in V_{he}^{\ast}$.
Take a nonzero $v\in V_{\varepsilon}^{+}$. Then $\left\Vert v\right\Vert e\geq
v\geq\left\Vert v\right\Vert \dfrac{\varepsilon-1}{\varepsilon+1}e$ and
\[
\left\langle v,\varphi\right\rangle =\left(  1+s\right)  \left\langle
v,\phi\right\rangle -s\left\langle v,\psi\right\rangle \geq\left(  1+s\right)
\dfrac{\varepsilon-1}{\varepsilon+1}\left\Vert v\right\Vert -s\left\Vert
v\right\Vert =s\left\Vert v\right\Vert \left(  \dfrac{1+s}{s}\dfrac
{\varepsilon-1}{\varepsilon+1}-1\right)  .
\]
The real-valued function $f\left(  s\right)  =\dfrac{1+s}{s}$, $0<s\leq
2^{-1}\left(  \varepsilon-1\right)  $ is decreasing with $\wedge f=\min
f=f\left(  2^{-1}\left(  \varepsilon-1\right)  \right)  =\dfrac{\varepsilon
+1}{\varepsilon-1}$, therefore $\left\langle v,\varphi\right\rangle \geq0$ or
$\varphi\in S\left(  V_{\varepsilon}^{+}\right)  $. Thus $S\left(
V_{+}\right)  \subseteq A_{\varepsilon}\subseteq S\left(  V_{\varepsilon}%
^{+}\right)  \subseteq S_{\varepsilon}$. It remains to prove that
$S_{\varepsilon}\backslash S\left(  V_{+}\right)  \subseteq A_{\varepsilon}$.

Take $\varphi\in S_{\varepsilon}\backslash S\left(  V_{+}\right)  $. Being a
hermitian functional on $V$, it admits an orthogonal decomposition by Theorem
\ref{thOS1}. Namely, $\varphi=\varphi_{+}-\varphi_{-}$ with positive
functionals $\varphi_{+}$ and $\varphi_{-}$ such that $\left\Vert
\varphi\right\Vert =\left\Vert \varphi_{+}\right\Vert +\left\Vert \varphi
_{-}\right\Vert $. Put $t=\left\Vert \varphi_{+}\right\Vert =\left\langle
e,\varphi_{+}\right\rangle $ and $s=\left\Vert \varphi_{-}\right\Vert
=\left\langle e,\varphi_{-}\right\rangle $. In this case, $t+s\leq\varepsilon$
and $1=\left\langle e,\varphi\right\rangle =t-s$. Then $t>0$. But $s>0$ too,
otherwise $\varphi=\varphi_{+}\in S\left(  V_{+}\right)  $. Thus both
$\varphi_{+}$ and $\varphi_{-}$ are nonzero functionals. Put $\phi
=t^{-1}\varphi_{+}$ and $\psi=s^{-1}\varphi_{-}$. Then $\varphi=\left(
1+s\right)  \phi-s\psi$ with $\phi,\psi\in S\left(  V_{+}\right)  $ and $0\leq
s\leq2^{-1}\left(  \varepsilon-1\right)  $. It means that $\varphi\in
A_{\varepsilon}$.
\end{proof}

The $\varepsilon$-positivity generates the $\varepsilon$-norm $\left\Vert
v\right\Vert _{\varepsilon}=\vee\left\vert \left\langle v,S\left(
V_{\varepsilon}^{+}\right)  \right\rangle \right\vert $, $v\in V$ on $V$. For
$\varepsilon=1$ we have $\left\Vert \cdot\right\Vert _{\varepsilon}=\left\Vert
\cdot\right\Vert _{e}$ (see Subsection \ref{subsecSOS11}), which is equivalent
to the original norm of $V$ by Proposition \ref{propAp1}.

\begin{corollary}
If $V$ is an operator system then its $\varepsilon$-norm is equivalent to the
original norm for every $\varepsilon\geq1$.
\end{corollary}

\begin{proof}
If $\varphi\in S\left(  V_{\varepsilon}^{+}\right)  $ then $\varphi\in
A_{\varepsilon}$ by Lemma \ref{lemEp1}, which means that $\varphi=\left(
1+s\right)  \phi-s\psi$ for some $\phi,\psi\in S_{1}$ and $0\leq s\leq
2^{-1}\left(  \varepsilon-1\right)  $. Since $\left\vert 1+s\right\vert
+\left\vert -s\right\vert =1+2s\leq\varepsilon$, we conclude that $\varphi
\in\varepsilon\operatorname{abc}\left(  S_{1}\right)  $. Thus
\[
S_{1}=V_{he}^{\ast}\cap\operatorname{ball}V^{\ast}\subseteq V_{he}^{\ast}%
\cap\varepsilon\operatorname{ball}V^{\ast}=S\left(  V_{\varepsilon}%
^{+}\right)  =A_{\varepsilon}\subseteq\varepsilon\operatorname{abc}\left(
S_{1}\right)  ,
\]
which in turn implies the inclusions $\varepsilon^{-1}S_{1}^{\circ
}=\varepsilon^{-1}\operatorname{abc}\left(  S_{1}\right)  ^{\circ}\subseteq
S\left(  V_{\varepsilon}^{+}\right)  ^{\circ}\subseteq S_{1}^{\circ}$ for the
polars in $V$. Then $\left\Vert v\right\Vert _{\varepsilon}\leq\left\Vert
v\right\Vert _{e}\leq\varepsilon\left\Vert v\right\Vert _{\varepsilon}$ for
all $v\in V$, or $\left\Vert \cdot\right\Vert _{\varepsilon}\sim\left\Vert
\cdot\right\Vert _{e}$. It remains to use Proposition \ref{propAp1}.
\end{proof}

\subsection{The continuous functions with $\varepsilon$-oscillations}

Now assume for a while that $V=C\left(  \mathcal{T}\right)  $ is an abelian,
unital $C^{\ast}$-algebra, where $\mathcal{T}$ is a compact, Hausdorff
topological space. Recall that\textit{ the oscillation of a function}
$h:\mathcal{T}\rightarrow\mathbb{C}$ over a subset $E\subseteq\mathcal{T}$ is
given by the quantity $\omega\left(  h,E\right)  =\vee\left\{  \left\vert
h\left(  t\right)  -h\left(  l\right)  \right\vert :t,l\in E\right\}  $ by
measuring the highest amplitude of $h$ over $E$. Take $f\in C\left(
\mathcal{T}\right)  _{h}$. By its very definition, $f\in C\left(
\mathcal{T}\right)  _{\varepsilon}^{+}$ iff $f\left(  t\right)  \geq\left\Vert
f\right\Vert _{\infty}\dfrac{\varepsilon-1}{\varepsilon+1}$ for all
$t\in\mathcal{T}$. A function $f\in C\left(  \mathcal{T}\right)  _{+}$ is said
to has the $\varepsilon$\textit{-oscillation }if either $f=0$ or
$\omega\left(  \ln f,\mathcal{T}\right)  \leq c_{\varepsilon}$, where
$c_{\varepsilon}=\ln\dfrac{\varepsilon+1}{\varepsilon-1}$. Every function from
$C\left(  \mathcal{T}\right)  _{+}$ has the $\varepsilon$-oscillation for
$\varepsilon=1$ (we put $c_{1}=+\infty$). More concrete example can be given
by the function $f\left(  t\right)  =3^{t}$ on $\mathcal{T=}\left[
0,1\right]  $ for $\varepsilon=2$, whereas $3^{3^{t}}$ has no $2$-oscillation.

Consider the unital $\ast$-representation $\Phi:V\rightarrow C\left(  S\left(
V_{+}\right)  \right)  $, $\Phi\left(  v\right)  \left(  \varphi\right)
=\left\langle v,\varphi\right\rangle $, $v\in V$, $\varphi\in S\left(
V_{+}\right)  $ with $\Phi\left(  V_{+}\right)  =C\left(  S\left(
V_{+}\right)  \right)  _{+}\cap\Phi\left(  V_{h}\right)  $ (see Corollary
\ref{corOKE31}).

\begin{lemma}
\label{lemEp2}The equality $\Phi\left(  V_{\varepsilon}^{+}\right)  =C\left(
\mathcal{T}\right)  _{\varepsilon}^{+}\cap\Phi\left(  V_{h}\right)  $ holds.
Moreover, for every nonzero $v\in V_{h}$ the following assertions are equivalent:

$\left(  i\right)  $ $v\in V_{\varepsilon}^{+}$;

$\left(  ii\right)  $ $\Phi\left(  v\right)  $ has the $\varepsilon$-oscillation;

$\left(  iii\right)  $ $\left\langle v,S\left(  V_{\varepsilon}^{+}\right)
\right\rangle \geq0$.

Thus the bipolar theorem (see \cite[Remark 5.3]{DSbM}) holds for all unital
cones $V_{\varepsilon}^{+}$, $\varepsilon\geq1$.
\end{lemma}

\begin{proof}
Take $v\in V_{h}$. Then $v\in V_{\varepsilon}^{+}$ iff $\left\langle
v-\left\Vert v\right\Vert \dfrac{\varepsilon-1}{\varepsilon+1}e,S\left(
V_{+}\right)  \right\rangle \geq0$ by (\ref{sp}). It means that $\Phi\left(
v\right)  \geq\left\Vert v\right\Vert \dfrac{\varepsilon-1}{\varepsilon
+1}=\left\Vert \Phi\left(  v\right)  \right\Vert _{\infty}\dfrac
{\varepsilon-1}{\varepsilon+1}$ thanks to Corollary \ref{corOKE31}. Hence
$\Phi\left(  V_{\varepsilon}^{+}\right)  =C\left(  \mathcal{T}\right)
_{\varepsilon}^{+}\cap\Phi\left(  V_{h}\right)  $.

Now take a nonzero $v\in V_{h}$. If $v\in V_{\varepsilon}^{+}$ then
$\dfrac{\varepsilon+1}{\varepsilon-1}v\geq\left\Vert v\right\Vert e\geq v$
and
\[
\left\langle v,\varphi\right\rangle \dfrac{\varepsilon+1}{\varepsilon-1}%
\geq\left\Vert v\right\Vert \geq\left\langle v,\psi\right\rangle
\]
for all $\varphi,\psi\in S\left(  V_{+}\right)  $. In particular,
$\left\langle v,\phi\right\rangle >0$ for all $\phi\in S\left(  V_{+}\right)
$, and $\dfrac{\left\langle v,\psi\right\rangle }{\left\langle v,\varphi
\right\rangle }\leq\dfrac{\varepsilon+1}{\varepsilon-1}$ for all $\varphi
,\psi\in S\left(  V_{+}\right)  $. Thus $\dfrac{\Phi\left(  v\right)  \left(
\psi\right)  }{\Phi\left(  v\right)  \left(  \varphi\right)  }\leq
\dfrac{\varepsilon+1}{\varepsilon-1}$ for all $\varphi,\psi\in S\left(
V_{+}\right)  $, which means that%
\[
\omega\left(  \ln\Phi\left(  v\right)  ,S\left(  V_{+}\right)  \right)
=\vee\left\{  \left\vert \ln\Phi\left(  v\right)  \left(  \psi\right)
-\ln\Phi\left(  v\right)  \left(  \varphi\right)  \right\vert :\psi,\varphi\in
S\left(  V_{+}\right)  \right\}  \leq c_{\varepsilon}%
\]
or $\Phi\left(  v\right)  $ has the $\varepsilon$-oscillation. Thus $\left(
i\right)  \Rightarrow\left(  ii\right)  $ holds.

Now suppose that $\Phi\left(  v\right)  $ possesses the $\varepsilon
$-oscillation. Then $\left\langle v,\phi\right\rangle >0$ for all $\phi\in
S\left(  V_{+}\right)  $, and $\dfrac{\left\langle v,\psi\right\rangle
}{\left\langle v,\varphi\right\rangle }\leq\dfrac{\varepsilon+1}%
{\varepsilon-1}$. As in the proof of Lemma \ref{lemEp1}, we have
$\dfrac{\left\langle v,\psi\right\rangle }{\left\langle v,\varphi\right\rangle
}\leq\dfrac{s+1}{s}$ for all $0<s\leq2^{-1}\left(  \varepsilon-1\right)  $. It
follows that $\left\langle v,\left(  1+s\right)  \phi-s\psi\right\rangle
\geq0$ for all $0\leq s\leq2^{-1}\left(  \varepsilon-1\right)  $ (if $s=0$
then $\left\langle v,\phi\right\rangle >0$ for all $\phi\in S\left(
V_{+}\right)  $). Thus $\left\langle v,A_{\varepsilon}\right\rangle \geq0$ and
$\left(  ii\right)  \Rightarrow\left(  iii\right)  $ follows from Lemma
\ref{lemEp1}.

Finally, prove that $\left(  iii\right)  \Rightarrow\left(  i\right)  $. If
$\left\langle v,A_{\varepsilon}\right\rangle \geq0$ then $\left(  1+s\right)
\left\langle v,\phi\right\rangle \geq s\left\langle v,\psi\right\rangle $ for
all $\phi,\psi\in S\left(  V_{+}\right)  $ and $0\leq s\leq2^{-1}\left(
\varepsilon-1\right)  $. But $\left\Vert v\right\Vert =\left\Vert v\right\Vert
_{e}>0$, therefore $\left\langle v,\psi\right\rangle >0$ for some $\psi\in
S\left(  V_{+}\right)  $ (see Corollary \ref{corOKE44}). It follows that
$\left\langle v,\phi\right\rangle >0$ for all $\phi\in S\left(  V_{+}\right)
$. As above $\dfrac{\left\langle v,\psi\right\rangle }{\left\langle
v,\varphi\right\rangle }\leq\dfrac{\varepsilon+1}{\varepsilon-1}$ holds for
all $\phi,\psi\in S\left(  V_{+}\right)  $. Then%
\[
\dfrac{\varepsilon-1}{\varepsilon+1}\left\Vert v\right\Vert =\dfrac
{\varepsilon-1}{\varepsilon+1}\vee\left\langle v,S\left(  V_{+}\right)
\right\rangle =\vee\left\langle v,\dfrac{\varepsilon-1}{\varepsilon+1}S\left(
V_{+}\right)  \right\rangle \leq\left\langle v,\varphi\right\rangle
\]
for all $\varphi\in S\left(  V_{+}\right)  $, that is, $\left\langle
v-\left\Vert v\right\Vert \dfrac{\varepsilon-1}{\varepsilon+1}e,S\left(
V_{+}\right)  \right\rangle \geq0$. By (\ref{sp}), we deduce $v\in
V_{\varepsilon}^{+}$.
\end{proof}

In the case of $V=C\left(  \mathcal{T}\right)  $ we have $\mathcal{T=\partial
}P\left(  \mathcal{T}\right)  \subseteq P\left(  \mathcal{T}\right)  =S\left(
C\left(  \mathcal{T}\right)  _{+}\right)  $ (see Subsection \ref{subsecEP})
and the related $\ast$-representation $\Phi:C\left(  \mathcal{T}\right)
\rightarrow C\left(  P\left(  \mathcal{T}\right)  \right)  $ is reduced to the
natural extension $\Phi\left(  f\right)  $ of a function $f\in C\left(
\mathcal{T}\right)  $ to $P\left(  \mathcal{T}\right)  $ by $\Phi\left(
f\right)  \left(  \mu\right)  =\left\langle f,\mu\right\rangle $, $\mu\in
P\left(  \mathcal{T}\right)  $. If $\mu=\delta_{t}$ for some $t\in\mathcal{T}%
$, then $\Phi\left(  f\right)  \left(  \delta_{t}\right)  =f\left(  t\right)
$.

\begin{corollary}
The unital cone $C\left(  \mathcal{T}\right)  _{\varepsilon}^{+}$ consists of
all functions with the $\varepsilon$-oscillations precisely.
\end{corollary}

\begin{proof}
Take a nonzero $f\in C\left(  \mathcal{T}\right)  _{h}$. First assume that
$f\in C\left(  \mathcal{T}\right)  _{\varepsilon}^{+}$. By Lemma \ref{lemEp2},
$\Phi\left(  f\right)  $ has the $\varepsilon$-oscillation on $P\left(
\mathcal{T}\right)  $. In particular, it has the $\varepsilon$-oscillation on
the subset $\mathcal{T}$ either. It remains to notify that $f=\Phi\left(
f\right)  |\mathcal{T}$.

Conversely, suppose that $f$ has the $\varepsilon$-oscillation. Then
$\dfrac{f\left(  t\right)  }{f\left(  l\right)  }\leq\dfrac{\varepsilon
+1}{\varepsilon-1}$, $t,l\in\mathcal{T}$, which means that $\left(
1+s\right)  \left\langle f,\delta_{t}\right\rangle \geq s\left\langle
f,\delta_{l}\right\rangle $ for all $t,l\in\mathcal{T}$ and $0\leq s\leq
2^{-1}\left(  \varepsilon-1\right)  $. By Proposition \ref{propBPM}, we have
$\operatorname{co}\left(  \mathcal{T}\right)  ^{-w^{\ast}}=\operatorname{co}%
\left(  \mathcal{\partial}P\left(  \mathcal{T}\right)  \right)  ^{-w^{\ast}%
}=P\left(  \mathcal{T}\right)  $, therefore $\left(  1+s\right)  \left\langle
f,\mu\right\rangle \geq s\left\langle f,\nu\right\rangle $ for all $\mu,\nu
\in\mathcal{P}\left(  \mathcal{T}\right)  $ and $0\leq s\leq2^{-1}\left(
\varepsilon-1\right)  $. It follows that $\left\langle f,S\left(  C\left(
\mathcal{T}\right)  _{\varepsilon}^{+}\right)  \right\rangle \geq0$ by Lemma
\ref{lemEp1}. Using Lemma \ref{lemEp2}, we obtain that $f\in C\left(
\mathcal{T}\right)  _{\varepsilon}^{+}$ (the Bipolar Theorem).
\end{proof}

\section{\textbf{Abstract function systems}\label{secPr}}

In this section we prove the main result on the projective function systems.

\subsection{The unital $\ast$-normed spaces\label{subsecUNS1}}

Let $V$ be a $\ast$-vector space, that is, $V$ is a complex vector space with
an involution. A norm $\left\Vert \cdot\right\Vert $ on $V$ is said to be a
$\ast$\textit{-norm} if the unit $\operatorname{ball}V$ is hermitian, that is,
$\left(  \operatorname{ball}V\right)  ^{\ast}=\operatorname{ball}V$. The
latter is equivalent to the property $\left\Vert v^{\ast}\right\Vert
=\left\Vert v\right\Vert $ for all $v\in V$. The real vector space of all
hermitian elements in $V$ is denoted by $V_{h}$. Notice that the norm dual
space $V^{\ast}$ also possesses the involution $\varphi\mapsto\varphi^{\ast}$
given by $\left\langle v,\varphi^{\ast}\right\rangle =\left\langle v^{\ast
},\varphi\right\rangle ^{\ast}$. If $V$ is a $\ast$-normed space then
$\left\Vert \varphi^{\ast}\right\Vert =\left\Vert \varphi\right\Vert $ for all
$\varphi\in V^{\ast}$, and the involution on $V$ is a $\sigma\left(
V,V^{\ast}\right)  $-continuous mapping. Thus $\left(  V,V^{\ast}\right)  $ is
a dual $\ast$-pair (see \cite{Dhjm}, \cite{DCcones}). The real vector space of
all bounded hermitian functionals on $V$ is denoted by $V_{h}^{\ast}$.

Now fix a nonzero $e\in V_{h}$ of a $\ast$-normed vector space $V$. Put
$V_{he}^{\ast}=\left\{  y\in V_{h}^{\ast}:\left\langle e,y\right\rangle
=1\right\}  $ to be the real hyperplane of all unital hermitian functionals
from $V^{\ast}$. As above put $S=V_{he}^{\ast}\cap\operatorname{ball}V^{\ast}$
(or $V_{he}^{\ast}\cap\left(  \operatorname{ball}V\right)  ^{\circ}$) to be a
convex set. We say that $e$ is \textit{a unit element }(or just \textit{unit}%
)\textit{ for }$V$\textit{ }if $S\neq\varnothing$. Thus $1=\left\langle
e,y\right\rangle \leq\left\Vert e\right\Vert \left\Vert y\right\Vert
\leq\left\Vert e\right\Vert $ for some $y\in S$, that is, $\left\Vert
e\right\Vert \geq1$. Every hermitian unit vector $e$ of a $\ast$-normed space
$V$ is a unit automatically. Indeed, since $\left\Vert e\right\Vert =1$, it
follows that $\left\langle e,y\right\rangle =1$ for some $y\in
\operatorname{ball}V^{\ast}$ by Hahn-Banach extension property. But
$\operatorname{Re}y\in\operatorname{ball}V^{\ast}$ and $\left\langle
e,\operatorname{Re}y\right\rangle =2^{-1}\left(  \left\langle e,y\right\rangle
+\left\langle e,y\right\rangle ^{\ast}\right)  =1$, that is,
$\operatorname{Re}y\in V_{he}^{\ast}\cap\operatorname{ball}V^{\ast}=S$. Hence
$e$ is a unit for $V$ and $S\subseteq\partial\operatorname{ball}V^{\ast}$. A
$\ast$-normed vector space $V$ with a fixed unit $\left(  V,e\right)  $ is
called \textit{a unital }$\ast$\textit{-normed space.}

\subsection{The unital ball of a unital $\ast$-normed space}

The unit $\operatorname{ball}V$ of a unital $\ast$-normed space $\left(
V,e\right)  $ is said to be \textit{a unital ball} if $\operatorname{ball}V$
is the absolute polar of its state space $S$, that is, $\operatorname{ball}%
V=S^{\circ}$. A unital $\ast$-normed space $\left(  V,e\right)  $ with its
unital ball is called\textit{ an abstract function system }or \textit{a
commutative operator system with its state space }$S$\textit{. }For brevity we
say that $V$ is \textit{a function system}. A unital self-adjoint subspace $V$
of $C\left(  \mathcal{T}\right)  $ is an example of a (concrete) function
system (see Corollary \ref{corOKE44}). In the case of $V=C\left(
\mathcal{T}\right)  $ with $e=1$ we obtain that $S=P\left(  \mathcal{T}%
\right)  $ (see Proposition \ref{propBPM}).

\begin{lemma}
\label{lemUNS0}Let $\left(  V,e\right)  $ be a function system, $\mathfrak{c=}%
\left\{  v\in V_{h}:\left\langle v,S\right\rangle \geq0\right\}  $ and let
$\left\Vert v\right\Vert _{e}=\vee\left\vert \left\langle v,S\right\rangle
\right\vert $, $v\in V$. Then $\mathfrak{c}$ is a separated, closed, unital
cone in $V$ such that $S\left(  \mathfrak{c}\right)  =S$ and $\left\Vert
\cdot\right\Vert =\left\Vert \cdot\right\Vert _{e}$. In particular, $e$ is a
unit vector.
\end{lemma}

\begin{proof}
Using the Absolute Bipolar Theorem, we derive that $\operatorname{ball}%
V^{\ast}=\left(  \operatorname{ball}V\right)  ^{\circ}=S^{\circ\circ
}=\operatorname{abc}\left(  S\right)  ^{-w^{\ast}}$, where $\operatorname{abc}%
\left(  S\right)  ^{-w^{\ast}}$ is the $w^{\ast}$-closure of the absolutely
convex hull $\operatorname{abc}\left(  S\right)  $ of $S$. Therefore%
\[
\left\Vert v\right\Vert =\vee\left\vert \left\langle v,\operatorname{ball}%
V^{\ast}\right\rangle \right\vert =\vee\left\vert \left\langle
v,\operatorname{abc}\left(  S\right)  ^{-w^{\ast}}\right\rangle \right\vert
=\vee\left\vert \left\langle v,S\right\rangle \right\vert =\left\Vert
v\right\Vert _{e}%
\]
for all $v\in V$. In particular, $\left\Vert e\right\Vert =\left\Vert
e\right\Vert _{e}=\sup\left\vert \left\langle e,S\right\rangle \right\vert
=1$, which means that $e$ is a unit vector of $V$. Further, take $v\in V_{h}$.
If $\pm v\in\mathfrak{c}$ then $\left\langle v,S\right\rangle =\left\{
0\right\}  $ or $\left\Vert v\right\Vert _{e}=0$, that is, $v=0$. Further,
$\left\langle \left\Vert v\right\Vert e\pm v,S\right\rangle =\left\Vert
v\right\Vert \pm\left\langle v,S\right\rangle =\left\Vert v\right\Vert _{e}%
\pm\left\langle v,S\right\rangle \geq0$, which means that $\left\Vert
v\right\Vert e\pm v\in\mathfrak{c}$. Thus $\mathfrak{c}$ is a separated,
closed, unital cone in $V$. Actually, $\mathfrak{c}$ is weakly closed.

Further, take $y\in S\left(  \mathfrak{c}\right)  $, which means that
$y:V\rightarrow\mathbb{C}$ is a linear functional such that $\left\langle
\mathfrak{c},y\right\rangle \geq0$ and $\left\langle e,y\right\rangle =1$.
Since $\left\Vert v\right\Vert e\pm v\in\mathfrak{c}$ for $v\in V_{h}$, we
obtain that $\left\vert \left\langle v,y\right\rangle \right\vert
\leq\left\Vert v\right\Vert $. Taking into account that $V=V_{h}+iV_{h}$ and
$V_{h}=\mathfrak{c-c}$, we conclude that $y\in V_{he}^{\ast}$ (see Subsection
\ref{subsecOS22}). For $v\in V$ with nonzero $\left\langle v,y\right\rangle $
we have $\left\vert \left\langle v,y\right\rangle \right\vert =\left\langle
\theta v,y\right\rangle =\operatorname{Re}\left\langle \theta v,y\right\rangle
=\left\langle \operatorname{Re}\left(  \theta v\right)  ,y\right\rangle
\leq\left\Vert \operatorname{Re}\left(  \theta v\right)  \right\Vert
\leq\left\Vert \theta v\right\Vert =\left\Vert v\right\Vert $ for some
$\theta\in\mathbb{C}$, $\left\vert \theta\right\vert =1$. Hence $y\in
V_{he}^{\ast}\cap\operatorname{ball}V^{\ast}=S$, that is, $S\left(
\mathfrak{c}\right)  =S$.
\end{proof}

Thus every function system turns out to be an Archimedian ordered $\ast
$-vector space in the sense of \cite{PT}. Using the $\ast$-representation from
Corollary \ref{corOKE31}, we can identify $V$ with a unital self-adjoint
subspace of $C\left(  S\left(  \mathfrak{c}\right)  \right)  $ up to an order
and isometric isomorphism, that is, $\left(  V,\mathfrak{c}\right)  $ is a
concrete function system on $S\left(  \mathfrak{c}\right)  $.

\subsection{Non-unital balls}

Let $\left(  V,e\right)  $ be a unital $\ast$-normed space. Put
$S_{\varepsilon}=V_{he}^{\ast}\cap\varepsilon\operatorname{ball}V^{\ast}$ for
$\varepsilon>0$. Thus $\operatorname{ball}V$ is unital whenever
$\operatorname{ball}V=S_{1}^{\circ}$, and in this case $\operatorname{ball}%
V^{\ast}=\operatorname{abc}\left(  S_{1}\right)  ^{-w^{\ast}}$. In the case of
a non-unital $\operatorname{ball}V$ we can pursue the property to be
$\operatorname{ball}V\sim S_{\varepsilon}^{\circ}$ for some $\varepsilon>0$.
Recall that the equivalence $\operatorname{ball}V\sim S_{\varepsilon}^{\circ}$
of balls means that $rS_{\varepsilon}^{\circ}\subseteq\operatorname{ball}%
V\subseteq RS_{\varepsilon}^{\circ}$ for some positive real $r$ and $R$. If
$\operatorname{ball}V\sim S_{1}^{\circ}$, then we also refer it to the unital
case. In this case, the norms $\left\Vert \cdot\right\Vert $ and $\left\Vert
\cdot\right\Vert _{e}$ are equivalent, and $e$ turns out to be a unit vector
with respect to the norm $\left\Vert \cdot\right\Vert _{e}$. By Proposition
\ref{propAp1}, every operator system (if we ignore its quantizations or
noncommutative structures) turns out to be a function system up to an
equivalent norm (see also Corollary \ref{corOKE31}). In the non-unital case,
$e$ may not be a unit vector.

\begin{lemma}
\label{lemUNS1}Let $\left(  V,e\right)  $ be a unital $\ast$-normed space with
its state space $S$. Then $\operatorname{ball}V\sim S_{2}^{\circ}$.
\end{lemma}

\begin{proof}
We use the arguments from \cite[Proposition 4.1]{DSbM} for the quantum
systems. First note that $\operatorname{Re}\left(  \operatorname{ball}V^{\ast
}\right)  \subseteq V_{h}^{\ast}\cap\operatorname{ball}V^{\ast}%
=\operatorname{ball}V_{h}^{\ast}$. Similarly, $\operatorname{Im}\left(
\operatorname{ball}V^{\ast}\right)  \subseteq\operatorname{ball}V_{h}^{\ast}$
and $\operatorname{ball}V^{\ast}\subseteq\operatorname{Re}\left(
\operatorname{ball}V^{\ast}\right)  +i\operatorname{Im}\left(
\operatorname{ball}V^{\ast}\right)  \subseteq2\operatorname{abc}\left(
\operatorname{ball}V_{h}^{\ast}\right)  $.

Now take $z\in\operatorname{ball}V_{h}^{\ast}$ and fix $y\in S=V_{he}^{\ast
}\cap\operatorname{ball}V^{\ast}$ (by assumption $S\neq\varnothing$). Note
that $\left\langle e,z\right\rangle \in\mathbb{R}$ with $-\left\Vert
e\right\Vert \leq\left\langle e,z\right\rangle \leq\left\Vert e\right\Vert $,
that is, $0\leq\left\langle e,\pm z\right\rangle \leq\left\Vert e\right\Vert
$. Put $w=\pm z+y$ with $a=\left\langle e,w\right\rangle $, that is, $1\leq
a\leq1+\left\Vert e\right\Vert $. If $u=a^{-1}w$ then $u\in V_{he}^{\ast}$ and
$\left\Vert u\right\Vert \leq a^{-1}\left\Vert w\right\Vert \leq\left\Vert \pm
z\right\Vert +\left\Vert y\right\Vert \leq2$, which means that $u\in
V_{he}^{\ast}\cap2\operatorname{ball}V^{\ast}=S_{2}$. But $z=\mp au\pm y$ and
$\left\vert \mp a\right\vert +\left\vert \pm1\right\vert \leq2+\left\Vert
e\right\Vert $, that is, $z\in\left(  2+\left\Vert e\right\Vert \right)
\operatorname{abc}\left(  S_{2}\right)  $. Thus $\operatorname{ball}V^{\ast
}\subseteq2\operatorname{abc}\left(  \operatorname{ball}V_{h}^{\ast}\right)
\subseteq2\left(  2+\left\Vert e\right\Vert \right)  \operatorname{abc}\left(
S_{2}\right)  $. Taking into account that $S_{2}\subseteq2\operatorname{ball}%
V^{\ast}$, we conclude that
\begin{equation}
2^{-1}\operatorname{abc}\left(  S_{2}\right)  \subseteq\operatorname{ball}%
V^{\ast}\subseteq2\left(  2+\left\Vert e\right\Vert \right)
\operatorname{abc}\left(  S_{2}\right)  , \label{abs1}%
\end{equation}
that is, $\operatorname{abc}\left(  S_{2}\right)  \sim\operatorname{ball}%
V^{\ast}$ or $S_{2}^{\circ}=\left(  \operatorname{abc}\left(  S_{2}\right)
\right)  ^{\circ}\sim\operatorname{ball}V$.
\end{proof}

The assertion of Lemma \ref{lemUNS1} can be generalized in the following way.

\begin{theorem}
\label{propUNS1}Let $\left(  V,e\right)  $ be a unital $\ast$-normed space
with its state space $S$. Then $\operatorname{ball}V\sim S_{\varepsilon
}^{\circ}$ for every $\varepsilon>1/\left\Vert e\right\Vert $.
\end{theorem}

\begin{proof}
For a positive real $r$ we put $V_{h\left(  re\right)  }^{\ast}=\left\{  y\in
V_{h}^{\ast}:\left\langle re,y\right\rangle =1\right\}  $ to be the real
hyperplane of all $re$-unital hermitian functionals from $V^{\ast}$. Note that
$V_{h\left(  re\right)  }^{\ast}=r^{-1}V_{he}^{\ast}$, and as above let us
define the sets $S_{re,\varepsilon}=V_{h\left(  re\right)  }^{\ast}%
\cap\varepsilon\operatorname{ball}V^{\ast}$ for all $\varepsilon>0$. Prove
that $S_{re,\varepsilon}=r^{-1}S_{r\varepsilon}$. Indeed, $y\in
S_{re,\varepsilon}$ iff $\left\Vert y\right\Vert \leq\varepsilon$,
$\left\langle re,y\right\rangle =1$, which means $y=r^{-1}ry$ with $\left\Vert
ry\right\Vert \leq r\varepsilon$, $\left\langle e,ry\right\rangle =1$, that
is, $ry\in S_{r\varepsilon}$. Conversely, if $y\in S_{r\varepsilon}$ then
$\left\Vert r^{-1}y\right\Vert \leq r^{-1}r\varepsilon=\varepsilon$ and
$\left\langle re,r^{-1}y\right\rangle =\left\langle e,y\right\rangle =1$,
which means that $r^{-1}y\in S_{re,\varepsilon}$. Thus $S_{re,\varepsilon
}=r^{-1}S_{r\varepsilon}$ and by passing to the polars we obtain that
$S_{re,\varepsilon}^{\circ}=rS_{r\varepsilon}^{\circ}$. In particular, for
$\varepsilon>1/\left\Vert e\right\Vert $ we have $rS_{\varepsilon}^{\circ
}=rS_{r\varepsilon/r}^{\circ}=S_{re,\varepsilon/r}^{\circ}$ with
$\varepsilon/r>1/\left(  \left\Vert e\right\Vert r\right)  $. Hence
$\operatorname{ball}V\sim S_{\varepsilon}^{\circ}$, $\varepsilon>1/\left\Vert
e\right\Vert $ iff $\operatorname{ball}V\sim S_{re,\varepsilon}^{\circ}$ for
all $\varepsilon>1/\left(  \left\Vert e\right\Vert r\right)  $. Thus the
assertion is reduced to the case of a unit vector $e$ (just put $r=\left\Vert
e\right\Vert ^{-1}$), and we have to prove that $\operatorname{ball}V\sim
S_{\varepsilon}^{\circ}$ for every $\varepsilon>1$.

By Lemma \ref{lemUNS1}, $\operatorname{ball}V\sim S_{2}^{\circ}$, namely,
$2^{-1}\operatorname{abc}\left(  S_{2}\right)  \subseteq\operatorname{ball}%
V^{\ast}\subseteq6\operatorname{abc}\left(  S_{2}\right)  $ (see
(\ref{abs1})). Assume that $\varepsilon^{-1}\operatorname{abc}\left(
S_{\varepsilon}\right)  \subseteq\operatorname{ball}V^{\ast}\subseteq
R\operatorname{abc}\left(  S_{\varepsilon}\right)  $ for some $\varepsilon>1$
and real $R$. Pick $0<\delta\leq\left(  2R+1\right)  ^{-1}$ and $\sigma
=\varepsilon-\left(  \varepsilon-1\right)  \delta$. Notice that
\[
1<\dfrac{2R\varepsilon+1}{2R+1}=\varepsilon-\dfrac{\varepsilon-1}{2R+1}%
\leq\varepsilon-\left(  \varepsilon-1\right)  \delta=\sigma<\varepsilon,
\]
that is, $\delta<1<\sigma<\varepsilon$. For $\varepsilon=2$ we have $R=6$ and
$\delta\leq13^{-1}$. Prove that $\operatorname{ball}V^{\ast}\subseteq\left(
R+1\right)  \operatorname{abc}\left(  S_{\sigma}\right)  $. Take
$w\in\operatorname{ball}V^{\ast}$. Then $w\in R\operatorname{abc}\left(
S_{\varepsilon}\right)  $, which means that $w=Rz$ with $z\in
\operatorname{abc}\left(  S_{\varepsilon}\right)  $, $z=\sum_{i=1}^{n}%
a_{i}z_{i}$, $\sum\left\vert a_{i}\right\vert \leq1$ and $\left\{
z_{i}\right\}  \subseteq S_{\varepsilon}$. As in the proof of Lemma
\ref{lemUNS1}, fix $y\in S_{1}$ and put $z_{i,\delta}=\left(  1-\delta\right)
z_{i}+\delta y$ with $\left\langle e,z_{i,\delta}\right\rangle =\left(
1-\delta\right)  +\delta=1$. Thus $\left\{  z_{i,\delta}\right\}  \subseteq
V_{he}^{\ast}$ and $\left\Vert z_{i,\delta}\right\Vert \leq\left(
1-\delta\right)  \left\Vert z_{i}\right\Vert +\delta\leq\left(  1-\delta
\right)  \varepsilon+\delta=\sigma$, which means that $\left\{  z_{i,\delta
}\right\}  \subseteq V_{he}^{\ast}\cap\sigma\operatorname{ball}V^{\ast
}=S_{\sigma}$. If $z_{\delta}=\sum_{i=1}^{n}a_{i}z_{i,\delta}$ and $w_{\delta
}=Rz_{\delta}$, then $z_{\delta}\in\operatorname{abc}\left(  S_{\sigma
}\right)  $ and
\[
w_{\delta}=\left(  1-\delta\right)  Rz+R\delta ay=\left(  1-\delta\right)
w+R\delta ay,
\]
where $a=\sum a_{i}$ with $\left\vert a\right\vert \leq1$. It follows that%
\[
\left(  R+1\right)  ^{-1}w=\frac{R}{\left(  1-\delta\right)  \left(
R+1\right)  }z_{\delta}-\frac{\delta R}{\left(  1-\delta\right)  \left(
R+1\right)  }ay
\]
with
\[
\frac{R}{\left(  1-\delta\right)  \left(  R+1\right)  }+\frac{R\delta}{\left(
1-\delta\right)  \left(  R+1\right)  }=\frac{R\left(  1+\delta\right)
}{\left(  R+1\right)  \left(  1-\delta\right)  }\leq\frac{R\left(  1+\left(
2R+1\right)  ^{-1}\right)  }{\left(  R+1\right)  \left(  1-\left(
2R+1\right)  ^{-1}\right)  }=1.
\]
Consequently, $\left(  R+1\right)  ^{-1}w\in\operatorname{abc}\left\{
z_{\delta},ay\right\}  \subseteq\operatorname{abc}\left(  S_{\sigma}\right)  $
(note that $y\in S_{1}\subseteq S_{\sigma}$, $\sigma>1$) or $w\in\left(
R+1\right)  \operatorname{abc}\left(  S_{\sigma}\right)  $. Thus%
\begin{equation}
\sigma^{-1}\operatorname{abc}\left(  S_{\sigma}\right)  \subseteq
\operatorname{ball}V^{\ast}\subseteq\left(  R+1\right)  \operatorname{abc}%
\left(  S_{\sigma}\right)  , \label{abs19}%
\end{equation}
which means that $\operatorname{abc}\left(  S_{\sigma}\right)  \sim
\operatorname{ball}V^{\ast}$.

Now based on Lemma \ref{lemUNS1}, we put $\sigma_{0}=2$, $R=6$, $\delta
_{0}=\left(  2R+1\right)  ^{-1}=1/13$, $\sigma_{1}=\sigma_{0}\left(
1-\delta_{0}\right)  +\delta_{0}=25/13$ and $\delta_{1}=\left(  2\left(
R+1\right)  +1\right)  ^{-1}=1/15$. One can iterate the same argument to
obtain a recursive sequence. Namely, we have $\delta_{n}=\left(  2\left(
6+n\right)  +1\right)  ^{-1}$ and $\sigma_{n}=\sigma_{n-1}\left(
1-\delta_{n-1}\right)  +\delta_{n-1}$, that is,
\[
\sigma_{n}=\sigma_{n-1}\frac{2n+10}{2n+11}+\frac{1}{2n+11},\quad\sigma
_{1}=25/13.
\]
By induction on $n$ one can prove that the recursive sequence is given by the
rule
\[
\sigma_{n}=\frac{693\sqrt{\pi}\Gamma\left(  n+6\right)  }{512\Gamma
(n+13/2)}+1
\]
and $\lim_{n}\sigma_{n}=1$ (the formula itself is due to WolframAlpha
computation programing). Optionally, one can use the relation $\sigma
_{n}=\left(  \sigma_{n-1}-\sigma_{n}\right)  \left(  2n+10\right)  +1$ and the
fact $\sigma_{n}>1$ to conclude that $\sigma_{n}\downarrow1$. Then for every
$\varepsilon>1$ we have $\sigma_{n}<\varepsilon$ for large $n$, and
\[
\varepsilon^{-1}\operatorname{abc}\left(  S_{\varepsilon}\right)
\subseteq\operatorname{ball}V^{\ast}\subseteq\left(  6+n+1\right)
\operatorname{abc}\left(  S_{\sigma_{n}}\right)  \subseteq\left(
6+n+1\right)  \operatorname{abc}\left(  S_{\varepsilon}\right)
\]
by virtue of (\ref{abs19}). Whence $\operatorname{ball}V\sim S_{\varepsilon
}^{\circ}$ for every $\varepsilon>1$.
\end{proof}

\begin{corollary}
\label{corUNS1}Let $\left(  V,e\right)  $ be a unital $\ast$-normed space with
its state space and let $\varepsilon>1/\left\Vert e\right\Vert $. Then
$\left\Vert v\right\Vert _{\varepsilon}=\vee\left\vert \left\langle
v,S_{\varepsilon}\right\rangle \right\vert $, $v\in V$ is a norm on $V$
equivalent to the original norm of $V$.
\end{corollary}

\begin{proof}
The norm $\left\Vert \cdot\right\Vert _{\varepsilon}$ is the Minkowski
functional of $S_{\varepsilon}^{\circ}$. It remains to use Theorem
\ref{propUNS1}.
\end{proof}

\subsection{The $\varepsilon$-cones of a unital $\ast$-normed space}

As above let $\left(  V,e\right)  $ be a unital $\ast$-normed space with its
hermitian $\operatorname{ball}V$ and the unit $e$. If $\operatorname{ball}V$
is unital then $V$ is a function system on $S_{1}$ (up to an equivalent norm)
by virtue of Lemma \ref{lemUNS0}. If $\operatorname{ball}V$ is not unital then
$\operatorname{ball}V\sim S_{\varepsilon}^{\circ}$ for every $\varepsilon
>1/\left\Vert e\right\Vert $ thanks to Theorem \ref{propUNS1}. Put
$\mathfrak{c}_{\varepsilon}$ to be the set of those $v\in V_{h}$ such that
$\left\langle v,S_{\varepsilon}\right\rangle \geq0$.

\begin{theorem}
\label{thFS1}Let $\left(  V,e\right)  $ be a unital $\ast$-normed space with
its state space $S$. Then $\mathfrak{c}_{\varepsilon}$ is a separated, closed,
unital, cone in $V$, $S\left(  \mathfrak{c}_{\varepsilon}\right)
=S_{\varepsilon}$ and $\left(  V,\mathfrak{c}_{\varepsilon}\right)  $ is a
concrete function system on $S_{\varepsilon}$ whose $\varepsilon$-norm is
equivalent to the original one of $V$ for every $\varepsilon>1/\left\Vert
e\right\Vert $.
\end{theorem}

\begin{proof}
Consider the absolute polar $S_{\varepsilon}^{\circ}$ of $S_{\varepsilon}$.
Prove that $S_{\varepsilon}^{\circ}$ is a unital ball, whose state space
equals to $V_{he}^{\ast}\cap S_{\varepsilon}^{\circ\circ}$. Note that
$S_{\varepsilon}=V_{he}^{\ast}\cap S_{\varepsilon}\subseteq V_{he}^{\ast}\cap
S_{\varepsilon}^{\circ\circ}\subseteq S_{\varepsilon}^{\circ\circ}$, which in
turn implies that $S_{\varepsilon}^{\circ}=S_{\varepsilon}^{\circ\circ\circ
}\subseteq\left(  V_{he}^{\ast}\cap S_{\varepsilon}^{\circ\circ}\right)
^{\circ}\subseteq S_{\varepsilon}^{\circ}$. Thus $S_{\varepsilon}^{\circ
}=\left(  S_{\varepsilon}^{\circ\circ}\cap V_{he}^{\ast}\right)  ^{\circ}$,
that is, the ball $S_{\varepsilon}^{\circ}$ is the polar of its state space
$S_{\varepsilon}^{\circ\circ}\cap V_{he}^{\ast}$. It means that
$S_{\varepsilon}^{\circ}$ is a unital ball. Thus $V$ equipped with the
Minkowski functional of $S_{\varepsilon}^{\circ}$ turns out to be a function
system. Since $S_{\varepsilon}\subseteq\varepsilon\operatorname{ball}V^{\ast}%
$, it follows that $\operatorname{abc}\left(  S_{\varepsilon}\right)
^{-w^{\ast}}\subseteq\varepsilon\operatorname{ball}V^{\ast}$ and%
\[
S_{\varepsilon}\subseteq V_{he}^{\ast}\cap S_{\varepsilon}^{\circ\circ}%
=V_{he}^{\ast}\cap\operatorname{abc}\left(  S_{\varepsilon}\right)
^{-w^{\ast}}\subseteq V_{he}^{\ast}\cap\varepsilon\operatorname{ball}V^{\ast
}=S_{\varepsilon},
\]
that is, $S_{\varepsilon}=V_{he}^{\ast}\cap S_{\varepsilon}^{\circ\circ}$
turns out to be the state space of the unital ball $S_{\varepsilon}^{\circ}$.
In particular, $\mathfrak{c}_{\varepsilon}=\left\{  v\in V_{h}:\left\langle
v,V_{he}^{\ast}\cap S_{\varepsilon}^{\circ\circ}\right\rangle \geq0\right\}
$. By Lemma \ref{lemUNS0}, $\mathfrak{c}_{\varepsilon}$ is a separated,
closed, unital cone in $V$ such that $S\left(  \mathfrak{c}_{\varepsilon
}\right)  =V_{he}^{\ast}\cap S_{\varepsilon}^{\circ\circ}=S_{\varepsilon}$.
Thus $\left(  V,\mathfrak{c}_{\varepsilon}\right)  $ turns out to be a
concrete function system on $S_{\varepsilon}$. The norm $\left\Vert
v\right\Vert _{\varepsilon}=\vee\left\vert \left\langle v,S_{\varepsilon
}\right\rangle \right\vert $, $v\in V$ of this function system is equivalent
to the original one by Corollary \ref{corUNS1}.
\end{proof}

Thus every unital $\ast$-normed space $V$ turns out to be a function system
equipped with the cones $\mathfrak{c}_{\varepsilon}$, $\varepsilon
>1/\left\Vert e\right\Vert $. In this case, we say that a unital $\ast$-normed
space $V$ can be equipped with \textit{a projective positivity. }Let
us\textit{ }illustrate the projective positivity in the case of the
$2$-dimensional $\ell^{p}$ spaces. First consider the operator $\ell^{\infty}%
$-system $V=\mathbb{C}\overset{\infty}{\oplus}\mathbb{C}$ with its unit
$e=\left(  1,1\right)  $. In this case, $V^{\ast}=\mathbb{C}\overset{1}{\oplus
}\mathbb{C}$, $\operatorname{ball}V^{\ast}=\left\{  w=\left(  w_{0}%
,w_{1}\right)  \in\mathbb{C}^{2}:\left\vert w_{0}\right\vert +\left\vert
w_{1}\right\vert \leq1\right\}  $ and
\[
\left\langle z,w^{\ast}\right\rangle =\left\langle z^{\ast},w\right\rangle
^{\ast}=\left(  z_{0}^{\ast}w_{0}+z_{1}^{\ast}w_{1}\right)  ^{\ast}=z_{0}%
w_{0}^{\ast}+z_{1}w_{1}^{\ast}%
\]
for all $z=\left(  z_{0},z_{1}\right)  \in V$ and $w=\left(  w_{0}%
,w_{1}\right)  \in V^{\ast}$, which means that $w^{\ast}=\left(  w_{0}^{\ast
},w_{1}^{\ast}\right)  $. In particular, $V_{h}^{\ast}=\mathbb{R}%
\overset{1}{\oplus}\mathbb{R}$, $V_{he}^{\ast}=\left\{  s=\left(  s_{0}%
,s_{1}\right)  \in\mathbb{R}^{2}:s_{0}+s_{1}=1\right\}  $ and
\[
S_{1}=V_{he}^{\ast}\cap\operatorname{ball}\left(  \mathbb{R}\overset{1}{\oplus
}\mathbb{R}\right)  =\left\{  s=\left(  s_{0},s_{1}\right)  \in\mathbb{R}%
^{2}:s_{0}+s_{1}=1,s_{j}\geq0\right\}
\]
is the line segment from $\mathbb{R}^{2}$ connecting $\left(  0,1\right)  $
and $\left(  1,0\right)  $. In particular, for $z\in V$ we have
\[
\left\Vert z\right\Vert _{e}=\vee\left\vert \left\langle z,S_{1}\right\rangle
\right\vert =\vee\left\{  \left\vert s_{0}z_{0}+s_{1}z_{1}\right\vert
:s_{0}+s_{1}=1,s_{j}\geq0\right\}  =\left\vert z_{0}\right\vert \vee\left\vert
z_{1}\right\vert =\left\Vert z\right\Vert _{\infty},
\]
that is, $\operatorname{ball}V$ is unital and $V$ is a function system. For
$r\in\mathbb{R}^{2}$ we have $\left\langle r,S_{1}\right\rangle \geq0$ iff
$r_{j}\geq0$. Hence $\mathfrak{c}_{1}=\mathbb{R}_{+}^{2}$ as it used to be,
where $\mathbb{R}_{+}=\left\{  x\geq0\right\}  $. Now for every $1\leq
p<\infty$ consider the $\ast$-vector space $V=\mathbb{C}\overset{p}{\oplus
}\mathbb{C}$ with the unit $e_{p}=2^{-1/p}\left(  1,1\right)  $. If $z\in V$
then $z^{\ast}=\left(  z_{0}^{\ast},z_{1}^{\ast}\right)  $ and $\left\Vert
z^{\ast}\right\Vert _{p}=\left(  \left\vert z_{0}^{\ast}\right\vert
^{p}+\left\vert z_{1}^{\ast}\right\vert ^{p}\right)  ^{1/p}=\left\Vert
z\right\Vert _{p}$, which means that $\operatorname{ball}V$ is a hermitian
ball. Thus $\left(  V,e_{p}\right)  $ is a unital $\ast$-normed space. In this
case, $V_{h}^{\ast}=\mathbb{R}\overset{q}{\oplus}\mathbb{R}$ with the
coujugate couple $\left(  p,q\right)  $, and $V_{he}^{\ast}=\left\{  s=\left(
s_{0},s_{1}\right)  \in\mathbb{R}^{2}:s_{0}+s_{1}=2^{1/p}\right\}  $. In
particular,%
\[
S_{1}=V_{he}^{\ast}\cap\operatorname{ball}\left(  \mathbb{R}\overset{q}{\oplus
}\mathbb{R}\right)  =\left\{  s=\left(  s_{0},s_{1}\right)  \in\mathbb{R}%
^{2}:s_{0}^{q}+s_{1}^{q}\leq1,s_{0}+s_{1}=2^{1/p}\right\}  =\left\{
e_{q}\right\}  ,
\]
for $s_{0}+s_{1}=2^{1/p}$ is the tangent line to $\operatorname{ball}\left(
\mathbb{R}\overset{q}{\oplus}\mathbb{R}\right)  $ at $e_{q}$. Thus
$S_{1}^{\circ}$ is not separated and $\operatorname{ball}V$ is not unital.
Nonetheless, every $S_{\varepsilon}$ for $\varepsilon>1$ provides a unital
ball thanks to Theorem \ref{thFS1}.

It is worth to notify that $S_{\varepsilon}$ for $\varepsilon=2^{1/p}$
generates exactly the same cone $\mathbb{R}_{+}^{2}$. But the same phenomenon
can not be expected in higher dimensions $n\geq3$. The reason is that the
geometry of the related state space has a rounded shape, which does not define
the original cone in $\ell^{p}\left(  n\right)  $. For example, in the case of
$\ell^{2}\left(  3\right)  $ we have $e_{2}=3^{-1/2}\left(  1,1,1\right)  $,
$S_{\varepsilon}=\ell^{2}\left(  3\right)  _{he}^{\ast}\cap\varepsilon
\operatorname{ball}\left(  \ell^{2}\left(  3\right)  _{h}\right)  $, that is,
\[
S_{1}=\left\{  s=\left(  s_{0},s_{1},s_{2}\right)  \in\mathbb{R}^{2}:s_{0}%
^{2}+s_{1}^{2}+s_{2}^{2}\leq1,s_{0}+s_{1}+s_{2}=3^{1/2}\right\}  =\left\{
e_{2}\right\}  ,
\]
and $S_{\varepsilon}=\left\{  s=\left(  s_{0},s_{1},s_{2}\right)
\in\mathbb{R}^{2}:s_{0}^{2}+s_{1}^{2}+s_{2}^{2}\leq\varepsilon^{2},s_{0}%
+s_{1}+s_{2}=3^{1/2}\right\}  $ is a disk centered at $e_{2}$ on the plane
$s_{0}+s_{1}+s_{2}=3^{1/2}$. But that disk $S_{\varepsilon}$ can either be
inscribed into the triangle $\Delta=\left\{  s_{0}+s_{1}+s_{2}=3^{1/2}%
,s_{j}\geq0,j=1,2,3\right\}  $ or cover it up, which means that the related
cone would never be the original cone of $\ell^{2}\left(  3\right)  $.
Nonetheless the projective positivity of $\ell^{2}\left(  3\right)  $ allows
to realize it as a function system on the disk $S_{\varepsilon}$ thanks to
Theorem \ref{thFS1}.

\section{Projective positivity in $L^{p}$-spaces\label{secLP}}

In this section we analyze the projective positivity in the unital $\ast
$-spaces $L_{\mu}^{p}\left(  \mathcal{T}\right)  $ for a locally compact,
$\sigma$-compact, Hausdorff topological space $\mathcal{T}$ equipped with a
finite Radon measure $\mu$. It is well known that $L_{\mu}^{p}\left(
\mathcal{T}\right)  _{h}$ is a Banach lattice with its natural positivity
given by the cone $L_{\mu}^{p}\left(  \mathcal{T}\right)  _{+}$ of $\mu
$-almost everywhere positive functions from $L_{\mu}^{p}\left(  \mathcal{T}%
\right)  $. How does the cone $L_{\mu}^{p}\left(  \mathcal{T}\right)  _{+}$
relate to the projective positivity in the unital $\ast$-space $L_{\mu}%
^{p}\left(  \mathcal{T}\right)  $, that is the main target of the present section.

\subsection{$\varepsilon$-positivity in the unital $\ast$-spaces $L_{\mu}%
^{p}\left(  \mathcal{T}\right)  $}

Let $\mu\in M\left(  \mathcal{T}\right)  _{+}$, $\left\Vert \mu\right\Vert
=\vee\left\{  \left\Vert \mu\right\Vert _{K}:K\subseteq\mathcal{T}\right\}
=\mu\left(  \mathcal{T}\right)  <\infty$, where $K\subseteq\mathcal{T}$ is
running over all compact subsets, $\left\Vert \mu\right\Vert _{K}=\left\Vert
\mu|C_{K}\left(  \mathcal{T}\right)  \right\Vert $ and $C_{K}\left(
\mathcal{T}\right)  \subseteq C_{c}\left(  \mathcal{T}\right)  $ is the $\ast
$-ideal of those continuous functions with supports inside of $K$. Note that
the related Banach spaces $L_{\mu}^{p}\left(  \mathcal{T}\right)  $, $1\leq
p\leq\infty$ are unital $\ast$-normed spaces with their units $e_{p}%
=\mu\left(  \mathcal{T}\right)  ^{-1/p}1$. For $p=\infty$ we come up with the
abelian von Neumann algebra $L_{\mu}^{\infty}\left(  \mathcal{T}\right)  $
with the unit $e_{\infty}=1$. For brevity we skip the notation $\mu$ and
consider the unital $\ast$-normed spaces $\left(  L^{p}\left(  \mathcal{T}%
\right)  ,e_{p}\right)  $, $1\leq p\leq\infty$. Certainly $\left(  L^{\infty
}\left(  \mathcal{T}\right)  ,e_{\infty}\right)  $ is an operator system,
which turns out to be a function system (see Proposition \ref{propAp1}). In
particular, $\operatorname{ball}L^{\infty}\left(  \mathcal{T}\right)  $ is
unital. But that is not the case for the rest $\operatorname{ball}L^{p}\left(
\mathcal{T}\right)  $, $1\leq p<\infty$ (see below Lemma \ref{lemLP0}).

We fix $1\leq p<\infty$ and $\varepsilon\geq1$, and put $S_{\varepsilon}%
^{p}=L^{p}\left(  \mathcal{T}\right)  _{he}^{\ast}\cap\varepsilon
\operatorname{ball}L^{q}\left(  \mathcal{T}\right)  $, where $\left(
p,q\right)  $ is the conjugate couple. Recall that $\left(  L^{p}\left(
\mathcal{T}\right)  ,L^{q}\left(  \mathcal{T}\right)  \right)  $ is a dual
$\ast$-pair with the duality $\left\langle g,f\right\rangle =\int gf$, $g\in
L^{p}\left(  \mathcal{T}\right)  $, $f\in L^{q}\left(  \mathcal{T}\right)  $,
where $\int$ indicates to the integral related to $\mu$. In particular,
$\left\langle e_{p},f\right\rangle =\mu\left(  \mathcal{T}\right)  ^{-1/p}\int
f$, $f\in L^{q}\left(  \mathcal{T}\right)  $ and $\left\langle e_{p}%
,e_{q}\right\rangle =\mu\left(  \mathcal{T}\right)  ^{-1/p}\int e_{q}=1$. Note
that
\[
S_{1}^{p}=\left\{  f\in L^{q}\left(  \mathcal{T}\right)  _{h}:\left\Vert
f\right\Vert _{q}\leq1,\int f=\mu\left(  \mathcal{T}\right)  ^{1/p}\right\}
\subseteq\partial\operatorname{ball}L^{q}\left(  \mathcal{T}\right)
\]
and $e_{q}\in S_{1}^{p}$ (see Subsection \ref{subsecUNS1}). The class of all
$\mu$-measurable subsets $E\subseteq\mathcal{T}$ is denoted by $\mathfrak{M}%
_{\mu}$.

\begin{lemma}
\label{lemLP0}The equality $S_{1}^{p}=\left\{  e_{q}\right\}  $ holds, that
is, $\operatorname{ball}L^{p}\left(  \mathcal{T}\right)  $ is not unital for
all $1\leq p<\infty$.
\end{lemma}

\begin{proof}
First prove that $S_{1}^{p}\subseteq L^{q}\left(  \mathcal{T}\right)  _{+}$.
If $f\in S_{1}^{p}$ and $K\subseteq\left\{  f\leq-n^{-1}\right\}  $ is a
compact set, then $nf\left[  K\right]  +\left[  K\right]  \leq0$ and $n\int
f-n\int f\left[  \mathcal{T}\backslash K\right]  +\mu\left(  K\right)  \leq0$.
It follows that
\[
n\mu\left(  \mathcal{T}\right)  ^{1/p}+\mu\left(  K\right)  \leq n\int
f\left[  \mathcal{T}\backslash K\right]  \leq n\left\Vert f\right\Vert
_{q}\left\Vert \left[  \mathcal{T}\backslash K\right]  \right\Vert _{p}\leq
n\mu\left(  \mathcal{T}\backslash K\right)  ^{1/p}.
\]
But $\mu\left(  \mathcal{T}\right)  \geq\mu\left(  \mathcal{T}\backslash
K\right)  $ and $n\left(  \mu\left(  \mathcal{T}\right)  ^{1/p}-\mu\left(
\mathcal{T}\backslash K\right)  ^{1/p}\right)  +\mu\left(  K\right)  \leq0$,
therefore $\mu\left(  K\right)  =0$. Thus $\mu\left\{  f\leq0\right\}  =0$ or
$f\in L^{q}\left(  \mathcal{T}\right)  _{+}$.

Now assume that $p=1$ and $f\in S_{1}^{p}$. If $0<r<1$ and $K\subseteq\left\{
f\leq r\right\}  $ is a compact subset, then $f\left[  K\right]  \leq r\left[
K\right]  $, $\int f=\mu\left(  \mathcal{T}\right)  $ and $\mu\left(
\mathcal{T}\right)  -\int f\left[  \mathcal{T}\backslash K\right]  \leq
r\mu\left(  K\right)  $. It follows that
\[
\mu\left(  \mathcal{T}\right)  -r\mu\left(  K\right)  \leq\int f\left[
\mathcal{T}\backslash K\right]  \leq\left\Vert f\right\Vert _{\infty}%
\mu\left(  \mathcal{T}\backslash K\right)  \leq\mu\left(  \mathcal{T}%
\backslash K\right)  =\mu\left(  \mathcal{T}\right)  -\mu\left(  K\right)  ,
\]
that is, $\mu\left(  K\right)  =0$. Thus $f\geq1$ in $L^{\infty}\left(
\mathcal{T}\right)  _{h}$ and $\left\Vert f\right\Vert _{\infty}=1$. But
$\wedge\left\{  r\in\mathbb{R}:\mu\left\{  f>r\right\}  =0\right\}
=\left\Vert f\right\Vert _{\infty}=1$, which means that for every $r>1$,
$\mu\left\{  f>r\right\}  =0$. Hence $f=e_{\infty}$.

Finally, consider the case of $p>1$. Since $S_{1}^{p}\subseteq\partial
\operatorname{ball}L^{q}\left(  \mathcal{T}\right)  _{+}$ is a convex subset
and $e_{q}\in S_{1}^{p}$, it follows that $2^{-1}\left(  f+e_{q}\right)  \in
S_{1}^{p}$ whenever $f\in S_{1}^{p}$. In particular, $\left\Vert 2^{-1}\left(
f+e_{q}\right)  \right\Vert _{q}=\left\Vert f\right\Vert _{q}=\left\Vert
e_{q}\right\Vert _{q}=1$. Based on the geometry of $\operatorname{ball}%
L^{q}\left(  \mathcal{T}\right)  $ given by the property to be uniformly
convex set (see \cite{HHO}), we obtain that for every $\varepsilon>0$ there
corresponds $\delta>0$ such that $g,h\in\operatorname{ball}L^{q}\left(
\mathcal{T}\right)  $ with $\left\Vert g+h\right\Vert _{q}>2-\delta$ implies
that $\left\Vert g-h\right\Vert <\varepsilon$. Since $\left\Vert
f+e_{q}\right\Vert _{q}=2$, we obtain that $\left\Vert f-e_{q}\right\Vert
_{q}<\varepsilon$. Whence $f=e_{q}$. Optionally, one can use Clarkson's
inequalities in $L^{q}\left(  \mathcal{T}\right)  $.
\end{proof}

Thus $S_{1}^{p}$ can not be a state space of a reasonable unital cone in
$L^{p}\left(  \mathcal{T}\right)  $.

\subsection{Comparison of the cones}

By Theorem \ref{thFS1}, the unital $\ast$-normed space $L^{p}\left(
\mathcal{T}\right)  $ has the projective positivity given by the separated,
closed, unital cone $L^{p}\left(  \mathcal{T}\right)  _{\varepsilon}^{+}$
whose state space is $S_{\varepsilon}^{p}$ for $\varepsilon>1$. It turns out
that the original cone $L^{p}\left(  \mathcal{T}\right)  _{+}$ of the Banach
lattice $L^{p}\left(  \mathcal{T}\right)  _{h}$ and $L^{p}\left(
\mathcal{T}\right)  _{\varepsilon}^{+}$ for $1\leq p<\infty$ are totally
different cones, they are not comparable unless $\operatorname{supp}\left(
\mu\right)  $ is finite. That is the main result of the present section which
will be proven in several steps.

\begin{lemma}
\label{lemLP1}Let $\mathfrak{c}\subseteq L^{p}\left(  \mathcal{T}\right)
_{h}$ be a norm closed cone. Then $L^{p}\left(  \mathcal{T}\right)
_{+}\subseteq\mathfrak{c}$ iff $\left[  K\right]  \in\mathfrak{c}$ for every
compact set $K\subseteq\mathcal{T}$ with $\mu\left(  K\right)  >0$.
\end{lemma}

\begin{proof}
Since $\left[  K\right]  \in L^{p}\left(  \mathcal{T}\right)  _{+}$ for every
$K$, one side implication is immediate. Conversely, suppose that $\left\{
\left[  K\right]  :K\subseteq\mathcal{T}\right\}  \subseteq\mathfrak{c}$. Pick
$E\in\mathfrak{M}_{\mu}$. Since $\mu\left(  E\right)  =\vee\left\{  \mu\left(
K\right)  :K\subseteq\mathcal{T}\right\}  $, it follows that $\mu\left(
K_{n}\right)  \uparrow\mu\left(  E\right)  $ for a certain sequence $\left\{
K_{n}\right\}  $ of compact subsets in $E$. But $\left\Vert \left[  E\right]
-\left[  K\right]  \right\Vert _{p}=\left(  \int\left[  E\backslash
K_{n}\right]  \right)  ^{1/p}=\mu\left(  E\backslash K_{n}\right)
^{1/p}\rightarrow0$, therefore $\left[  E\right]  \in\mathfrak{c}$ being a
closed cone. In particular, all positive $\mu$-step functions belong to
$\mathfrak{c}$. If $f\in L^{p}\left(  \mathcal{T}\right)  _{+}$ then
$f_{n}\uparrow f$ for a sequence $\left\{  f_{n}\right\}  \subseteq
L^{p}\left(  \mathcal{T}\right)  _{+}$ of step functions. Then $\left(
f-f_{n}\right)  \downarrow0$ and $\left(  f-f_{n}\right)  ^{p}\leq f^{p}%
-f_{n}^{p}\leq f^{p}$ with $\int f^{p}<\infty$. By Lebesgue's convergence
theorem, $\int\left(  f-f_{n}\right)  ^{p}\downarrow0$. Taking into account
that $\left\{  f_{n}\right\}  \subseteq\mathfrak{c}$, we conclude that
$f\in\mathfrak{c}$ too.
\end{proof}

\begin{lemma}
\label{lemLP2}The inclusion $L^{p}\left(  \mathcal{T}\right)  _{+}\subseteq
L^{p}\left(  \mathcal{T}\right)  _{\varepsilon}^{+}$ for some $\varepsilon>1$
holds iff $S_{\varepsilon}^{p}\subseteq L^{q}\left(  \mathcal{T}\right)  _{+}$.
\end{lemma}

\begin{proof}
Suppose that $L^{p}\left(  \mathcal{T}\right)  _{+}\subseteq L^{p}\left(
\mathcal{T}\right)  _{\varepsilon}^{+}$. If $f\in S_{\varepsilon}^{p}$ then
$\left\langle L^{p}\left(  \mathcal{T}\right)  _{\varepsilon}^{+}%
,f\right\rangle \geq0$ and therefore $\left\langle \left[  K\right]
,f\right\rangle \geq0$. As in the proof of Lemma \ref{lemLP0}, for a compact
set $K\subseteq\left\{  f\leq-n^{-1}\right\}  $ we have $nf\left[  K\right]
+\left[  K\right]  \leq0$ and $n\left\langle \left[  K\right]  ,f\right\rangle
+\mu\left(  K\right)  \leq0$, which means that $\mu\left(  K\right)  =0$. It
follows that $f\in L^{q}\left(  \mathcal{T}\right)  _{+}$.

Conversely, suppose that $S_{\varepsilon}^{p}\subseteq L^{q}\left(
\mathcal{T}\right)  _{+}$. Then $\left\langle \left[  K\right]
,S_{\varepsilon}^{p}\right\rangle =\int_{K}S_{\varepsilon}^{p}\geq0$ for every
compact set $K$. By Theorem \ref{thFS1}, we have $S_{\varepsilon}^{p}=S\left(
L^{p}\left(  \mathcal{T}\right)  _{\varepsilon}^{+}\right)  $, which means
that $\left[  K\right]  \in L^{p}\left(  \mathcal{T}\right)  _{\varepsilon
}^{+}$. But $L^{p}\left(  \mathcal{T}\right)  _{\varepsilon}^{+}$ is a closed
cone and $\left\{  \left[  K\right]  :K\subseteq\mathcal{T}\right\}  \subseteq
L^{p}\left(  \mathcal{T}\right)  _{\varepsilon}^{+}$. By Lemma \ref{lemLP1},
$L^{p}\left(  \mathcal{T}\right)  _{+}\subseteq L^{p}\left(  \mathcal{T}%
\right)  _{\varepsilon}^{+}$ holds.
\end{proof}

Now consider the sets $E\in\mathfrak{M}_{\mu}$ with $\mu\left(  E\right)  >0$
and put $r_{\mu}=\wedge\left\{  \mu\left(  E\right)  :E\in\mathfrak{M}_{\mu
},\mu\left(  E\right)  >0\right\}  $.

\begin{lemma}
\label{lemLP3}If $L^{p}\left(  \mathcal{T}\right)  _{+}\subseteq L^{p}\left(
\mathcal{T}\right)  _{\varepsilon}^{+}$ for some $\varepsilon>1$, then
$r_{\mu}>0$.
\end{lemma}

\begin{proof}
Suppose that the inclusion $L^{p}\left(  \mathcal{T}\right)  _{+}\subseteq
L^{p}\left(  \mathcal{T}\right)  _{\varepsilon}^{+}$ holds for some
$\varepsilon>1$ whereas $r_{\mu}=0$. Thus one can choose the compact sets $K$
with sufficiently small $\mu\left(  K\right)  $. First consider the case of
$p=1$. Pick a small real $0<a<1$ and a compact $K$ with $0<\mu\left(
K\right)  /\mu\left(  \mathcal{T}\right)  \leq\left(  \varepsilon-1\right)
/\left(  \varepsilon+a\right)  $. Since the rational function $y=\dfrac
{x-1}{x+a}$, $1\leq x\leq\varepsilon$ is an increasing continuous function, it
follows that $\mu\left(  K\right)  /\mu\left(  \mathcal{T}\right)  =\left(
b-1\right)  /\left(  b+a\right)  $ for a certain $1<b\leq\varepsilon$. Then
$f=b\left[  \mathcal{T}\backslash K\right]  -a\left[  K\right]  \in L^{\infty
}\left(  \mathcal{T}\right)  _{h}\backslash L^{\infty}\left(  \mathcal{T}%
\right)  _{+}$ with
\begin{align*}
\left\langle e_{1},f\right\rangle  &  =\mu\left(  \mathcal{T}\right)
^{-1}\int f=\mu\left(  \mathcal{T}\right)  ^{-1}\left(  b\mu\left(
\mathcal{T}\backslash K\right)  -a\mu\left(  K\right)  \right)  =1,\text{
and}\\
\left\Vert f\right\Vert _{\infty}  &  =b\vee a\leq\varepsilon.
\end{align*}
It means that $f\in S_{\varepsilon}^{1}\backslash L^{\infty}\left(
\mathcal{T}\right)  _{+}$, which contradicts to Lemma \ref{lemLP2}.

Now assume that $1<p<\infty$. As above put $0<a<1$ and $b=\dfrac{a\mu\left(
K\right)  +\mu\left(  \mathcal{T}\right)  ^{1/p}}{\mu\left(  \mathcal{T}%
\backslash K\right)  }$. By choosing small $\mu\left(  K\right)  >0$, we see
that%
\[
b^{q-1}\mu\left(  \mathcal{T}\right)  ^{1/p}=\left(  \dfrac{a\mu\left(
K\right)  +\mu\left(  \mathcal{T}\right)  ^{1/p}}{\mu\left(  \mathcal{T}%
\backslash K\right)  }\right)  ^{q-1}\mu\left(  \mathcal{T}\right)  ^{1/p}%
\sim\left(  \dfrac{\mu\left(  \mathcal{T}\right)  ^{1/p}}{\mu\left(
\mathcal{T}\right)  }\right)  ^{q-1}\mu\left(  \mathcal{T}\right)
^{1/p}=1<\varepsilon^{q},
\]
that is, $b^{q-1}\mu\left(  \mathcal{T}\right)  ^{1/p}<\varepsilon^{q}$ for
some compact $K$ with small $\mu\left(  K\right)  >0$. Further, by squeezing
$a$, we can also assume that $b^{q-1}\mu\left(  \mathcal{T}\right)
^{1/p}+\left(  b^{q-1}+a^{q-1}\right)  a\mu\left(  K\right)  \leq
\varepsilon^{q}$. If $f=b\left[  \mathcal{T}\backslash K\right]  -a\left[
K\right]  $, then
\begin{align*}
\left\langle e_{p},f\right\rangle  &  =\mu\left(  \mathcal{T}\right)
^{-1/p}\int f=\mu\left(  \mathcal{T}\right)  ^{-1/p}\left(  b\mu\left(
\mathcal{T}\backslash K\right)  -a\mu\left(  K\right)  \right)  =1\text{,
and}\\
\left\Vert f\right\Vert _{q}^{q}  &  =\int\left\vert f\right\vert ^{q}%
=b^{q}\mu\left(  \mathcal{T}\backslash K\right)  +a^{q}\mu\left(  K\right)
=b^{q-1}\left(  a\mu\left(  K\right)  +\mu\left(  \mathcal{T}\right)
^{1/p}\right)  +a^{q}\mu\left(  K\right) \\
&  =b^{q-1}\mu\left(  \mathcal{T}\right)  ^{1/p}+\left(  b^{q-1}%
+a^{q-1}\right)  a\mu\left(  K\right)  \leq\varepsilon^{q},
\end{align*}
which means that $f\in L^{p}\left(  \mathcal{T}\right)  _{he}^{\ast}%
\cap\varepsilon\operatorname{ball}L^{q}\left(  \mathcal{T}\right)
=S_{\varepsilon}^{p}$. But $f\notin L^{q}\left(  \mathcal{T}\right)  _{+}$,
which again contradicts to Lemma \ref{lemLP2}.
\end{proof}

Recall that by an atom in $\mathcal{T}$ for $\mu$ we mean a measurable subset
$A\subseteq\mathcal{T}$, $\mu\left(  A\right)  >0$ such that for every
measurable subset $B\subseteq A$ with $\mu\left(  B\right)  <\mu\left(
A\right)  $ we have $\mu\left(  B\right)  =0$. A locally summable sequence
$\left\{  c_{t}:t\in S\right\}  $ in $\mathbb{R}_{+}$ defines an atomic
measure $\lambda=\sum_{t}c_{t}\delta_{t}\in M\left(  \mathcal{T}\right)  _{+}$
such that $\operatorname{supp}\left(  \lambda\right)  $ is the closure of its
atoms $S=\left\{  t\in\mathcal{T}:c_{t}>0\right\}  $. In this case,
$\lambda=\left[  S\right]  \lambda$ and $\left\Vert \lambda\right\Vert
=\sum_{t\in S}c_{t}$. The closed subspace $I^{\delta}\left(  \mathcal{T}%
\right)  _{h}$ in $M\left(  \mathcal{T}\right)  _{h}$ generated by $\left\{
\delta_{t}:t\in\mathcal{T}\right\}  $ is a closed ideal of the complete
lattice $M\left(  \mathcal{T}\right)  _{h}$. It turns out that $I^{\delta
}\left(  \mathcal{T}\right)  $ is the set of all real atomic charges on
$\mathcal{T}$ exactly. Moreover, we have the decomposition $M\left(
\mathcal{T}\right)  _{h}=I^{\delta}\left(  \mathcal{T}\right)  _{h}\oplus
I^{\delta}\left(  \mathcal{T}\right)  _{h}^{\perp}$ of the closed ideals such
that the ideal complement $I^{\delta}\left(  \mathcal{T}\right)  _{h}^{\perp}$
consists of all real scattered measures on $\mathcal{T}$. In particular, there
is a unique decomposition $\mu=\lambda+\varphi$ with a positive atomic
$\lambda$ and a positive scattered $\varphi$ such that $\left\Vert
\mu\right\Vert =\left\Vert \lambda\right\Vert +\left\Vert \varphi\right\Vert
$. In this case, there are disjoint $S,T\in\mathfrak{M}_{\mu}$ such that
$\left[  S\right]  \lambda=\lambda$ and $\left[  T\right]  \varphi=\varphi$
(see \cite[5.5]{BourInt} for the details)

\begin{lemma}
\label{lemLP4}A bounded, positive, Radon measure $\mu\in M\left(
\mathcal{T}\right)  $ with $r_{\mu}>0$ is an atomic measure with the finite
support $\operatorname{supp}\left(  \mu\right)  $.
\end{lemma}

\begin{proof}
First note that every $A\in\mathfrak{M}_{\mu}$ with $r_{\mu}\leq\mu\left(
A\right)  <2r_{\mu}$ is an atom for $\mu$. Indeed, if $B\in\mathfrak{M}_{\mu}$
with $B\subseteq A$, $\mu\left(  B\right)  <\mu\left(  A\right)  $, we have
$\mu\left(  B\right)  \geq r_{\mu}$ whenever $\mu\left(  B\right)  >0$. But
$\mu\left(  A/B\right)  =\mu\left(  A\right)  -\mu\left(  B\right)  >0$ too,
therefore $\mu\left(  A/B\right)  \geq r_{\mu}$. It follows that $\mu\left(
A\right)  =\mu\left(  B\right)  +\mu\left(  A/B\right)  \geq2r_{\mu}$, a
contradiction. Conversely, for every atom $A$ for $\mu$, we obviously have
$\mu\left(  A\right)  \geq r_{\mu}$ and $\mu\left(  A\right)  =\mu\left(
\left\{  t\right\}  \right)  $ for some $t\in A$ thanks to \cite[Proposition
11]{Nemesh}. In particular, the atomic part of $\mu$ in the expansion
$I^{\delta}\left(  \mathcal{T}\right)  _{h}\oplus I^{\delta}\left(
\mathcal{T}\right)  _{h}^{\perp}$ is not trivial. Put $\mu=\lambda+\varphi$
with an atomic $\lambda$ and a scattered $\varphi$ supported on disjoint $S$
and $T$, respectively. For every $t\in S$ we have $\mu\left(  \left\{
t\right\}  \right)  =\lambda\left(  \left\{  t\right\}  \right)  =c_{t}>0$. In
particular, $\wedge\left\{  c_{t}:t\in S\right\}  \geq r_{\mu}>0$. But
$\sum_{t\in S}c_{t}<\infty$, therefore $S$ is a finite set.

If $\varphi\neq0$ then $\varphi\left(  E\right)  >0$ for some $E\in
\mathfrak{M}_{\mu}$. Since $\left[  T\right]  \varphi=\varphi$, we can assume
that $E\subseteq T$ and $\mu\left(  E\right)  =\lambda\left(  E\right)
+\varphi\left(  E\right)  =\varphi\left(  E\right)  \geq r_{\mu}$. In
particular, $r_{\varphi}=\wedge\left\{  \varphi\left(  E\right)
:E\in\mathfrak{M}_{\varphi},\varphi\left(  E\right)  >0\right\}  \geq r_{\mu
}>0$. If $r_{\varphi}\leq\varphi\left(  A\right)  <2r_{\varphi}$, then $A$ is
an atom for $\varphi$. As above $\varphi\left(  A\right)  =\varphi\left(
\left\{  t\right\}  \right)  $, and taking into account that $\varphi$ is
scattered, we deduce that $\varphi\left(  \left\{  t\right\}  \right)  =0$.
Thus $r_{\varphi}=0$, a contradiction. Whence $\mu=\lambda\in I^{\delta
}\left(  \mathcal{T}\right)  _{h}$ and it has the finite support $S$.
\end{proof}

Now we can prove the following key result. To be punctual we say that some
cones $\mathfrak{c}$ and $\mathfrak{k}$ in a vector space are
\textit{comparable} if $\mathfrak{c\subseteq k}$ or $\mathfrak{k\subseteq c}$.

\begin{theorem}
\label{thmainComm}The cones $L^{p}\left(  \mathcal{T}\right)  _{+}$ and
$L^{p}\left(  \mathcal{T}\right)  _{\varepsilon}^{+}$ are comparable for some
$\varepsilon>1$ if and only if $\operatorname{supp}\left(  \mu\right)  $ is finite.
\end{theorem}

\begin{proof}
First assume that $\operatorname{supp}\left(  \mu\right)  $ is finite, that
is, $\mu=\sum_{t\in S}c_{t}\delta_{t}$ for a finite subset $S\subseteq
\mathcal{T}$ and $c_{t}>0$. Then $L^{p}\left(  \mathcal{T}\right)  =\ell
^{p}\left(  c\right)  $ for a finite sequence $c=\left\{  c_{j}:1\leq j\leq
n\right\}  $, $c_{j}>0$, and $\left\Vert x\right\Vert _{p}=\left(  \sum
_{j=1}^{n}\left\vert x_{j}\right\vert ^{p}c_{j}\right)  ^{1/p}$ for every
$x\in\ell^{p}\left(  c\right)  $. The duality of the pair $\left(  \ell
^{p}\left(  c\right)  ,\ell^{q}\left(  c\right)  \right)  $ if given by
$\left\langle x,y\right\rangle =\sum_{j=1}^{n}x_{j}y_{j}c_{j}$. Moreover,
$y\in S_{\varepsilon}^{p}$ for $\varepsilon>1$ iff $\left\Vert y\right\Vert
_{q}\leq1$ and $\sum_{j}y_{j}c_{j}=\int y=\mu\left(  \mathcal{T}\right)
^{1/p}=\left(  \sum_{j}c_{j}\right)  ^{1/p}$. Put
\[
r_{c}=\wedge\left\{  \left(  \dfrac{\sum_{j}c_{j}}{\sum_{j\neq k}c_{j}%
}\right)  ^{1/p}:1\leq k\leq n\right\}  >1
\]
and fix $\varepsilon$ with $1<\varepsilon\leq r_{c}$. If $y\in S_{\varepsilon
}^{p}$ then $y_{k}\geq0$ for all $k$. Indeed, in the case of $y_{k}<0$ for a
certain $k$, we deduce that
\begin{align*}
\left(  \sum_{j}c_{j}\right)  ^{1/p}  &  =\sum_{j}y_{j}c_{j}<\sum_{j\neq
k}y_{j}c_{j}\leq\sum_{j\neq k}\left\vert y_{j}\right\vert c_{j}^{1/q}%
c_{j}^{1/p}\leq\left(  \sum_{j\neq k}\left\vert y_{j}\right\vert ^{q}%
c_{j}\right)  ^{1/q}\left(  \sum_{j\neq k}c_{j}\right)  ^{1/p}\\
&  \leq\left\Vert y\right\Vert _{q}\left(  \sum_{j\neq k}c_{j}\right)
^{1/p}\leq\varepsilon\left(  \sum_{j\neq k}c_{j}\right)  ^{1/p},
\end{align*}
which means that $\varepsilon>\left(  \dfrac{\sum_{j}c_{j}}{\sum_{j\neq
k}c_{j}}\right)  ^{1/p}\geq r_{c}$, a contradiction. Thus $S_{\varepsilon}%
^{p}\subseteq\ell^{q}\left(  c\right)  _{+}$, which in turn implies that
$\ell^{p}\left(  c\right)  _{+}\subseteq\ell^{p}\left(  c\right)
_{\varepsilon}^{+}$ thanks to Lemma \ref{lemLP2}.

Conversely, if $L^{p}\left(  \mathcal{T}\right)  _{+}\subseteq L^{p}\left(
\mathcal{T}\right)  _{\varepsilon}^{+}$ holds for some $\varepsilon>1$, then
$r_{\mu}>0$ by Lemma \ref{lemLP3}. Hence $\operatorname{supp}\left(
\mu\right)  $ is finite thanks to Lemma \ref{lemLP4}.

If $L^{p}\left(  \mathcal{T}\right)  _{\varepsilon}^{+}\subseteq L^{p}\left(
\mathcal{T}\right)  _{+}$ then it is immediate that $L^{p}\left(
\mathcal{T}\right)  _{+}$ is unital which is not the case unless
$\operatorname{supp}\left(  \mu\right)  $ is finite. That fact is known in
Banach lattices theory \cite[Ch. 2]{PMN} claiming that $L^{p}$-spaces have no
strong order unit but only weak one. In the present case we can provide an
independent argument similar to those in the proof of Lemma \ref{lemLP3} to
fix it. Namely, suppose that $L^{p}\left(  \mathcal{T}\right)  _{\varepsilon
}^{+}\subseteq L^{p}\left(  \mathcal{T}\right)  _{+}$ holds for some
$\varepsilon>1$. Then $S^{p}\subseteq S_{\varepsilon}^{p}$, where $S^{p}$ is
the state space of the cone $L^{p}\left(  \mathcal{T}\right)  _{+}$. But
$S^{p}$ consists of those $f\in L^{q}\left(  \mathcal{T}\right)  _{+}$ such
that $\int f=\mu\left(  \mathcal{T}\right)  ^{1/p}$, for
\[
\left\langle L^{p}\left(  \mathcal{T}\right)  _{+},f\right\rangle
\geq0\Leftrightarrow\left\langle \left[  K\right]  ,f\right\rangle
\geq0\text{, }K\subseteq\mathcal{T}\Leftrightarrow f\in L^{q}\left(
\mathcal{T}\right)  _{+}%
\]
(see Lemma \ref{lemLP1}). Based on Lemma \ref{lemLP4}, it suffices to prove
that $r_{\mu}>0$. If $r_{\mu}=0$ there are compact sets $K$ with sufficiently
small $\mu\left(  K\right)  $, and as above fix%
\[
0<a<\mu\left(  \mathcal{T}\right)  ^{-1/q}\text{, \quad}b=\dfrac{\mu\left(
\mathcal{T}\right)  ^{1/p}-a\mu\left(  \mathcal{T}\backslash K\right)  }%
{\mu\left(  K\right)  }>1
\]
for $p>1$. Then $\mu\left(  \mathcal{T}\right)  ^{1/p}-a\mu\left(
\mathcal{T}\right)  >0$ and
\[
b^{q}\mu\left(  K\right)  =\dfrac{\left(  \mu\left(  \mathcal{T}\right)
^{1/p}-a\mu\left(  \mathcal{T}\backslash K\right)  \right)  ^{q}}{\mu\left(
K\right)  ^{q-1}}>\varepsilon^{q}%
\]
for small $\mu\left(  K\right)  >0$. Put $f=a\left[  \mathcal{T}\backslash
K\right]  +b\left[  K\right]  \in L^{q}\left(  \mathcal{T}\right)  _{+}$. Then
$\left\Vert f\right\Vert _{q}^{q}=\int f^{q}=a^{q}\mu\left(  \mathcal{T}%
\backslash K\right)  +b^{q}\mu\left(  K\right)  >\varepsilon^{q}$ and $\int
f=a\mu\left(  \mathcal{T}\backslash K\right)  +b\mu\left(  K\right)
=\mu\left(  \mathcal{T}\right)  ^{1/p}$, which means that $f\in S^{p}%
\backslash S_{\varepsilon}^{p}$, a contradiction. If $p=1$ we fix
$b>\varepsilon$ and $K\subseteq\mathcal{T}$ with $0<\mu\left(  K\right)
/\mu\left(  \mathcal{T}\right)  <1/b$. The rational function $y=\dfrac
{1-x}{b-x}$, $0\leq x\leq1$ is a continuous decreasing function, therefore
$\mu\left(  K\right)  /\mu\left(  \mathcal{T}\right)  =\dfrac{1-a}{b-a}$ for
some $0<a<1$. Put $f=a\left[  \mathcal{T}\backslash K\right]  +b\left[
K\right]  \in L^{\infty}\left(  \mathcal{T}\right)  _{+}$ with $\left\Vert
f\right\Vert _{\infty}=a\vee b=b>\varepsilon$. Since $\int f=b\mu\left(
\mathcal{T}\backslash K\right)  +a\mu\left(  K\right)  =\mu\left(
\mathcal{T}\right)  $, it follows that $f\in S^{1}\backslash S_{\varepsilon
}^{1}$, a contradiction.
\end{proof}

As we have noticed above the equality $\ell^{p}\left(  c\right)  =\ell
^{p}\left(  c\right)  _{\varepsilon}^{+}$ holds only in the case of dimension
$2$. The same can be observed in the noncommutative case too that we are going
to tackle with below.

\section{Projective positivity in $L^{p}$-spaces of a finite von Neumann
algebra\label{secVNM}}

In this section we investigate the projective cones in $L^{p}$-spaces of a
finite von Neumann algebra.

\subsection{Projective positivity in matrix algebras}

We do analysis of the full matrix algebra case. Consider the $\ast$-vector
space $\mathbb{M}_{n}$ of all complex $n\times n$-matrices for $n>1$. The
identity matrix in $\mathbb{M}_{n}$ is denoted by $e$. For every $1\leq
p\leq\infty$ we define the Schatten $p$-norm on $\mathbb{M}_{n}$ by
$\left\Vert x\right\Vert _{p}=\tau\left(  \left\vert x\right\vert ^{p}\right)
^{1/p}$, $x\in\mathbb{M}_{n}$, where $\tau$ is the standard trace on
$\mathbb{M}_{n}$. The normed vector space $\left(  \mathbb{M}_{n},\left\Vert
\cdot\right\Vert _{p}\right)  $ is denoted by $\mathfrak{S}_{p}$. For the
extremal values $p=1,\infty$ we also use the notations $T_{n}=\mathfrak{S}%
_{1}$ and $M_{n}=\mathfrak{S}_{\infty}$. Recall that $\mathfrak{S}_{p}$ and
$\mathfrak{S}_{q}$ are in the canonical duality through the pairing
$\left\langle \cdot,\cdot\right\rangle :\mathfrak{S}_{p}\times\mathfrak{S}%
_{q}\rightarrow\mathbb{C}$, $\left\langle x,y\right\rangle =\tau\left(
xy\right)  $ whenever $\left(  p,q\right)  $ is a conjugate pair. Thus
$\mathfrak{S}_{p}^{\ast}=\mathfrak{S}_{q}$ up to an isometric identification
given by $\tau$. In particular, $M_{n}^{\ast}=T_{n}$, and $M_{n}^{+}$ is the
standard unital cone of positive elements of the $C^{\ast}$-algebra $M_{n}$
with its state space
\[
S_{\tau}=\left\{  y\in T_{n}:y\geq0,\tau\left(  y\right)  =1\right\}
\]
to be the convex set of all density matrices from $\mathbb{M}_{n}$. Thus
\begin{equation}
S\left(  M_{n}^{+}\right)  =S_{\tau}=\operatorname{co}\left\{  \zeta\odot
\zeta:\left\Vert \zeta\right\Vert =1\right\}  \label{ssco}%
\end{equation}
is the convex hull of the one-rank projections $\left(  \zeta\odot
\zeta\right)  \left(  \eta\right)  =\left(  \eta,\zeta\right)  \zeta$,
$\zeta,\eta\in\ell^{2}\left(  n\right)  $. Along with the cone $M_{n}^{+}$ one
can generate its deformed versions $M_{n,\varepsilon}^{+}$, $\varepsilon\geq1$
considered above in Section \ref{SecDPOS}. Namely, $x\in M_{n,\varepsilon}%
^{+}$ iff $x\geq\left\Vert x\right\Vert \dfrac{\varepsilon-1}{\varepsilon+1}e$
in $M_{n}$. Thus $x\in M_{n,\varepsilon}^{+}$ means that $x\geq0$ and
$\wedge\sigma\left(  x\right)  \geq\dfrac{\varepsilon-1}{\varepsilon+1}\left(
\vee\sigma\left(  x\right)  \right)  $, where $\sigma\left(  x\right)  $
denotes the spectrum of $x$. In particular, every nonzero $x\in
M_{n,\varepsilon}^{+}$ is invertible automatically. By Lemma \ref{lemEp1}, the
relation
\[
S\left(  M_{n,\varepsilon}^{+}\right)  =\left\{  \left(  1+r\right)
z-rw:z,w\in S_{\tau},0\leq r\leq2^{-1}\left(  \varepsilon-1\right)  \right\}
\]
provides a detailed description of the state space of the cone
$M_{n,\varepsilon}^{+}$.

The class of all Hilbert-Schmidt matrices is denoted by $HS_{n}$, that is,
$HS_{n}=\mathfrak{S}_{2}$. Recall that $HS_{n}$ turns out to be a Hilbert
$\ast$-space with the inner product $\left(  x,y\right)  =\tau\left(
xy^{\ast}\right)  $, and every $x\in HS_{n}$ admits a unique orthonormal
expansion $x=x_{0}+\left(  x,e_{2}\right)  e_{2}$, $\left(  x_{0}%
,e_{2}\right)  =0$, where $e_{2}=n^{-1/2}e$ is the unit vector of $HS_{n}$.
One can define the following cone
\begin{equation}
HS_{n}^{+}=\left\{  x\in HS_{n,h}:\left\Vert x\right\Vert _{2}\leq\left(
\dfrac{n}{n-1}\right)  ^{1/2}\left(  x,e_{2}\right)  \right\}  \label{HSP}%
\end{equation}
in the unital Hilbert $\ast$-space $\left(  HS_{n},e_{2}\right)  $. This cone
is a deformation of the standard cone of a unital Hilbert $\ast$-space
considered in \cite{Dfaa19}. Actually, the constant $\left(  \dfrac{n}%
{n-1}\right)  ^{1/2}$ can be reduced to $\sqrt{2}$ by making a suitable
transformation of the state space, and $HS_{n}^{+}$ defines a projective
positivity (see below to the proof of the forthcoming assertion).

\begin{theorem}
\label{thMP1}All $\mathfrak{S}_{p}$ are unital $\ast$-normed spaces with their
common unit $e$. Thus $\mathfrak{S}_{p}$ is a function system equipped with
the cone $\mathfrak{S}_{p}^{+}$ given by $\varepsilon=1>1/\left\Vert
e\right\Vert _{p}$ such that%
\[
M_{n,\varepsilon}^{+}\subseteq\mathfrak{S}_{1}^{+}\subseteq\mathfrak{S}%
_{p}^{+}\subseteq\mathfrak{S}_{\infty}^{+}=M_{n}^{+},\quad\varepsilon\geq n
\]
is an increasing family of the unital cones with $S\left(  M_{n}^{+}\right)
=S_{\tau}$ and
\[
S\left(  \mathfrak{S}_{p}^{+}\right)  =\left\{  \left(  1+r\right)
z-rw:z,w\in S_{\tau},zw=0,\left\Vert \left(  1+r\right)  z+rw\right\Vert
_{q}\leq1\right\}
\]
for all conjugate couples $\left(  p,q\right)  $. Moreover, $\mathfrak{S}%
_{2}^{+}=HS_{n}^{+}$.
\end{theorem}

\begin{proof}
Take $x\in\mathfrak{S}_{p}$ with its polar decomposition $x=u\left\vert
x\right\vert $. Taking into account that $\left\vert x^{\ast}\right\vert
=u\left\vert x\right\vert u^{\ast}$, we derive that $\left\Vert x^{\ast
}\right\Vert _{p}=\tau\left(  \left\vert x^{\ast}\right\vert ^{p}\right)
^{1/p}=\tau\left(  u\left\vert x\right\vert ^{p}u^{\ast}\right)  ^{1/p}%
=\tau\left(  \left\vert x\right\vert ^{p}\right)  ^{1/p}=\left\Vert
x\right\Vert _{p}$, which means that $\mathfrak{S}_{p}$ is a $\ast$-normed
space. Prove that $e$ is a unit for $\mathfrak{S}_{p}$ (see Subsection
\ref{subsecUNS1}). Notice that $\mathfrak{S}_{p,he}^{\ast}=\left\{
y\in\mathfrak{S}_{q,h}:\tau\left(  y\right)  =1\right\}  $,
$\operatorname{ball}\mathfrak{S}_{p}^{\ast}=\operatorname{ball}\mathfrak{S}%
_{q}$, and $S_{1}^{p}=\mathfrak{S}_{p,he}^{\ast}\cap\operatorname{ball}%
\mathfrak{S}_{q}$ is the related convex set. One needs to find out a
description for $S_{1}^{p}$ and show that $S_{1}^{p}\neq\varnothing$.

The real vector space $\mathbb{R}^{n}$ equipped with the $q$-norm $\left\Vert
\mathbf{x}\right\Vert _{q}=\left(  \sum_{i=1}^{n}\left\vert x_{i}\right\vert
^{q}\right)  ^{1/q}$ is denoted by $\mathbb{R}_{q}^{n}$, that is,
$\mathbb{R}_{q}^{n}=\ell^{q}\left(  n\right)  _{h}$. Put $\Pi_{n}=\left\{
\mathbf{x}\in\mathbb{R}^{n}:\sum_{i=1}^{n}x_{i}=1\right\}  $ to be the
hyperplane in $\mathbb{R}^{n}$ with its normal vector $\mathbf{e=}\left(
1,\ldots,1\right)  $. Notice that $\Pi_{n}\cap\operatorname{ball}%
\mathbb{R}_{1}^{n}=S_{+}$, where $S_{+}=\left\{  \mathbf{x}\in\mathbb{R}%
^{n}:x_{i}\geq0,\sum_{i=1}^{n}x_{i}=1\right\}  $ is the positive surface of
the tetrahedron in $\mathbb{R}^{n}$. Moreover,%
\[
S_{+}\subseteq\Pi_{n}\cap\operatorname{ball}\mathbb{R}_{q}^{n}\subseteq\Pi
_{n}\cap\operatorname{ball}\mathbb{R}_{\infty}^{n}%
\]
is an increasing family of the hyperplane regions from $\Pi_{n}$. Pick
$\mathbf{x}\in\Pi_{n}\cap\operatorname{ball}\mathbb{R}_{q}^{n}$ and put
$x=x_{1}\oplus\cdots\oplus x_{n}\in\mathbb{M}_{n}$ to be the related diagonal
matrix. Then $\tau\left(  x\right)  =1$ and $\left\Vert x\right\Vert
_{q}=\left\Vert \mathbf{x}\right\Vert _{q}\leq1$, that is, $x\in
\mathfrak{S}_{p,he}^{\ast}\cap\operatorname{ball}\mathfrak{S}_{q}=S_{1}^{p}$.
Thus $\Pi_{n}\cap\operatorname{ball}\mathbb{R}_{q}^{n}\subseteq S_{1}^{p}$ up
to a canonical identification. Conversely, if $y\in S_{1}^{p}$ then there is a
unitary $u\in\mathbb{M}_{n}$ such that $y=u^{\ast}xu$ for a diagonal matrix
$x=x_{1}\oplus\cdots\oplus x_{n}$ with $\sigma\left(  y\right)  =\left\{
x_{i}\right\}  \subseteq\mathbb{R}$ or $\mathbf{x}\in\mathbb{R}^{n}$. But
$\left\vert y\right\vert ^{q}=u^{\ast}\left\vert x\right\vert ^{q}u$ with
$\left\vert x\right\vert =\left\vert x_{1}\right\vert \oplus\cdots
\oplus\left\vert x_{n}\right\vert $, and%
\begin{align*}
\left\Vert \mathbf{x}\right\Vert _{q}  &  =\left(  \sum_{i=1}^{n}\left\vert
x_{i}\right\vert ^{q}\right)  ^{1/q}=\tau\left(  \left\vert x\right\vert
^{q}\right)  ^{1/q}=\tau\left(  u^{\ast}\left\vert x\right\vert ^{q}u\right)
^{1/q}=\tau\left(  \left\vert y\right\vert ^{q}\right)  ^{1/q}=\left\Vert
y\right\Vert _{q}\leq1,\\
\sum_{i=1}^{n}x_{i}  &  =\tau\left(  x\right)  =\tau\left(  u^{\ast}xu\right)
=\tau\left(  y\right)  =1,
\end{align*}
which means that $\mathbf{x\in}\Pi_{n}\cap\operatorname{ball}\mathbb{R}%
_{q}^{n}$. Consequently, we obtain that
\begin{equation}
S_{1}^{p}=\cup\left\{  u^{\ast}\left(  \Pi_{n}\cap\operatorname{ball}%
\mathbb{R}_{q}^{n}\right)  u:u\in\mathbb{M}_{n},u^{\ast}u=e\right\}  .
\label{sps}%
\end{equation}
In particular, $S_{1}^{p}\neq\varnothing$ and $e$ turns out to be a unit for
every $\mathfrak{S}_{p}$ with $\left\Vert e\right\Vert _{p}=n^{1/p}>1$.

Further, put $\mathfrak{S}_{p}^{+}$ to be the set of those $x\in
\mathfrak{S}_{p,h}$ such that $\left\langle x,S_{1}^{p}\right\rangle
=\tau\left(  xS_{1}^{p}\right)  \geq0$. If $p=\infty$ then $\left\Vert
e\right\Vert _{\infty}=1$ and
\begin{align*}
S_{1}^{\infty}  &  =\cup\left\{  u^{\ast}\left(  \Pi_{n}\cap
\operatorname{ball}\mathbb{R}_{1}^{n}\right)  u:u\in\mathbb{M}_{n},u^{\ast
}u=e\right\}  =\cup\left\{  u^{\ast}S_{+}u:u\in\mathbb{M}_{n},u^{\ast
}u=e\right\} \\
&  =\left\{  y\in T_{n}:y\geq0,\tau\left(  y\right)  =1\right\}  =S_{\tau}.
\end{align*}
In this special case, we have $\mathfrak{S}_{\infty}^{+}=M_{n}^{+}$. Indeed,
$x\in\mathfrak{S}_{\infty}^{+}$ iff $\left(  x\zeta,\zeta\right)  =\tau\left(
x\left(  \zeta\odot\zeta\right)  \right)  \geq0$ for every unit vector
$\zeta\in\ell^{2}\left(  n\right)  $ (see (\ref{ssco})). If $p<\infty$ then
$\varepsilon=1>1/\left\Vert e\right\Vert _{p}$, and $\mathfrak{S}_{p}^{+}$
turns out to be a separated, closed, unital, cone in $\mathfrak{S}_{p}$ thanks
to Theorem \ref{thFS1}. Thus $\left(  \mathfrak{S}_{p},\mathfrak{S}_{p}%
^{+}\right)  $ is a function system on $S_{1}^{p}$ such that $S\left(
\mathfrak{S}_{p}^{+}\right)  =S_{1}^{p}$. In particular, for $p_{1}<p_{2}$ we
obtain that $q_{2}<q_{1}$, $S_{1}^{p_{2}}\subseteq S_{1}^{p_{1}}$ (see
(\ref{sps})), which in turn implies that $\mathfrak{S}_{p_{1}}^{+}%
\subseteq\mathfrak{S}_{p_{2}}^{+}$.

Now let us inspect the cone $\mathfrak{S}_{2}^{+}$. In this case, the norm
dual space $HS_{n}^{\ast}$ is canonically identified with the conjugate space
$\overline{HS_{n}}$, and for every $y\in\mathfrak{S}_{2,he}^{\ast}$ we have
the orthogonal expansion $y=y_{0}+\left(  1/n\right)  e$ with $y_{0}\perp e$
(or $\tau\left(  y_{0}\right)  =0$). Then $\left\Vert y\right\Vert _{2}%
^{2}=\left\Vert y_{0}\right\Vert _{2}^{2}+1/n$, and $y\in S_{1}^{2}$ iff
$\left\Vert y_{0}\right\Vert _{2}\leq\left(  \left(  n-1\right)  /n\right)
^{1/2}$. Hence
\[
S_{1}^{2}=\left(  \left(  n-1\right)  /n\right)  ^{1/2}\operatorname{ball}%
HS_{n,h}^{e}+\left(  1/n\right)  e=n^{-1/2}\left(  \left(  n-1\right)
^{1/2}\operatorname{ball}HS_{n,h}^{e}+e_{2}\right)  ,
\]
where $HS_{n,h}^{e}=\left\{  w\in HS_{n,h}:\tau\left(  w\right)  =0\right\}  $
is the real vector subspace of $\overline{HS_{n}}$ orthogonal to $e$, and
$e_{2}=n^{-1/2}e$ is the unit vector of $HS_{n}$. Thus $x\in\mathfrak{S}%
_{2}^{+}$ with its orthonormal expansion $x=x_{0}+\left(  x,e_{2}\right)
e_{2}$ iff $\left(  x,\left(  n-1\right)  ^{1/2}\operatorname{ball}%
HS_{n,h}^{e}+e_{2}\right)  \geq0$, which in turn is equivalent to
\[
\left(  n-1\right)  ^{1/2}\left\Vert x_{0}\right\Vert _{2}=\left(  n-1\right)
^{1/2}\vee\left\vert \left(  x_{0},\operatorname{ball}HS_{n,h}^{e}\right)
\right\vert \leq\left(  x,e_{2}\right)  \text{ or }\left\Vert x_{0}\right\Vert
_{2}\leq\left(  n-1\right)  ^{-1/2}\left(  x,e_{2}\right)  .
\]
Consequently, $\mathfrak{S}_{2}^{+}$ is reduced to the cone $HS_{n}^{+}$ from
(\ref{HSP}). If we choose $S_{\varepsilon}^{2}$ for $\varepsilon=\sqrt
{2/n}>1/\sqrt{n}$ (notice that $\left\Vert e\right\Vert _{2}=\sqrt{n}$)
instead of $S_{1}^{2}$, then based on Theorem \ref{thFS1}, the projective
positivity will exactly be reduced to one from \cite{Dfaa19}. Namely,
$y=y_{0}+\left(  1/n\right)  e\in S_{\varepsilon}^{2}$ iff $\left\Vert
y\right\Vert _{2}^{2}=\left\Vert y_{0}\right\Vert _{2}^{2}+1/n\leq
\varepsilon^{2}=2/n$ or $\left\Vert y_{0}\right\Vert _{2}\leq1/\sqrt{n}$. Thus
$S_{\varepsilon}^{2}=n^{-1/2}\left(  \operatorname{ball}HS_{n,h}^{e}%
+e_{2}\right)  $, and for $x\in HS_{n}$ we have $\left(  x,S_{\varepsilon}%
^{2}\right)  \geq0$ iff $\left\Vert x\right\Vert _{2}\leq\sqrt{2}\left(
x,e_{2}\right)  $.

Finally, let us prove the equality for the state spaces $S\left(
\mathfrak{S}_{p}^{+}\right)  $. Suppose that a hermitian matrix $y$ admits the
expansion $y=\left(  1+r\right)  z-rw$ with $z,w\in S_{\tau}$, $zw=0$ and
$\left\Vert z+r\left(  z+w\right)  \right\Vert _{q}\leq1$. Then $\tau\left(
y\right)  =\left(  1+r\right)  \tau\left(  z\right)  -r\tau\left(  w\right)
=1+r-r=1$, $\left\vert y\right\vert =\left(  1+r\right)  z+rw$ and $\left\Vert
y\right\Vert _{q}=\left\Vert \left\vert y\right\vert \right\Vert
_{q}=\left\Vert \left(  1+r\right)  z+rw\right\Vert _{q}\leq1$, that is,
$y\in\mathfrak{S}_{p,he}^{\ast}\cap\operatorname{ball}\mathfrak{S}_{q}%
=S_{1}^{p}=S\left(  \mathfrak{S}_{p}^{+}\right)  $. Conversely, take $y\in
S_{1}^{p}$. Using the spectral decomposition of $y$, we deduce that
$y=y_{+}-y_{-}=\tau\left(  y_{+}\right)  z-\tau\left(  y_{-}\right)  w$ with
$z,w\in S_{\tau}$ and $\tau\left(  y_{+}\right)  -\tau\left(  y_{-}\right)
=1$. If $r=\tau\left(  y_{-}\right)  $ then $y=\left(  1+r\right)  z-rw$ with
$z,w\in S_{\tau}$, $zw=0$ and $\left\Vert \left(  1+r\right)  z+rw\right\Vert
_{q}=\left\Vert \left\vert y\right\vert \right\Vert _{q}=\left\Vert
y\right\Vert _{q}\leq1$. Hence $S\left(  \mathfrak{S}_{p}^{+}\right)  $
consists of those $\left(  1+r\right)  z-rw$ with $z,w\in S_{\tau}$, $zw=0$,
and $\left\Vert \left(  1+r\right)  z+rw\right\Vert _{q}\leq1$. In particular,
$S\left(  \mathfrak{S}_{1}^{+}\right)  $ consists of those $y=\left(
1+r\right)  z-rw$ such that $z,w\in S_{\tau}$, $zw=0$ and $\left\Vert \left(
1+r\right)  z+rw\right\Vert _{\infty}\leq1$.

Pick $y\in S\left(  \mathfrak{S}_{1}^{+}\right)  $ with $\left(  1+r\right)
z+rw=\left\vert y\right\vert \leq e$. Then $\left\vert y\right\vert ^{p}\leq
e$, $p\geq1$, and $1+2r=\left(  1+r\right)  \tau\left(  z\right)
+r\tau\left(  w\right)  =\tau\left(  \left\vert y\right\vert \right)  \leq
\tau\left(  e\right)  =n$ (or $r\leq\left(  n-1\right)  /2$). Hence
\[
S\left(  \mathfrak{S}_{p}^{+}\right)  \subseteq\left\{  \left(  1+r\right)
z-rw:z,w\in S\left(  M_{n}^{+}\right)  ,0\leq r\leq\left(  n-1\right)
/2\right\}  =S\left(  M_{n,n}^{+}\right)  \subseteq S\left(  M_{n,\varepsilon
}^{+}\right)
\]
for all $p\geq1$ and $\varepsilon\geq n$ by virtue of Lemma \ref{lemEp1}. In
particular, $M_{n,\varepsilon}^{+}\subseteq\mathfrak{S}_{1}^{+}\subseteq
\mathfrak{S}_{p}^{+}\subseteq M_{n}^{+}$ is an increasing family of the unital cones.
\end{proof}

We also write $T_{n}^{+}$ instead of $\mathfrak{S}_{1}^{+}$. Based on Theorem
\ref{thMP1}, we have
\begin{align*}
S_{\tau}  &  \subseteq S\left(  T_{n}^{+}\right)  =\left\{  \left(
1+r\right)  z-rw:z,w\in S_{\tau},zw=0,\left\Vert \left(  1+r\right)
z+rw\right\Vert _{\infty}\leq1\right\} \\
&  \subseteq\left\{  \left(  1+r\right)  z-rw:z,w\in S_{\tau},0\leq
r\leq\left(  n-1\right)  /2\right\}  .
\end{align*}
Moreover, using (\ref{sps}), we obtain that
\[
S\left(  T_{n}^{+}\right)  =S_{1}^{1}=\cup\left\{  u^{\ast}\left(  \Pi_{n}%
\cap\operatorname{ball}\mathbb{R}_{\infty}^{n}\right)  u:u\in\mathbb{M}%
_{n},u^{\ast}u=e\right\}  ,
\]
which provides a geometric interpretation of the state space $S\left(
T_{n}^{+}\right)  $. Notice that the equality $S_{\tau}=S\left(  T_{n}%
^{+}\right)  $ is true only if $n=2$. For $n=3$ the matrices $\left(
-1\right)  \oplus1\oplus1$, $1\oplus\left(  -1\right)  \oplus1$ and
$1\oplus1\oplus\left(  -1\right)  $ are some (extreme) elements of $S\left(
T_{n}^{+}\right)  $ out of $S_{\tau}$.

Thus $T_{n}^{+}=\left\{  x\in T_{n,h}:\tau\left(  xS\left(  T_{n}^{+}\right)
\right)  \geq0\right\}  $ and $M_{n,\varepsilon}^{+}\subseteq T_{n}%
^{+}\subseteq M_{n}^{+}$ for all $\varepsilon\geq n$. For $n=3$ and
$x=2\oplus0\oplus1$, we have $x\in M_{3}^{+}\backslash T_{3}^{+}$, for
$\tau\left(  x\left(  \left(  -1\right)  \oplus1\oplus1\right)  \right)
=-1<0$. Thus the eigenvalues of a projective-positive matrix from $T_{3}^{+}$
must be strictly positive real numbers.

\subsection{Projective positivity in finite von Neumann algebras}

Now fix a finite von Neumann algebra $\mathcal{M}$ with its unit $e$ and a
faithful, finite, normal trace $\tau$. We can always assume that $\tau\left(
e\right)  >1$. The noncommutative $L^{p}$-space $L^{p}\left(  \mathcal{M}%
,\tau\right)  $ of $\mathcal{M}$ is given as the completion of $\mathcal{M}$
with respect to the norm $\left\Vert x\right\Vert _{p}=\tau\left(  \left\vert
x\right\vert ^{p}\right)  ^{1/p}$, $x\in\mathcal{M}$, $1\leq p\leq\infty$. As
above $L^{p}\left(  \mathcal{M},\tau\right)  $ is a unital $\ast$-vector space
with $\left\Vert e\right\Vert _{p}=\tau\left(  e\right)  ^{1/p}>1$, and
$S_{\tau}=\left\{  y\in\mathcal{M}_{+}:\tau\left(  y\right)  =1\right\}  $
denotes the convex set of all density elements in $\mathcal{M}$. The Banach
space $L^{1}\left(  \mathcal{M},\tau\right)  $ turns out to be the predual of
$\mathcal{M}$, which allows to identify $L^{1}\left(  \mathcal{M},\tau\right)
$ with a closed subspace of $\mathcal{M}^{\ast}$. In this case, $S_{\ast
}\left(  \mathcal{M}_{+}\right)  =S\left(  \mathcal{M}_{+}\right)  \cap
L^{1}\left(  \mathcal{M},\tau\right)  $ is the set of all normal states of the
cone $\mathcal{M}_{+}$, and $S_{\tau}^{-}=S_{\ast}\left(  \mathcal{M}%
_{+}\right)  $ for the norm closure of $S_{\tau}$ in $L^{1}\left(
\mathcal{M},\tau\right)  $ (see \cite[Proposition 7.3]{DSbM}). Moreover, we
have the canonical continuous inclusions
\[
\mathcal{M=}L^{\infty}\left(  \mathcal{M},\tau\right)  \subseteq L^{p}\left(
\mathcal{M},\tau\right)  \subseteq L^{1}\left(  \mathcal{M},\tau\right)
\]
in the descending order. Indeed, for $1\leq p<l<\infty$ and $r=l/p>1$ we have%
\[
\left\Vert y\right\Vert _{p}^{p}=\tau\left(  \left\vert y\right\vert
^{l/r}e\right)  \leq\tau\left(  \left\vert y\right\vert ^{rl/r}\right)
^{1/r}\tau\left(  e\right)  ^{1-1/r}=\left\Vert y\right\Vert _{l}^{p}%
\tau\left(  e\right)  ^{1-p/l}\text{ or }\left\Vert y\right\Vert _{p}%
\leq\left\Vert y\right\Vert _{l}\tau\left(  e\right)  ^{\left(  l-p\right)
/lp}%
\]
for all $y\in\mathcal{M}$. It means that $L^{l}\left(  \mathcal{M}%
,\tau\right)  \subseteq L^{p}\left(  \mathcal{M},\tau\right)  $ is a bounded
embedding with the norm $\tau\left(  e\right)  ^{\left(  l-p\right)  /lp}$ for
all $1\leq p<l<\infty$. If $l=\infty$ then $\left\Vert y\right\Vert _{p}%
=\tau\left(  \left\vert y\right\vert ^{p}e\right)  ^{1/p}\leq\left\Vert
y\right\Vert \tau\left(  e\right)  ^{1/p}$ for all $y\in\mathcal{M}$. Thus
\begin{equation}
\operatorname{ball}L^{l}\left(  \mathcal{M},\tau\right)  \subseteq\tau\left(
e\right)  ^{\left(  l-p\right)  /lp}\operatorname{ball}L^{p}\left(
\mathcal{M},\tau\right)  \label{iball}%
\end{equation}
for all $1\leq p<l\leq\infty$. As above $L^{2}\left(  \mathcal{M},\tau\right)
$ is a Hilbert $\ast$-space with the inner product $\left(  x,y\right)
=\tau\left(  xy^{\ast}\right)  $, and every $x\in L^{2}\left(  \mathcal{M}%
,\tau\right)  $ admits a unique orthonormal expansion $x=x_{0}+\left(
x,e_{2}\right)  e_{2}$, $\left(  x_{0},e_{2}\right)  =0$, where $e_{2}%
=\tau\left(  e\right)  ^{-1/2}e$. The cone similar to (\ref{HSP}) is defined
in the following way
\[
\mathfrak{c}_{\tau}=\left\{  x\in L^{2}\left(  \mathcal{M},\tau\right)
_{h}:\left\Vert x\right\Vert _{2}\leq\left(  \dfrac{\tau\left(  e\right)
}{\tau\left(  e\right)  -1}\right)  ^{1/2}\left(  x,e_{2}\right)  \right\}
\]
in the unital Hilbert $\ast$-space $\left(  L^{2}\left(  \mathcal{M}%
,\tau\right)  ,e_{2}\right)  $.

\begin{proposition}
\label{propFVNA1}All $L^{p}\left(  \mathcal{M},\tau\right)  $ are unital
$\ast$-normed spaces with their common unit $e$. Thus $L^{p}\left(
\mathcal{M},\tau\right)  $ is a function system equipped with the cone
$L^{p}\left(  \mathcal{M},\tau\right)  _{+}$ given by $\varepsilon
=1>1/\left\Vert e\right\Vert _{p}$ such that $L^{\infty}\left(  \mathcal{M}%
,\tau\right)  _{+}=\mathcal{M}_{+}$, $L^{2}\left(  \mathcal{M},\tau\right)
_{+}=\mathfrak{c}_{\tau}$, and $\mathcal{M}_{\varepsilon}^{+}\subseteq
L^{1}\left(  \mathcal{M},\tau\right)  _{+}\cap\mathcal{M}$ for all
$\varepsilon\geq\tau\left(  e\right)  $. Moreover,
\[
S\left(  L^{1}\left(  \mathcal{M},\tau\right)  _{+}\right)  =\left\{  \left(
1+r\right)  z-rw:z,w\in S_{\tau},zw=0,\left\Vert \left(  1+r\right)
z+rw\right\Vert \leq1\right\}  .
\]

\end{proposition}

\begin{proof}
We use the arguments similar to those from the proof of Theorem \ref{thMP1}.
Take $x\in\mathcal{M}$ with its polar decomposition $x=u\left\vert
x\right\vert $. Since $\left\vert x^{\ast}\right\vert =u\left\vert
x\right\vert u^{\ast}$, it follows that $\left\Vert x^{\ast}\right\Vert
_{p}=\tau\left(  \left\vert x^{\ast}\right\vert ^{p}\right)  ^{1/p}%
=\tau\left(  u\left\vert x\right\vert ^{p}u^{\ast}\right)  ^{1/p}=\tau\left(
\left\vert x\right\vert ^{p}\right)  ^{1/p}=\left\Vert x\right\Vert _{p}$,
which in turn implies that $L^{p}\left(  \mathcal{M},\tau\right)  $ is a
$\ast$-normed space. Prove that $e$ is a unit for $L^{p}\left(  \mathcal{M}%
,\tau\right)  $. For $p<\infty$ we have $L^{p}\left(  \mathcal{M},\tau\right)
_{he}^{\ast}=\left\{  y\in L^{q}\left(  \mathcal{M},\tau\right)  _{h}%
:\tau\left(  y\right)  =1\right\}  $ and $\operatorname{ball}L^{p}\left(
\mathcal{M},\tau\right)  ^{\ast}=\operatorname{ball}L^{q}\left(
\mathcal{M},\tau\right)  $ with the conjugate couple $\left(  p,q\right)  $.
Moreover, based on the duality between $L^{q}\left(  \mathcal{M},\tau\right)
$ and $L^{q}\left(  \mathcal{M},\tau\right)  $, we deduce that $\tau$ defines
a bounded linear functional on $L^{q}\left(  \mathcal{M},\tau\right)  $ and
\[
1<\tau\left(  e\right)  ^{1/p}=\tau\left(  \tau\left(  e\right)
^{-1/q}e\right)  \leq\vee\tau\left(  \operatorname{ball}L^{q}\left(
\mathcal{M},\tau\right)  _{h}\right)  \leq\vee\left\vert \tau\left(
\operatorname{ball}L^{q}\left(  \mathcal{M},\tau\right)  \right)  \right\vert
\leq\tau\left(  e\right)  ^{1/p},
\]
that is, $1<\tau\left(  e\right)  ^{1/p}=\vee\tau\left(  \operatorname{ball}%
L^{q}\left(  \mathcal{M},\tau\right)  _{h}\right)  $. But $\tau\left(
\operatorname{ball}L^{q}\left(  \mathcal{M},\tau\right)  _{h}\right)
\subseteq\mathbb{R}$ is a convex subset containing the origin. It follows that
$\tau\left(  y\right)  =1$ for a certain $y\in\operatorname{ball}L^{q}\left(
\mathcal{M},\tau\right)  _{h}$, that is, $S_{1}^{p}=L^{p}\left(
\mathcal{M},\tau\right)  _{he}^{\ast}\cap\operatorname{ball}L^{q}\left(
\mathcal{M},\tau\right)  \neq\varnothing$ whenever $p<\infty$. In particular,%
\[
S_{1}^{1}=\mathcal{M}_{he}\cap\operatorname{ball}\mathcal{M}=\left\{
y\in\operatorname{ball}\mathcal{M}_{h}:\tau\left(  y\right)  =1\right\}
\neq\varnothing.
\]
For $p=\infty$ we have $\left\Vert e\right\Vert _{\infty}=1$ and
$S_{1}^{\infty}=\mathcal{M}_{he}^{\ast}\cap\operatorname{ball}\mathcal{M}%
^{\ast}$, that is, every $\varphi\in S_{1}^{\infty}$ defines a unital
contraction $\varphi:\mathcal{M}\rightarrow\mathbb{C}$ with $1=\left\langle
e,\varphi\right\rangle \leq\vee\left\vert \left\langle \operatorname{ball}%
\mathcal{M},\varphi\right\rangle \right\vert =\left\Vert \varphi\right\Vert
\leq1$. By Lemma \ref{lemOS0}, $\varphi$ is a positive functional. Thus
$S_{1}^{\infty}=S\left(  \mathcal{M}^{+}\right)  $ is the state space of the
cone $\mathcal{M}^{+}$. Hence $e$ is a unit for every $L^{p}\left(
\mathcal{M},\tau\right)  $, $1\leq p\leq\infty$.

Now put $L^{p}\left(  \mathcal{M},\tau\right)  _{+}$ to be the set of those
$x\in L^{p}\left(  \mathcal{M},\tau\right)  _{h}$ such that $\left\langle
x,S_{1}^{p}\right\rangle \geq0$. If $p=\infty$ then $L^{\infty}\left(
\mathcal{M},\tau\right)  _{+}=\left\{  x\in\mathcal{M}_{h}:\left\langle
x,S\left(  \mathcal{M}_{+}\right)  \right\rangle \geq0\right\}  =\mathcal{M}%
_{+}$ (see \cite[7.1]{DSbM}). If $p<\infty$ then $\varepsilon=1>1/\left\Vert
e\right\Vert _{p}$ defines the separated, closed, unital, cone $L^{p}\left(
\mathcal{M},\tau\right)  _{+}$ in $L^{p}\left(  \mathcal{M},\tau\right)  $,
$S\left(  L^{p}\left(  \mathcal{M},\tau\right)  _{+}\right)  =S_{1}^{p}$ and
$\left(  L^{p}\left(  \mathcal{M},\tau\right)  ,L^{p}\left(  \mathcal{M}%
,\tau\right)  _{+}\right)  $ is a function system on $S_{1}^{p}$ thanks to
Theorem \ref{thFS1}.

If $p=2$ then the norm dual space $L^{2}\left(  \mathcal{M},\tau\right)
^{\ast}$ is canonically identified with the conjugate space $\overline
{L^{2}\left(  \mathcal{M},\tau\right)  }$, and for every $y\in L^{2}\left(
\mathcal{M},\tau\right)  _{he}^{\ast}$ we have the orthogonal expansion
$y=y_{0}+\tau\left(  e\right)  ^{-1}e$ with $y_{0}\perp e$ (or $\tau\left(
y\right)  =0$). Then $\left\Vert y\right\Vert _{2}^{2}=\left\Vert
y_{0}\right\Vert _{2}^{2}+\tau\left(  e\right)  ^{-1}$, and $y\in S_{1}^{2}$
iff $\left\Vert y_{0}\right\Vert _{2}\leq\left(  \left(  \tau\left(  e\right)
-1\right)  /\tau\left(  e\right)  \right)  ^{1/2}$. Hence
\begin{align*}
S_{1}^{2}  &  =\left(  \left(  \tau\left(  e\right)  -1\right)  /\tau\left(
e\right)  \right)  ^{1/2}\operatorname{ball}L^{2}\left(  \mathcal{M}%
,\tau\right)  _{h}^{e}+\tau\left(  e\right)  ^{-1}e\\
&  =\tau\left(  e\right)  ^{-1/2}\left(  \left(  \tau\left(  e\right)
-1\right)  ^{1/2}\operatorname{ball}L^{2}\left(  \mathcal{M},\tau\right)
_{h}^{e}+e_{2}\right)  ,
\end{align*}
where $L^{2}\left(  \mathcal{M},\tau\right)  _{h}^{e}=\left\{  w\in
L^{2}\left(  \mathcal{M},\tau\right)  _{h}:\tau\left(  w\right)  =0\right\}  $
is the real vector subspace of $\overline{L^{2}\left(  \mathcal{M}%
,\tau\right)  }$ orthogonal to $e$, and $e_{2}=\tau\left(  e\right)  ^{-1/2}e$
is the unit vector of $L^{2}\left(  \mathcal{M},\tau\right)  $. Thus $x\in
L^{2}\left(  \mathcal{M},\tau\right)  _{+}$ with its orthonormal expansion
$x=x_{0}+\left(  x,e_{2}\right)  e_{2}$ iff $\left(  x,\left(  \tau\left(
e\right)  -1\right)  ^{1/2}\operatorname{ball}L^{2}\left(  \mathcal{M}%
,\tau\right)  _{h}^{e}+e_{2}\right)  \geq0$, which is equivalent to%
\begin{align*}
\left(  \tau\left(  e\right)  -1\right)  ^{1/2}\left\Vert x_{0}\right\Vert
_{2}  &  =\left(  \tau\left(  e\right)  -1\right)  ^{1/2}\left(
\vee\left\vert \left(  x_{0},\operatorname{ball}L^{2}\left(  \mathcal{M}%
,\tau\right)  _{h}^{e}\right)  \right\vert \right)  \leq\left(  x,e_{2}\right)
\\
\text{or }\left\Vert x_{0}\right\Vert _{2}  &  \leq\left(  \tau\left(
e\right)  -1\right)  ^{-1/2}\left(  x,e_{2}\right)  .
\end{align*}
Hence $L^{2}\left(  \mathcal{M},\tau\right)  _{+}=\mathfrak{c}_{\tau}$. If we
choose $S_{\varepsilon}^{2}$ for $\varepsilon=\left(  2/\tau\left(  e\right)
\right)  ^{1/2}>1/\sqrt{\tau\left(  e\right)  }$ instead of $S_{1}^{2}$ then
as in the proof of Theorem \ref{thFS1}, the projective positivity will exactly
be reduced to the one from \cite{Dfaa19} of the unital Hilbert $\ast$-space
$\left(  L^{2}\left(  \mathcal{M},\tau\right)  ,e_{2}\right)  $.

Finally, let us prove the equality for the state space $S\left(  L^{1}\left(
\mathcal{M},\tau\right)  _{+}\right)  $, which is a convex subset of
$\mathcal{M}_{h}$. Suppose that $y\in\mathcal{M}_{h}$ admits the expansion
$y=\left(  1+r\right)  z-rw$ with $z,w\in S_{\tau}$, $zw=0$ and $\left\Vert
z+r\left(  z+w\right)  \right\Vert \leq1$. Then $\tau\left(  y\right)
=\left(  1+r\right)  \tau\left(  z\right)  -r\tau\left(  w\right)  =1+r-r=1$,
$\left\vert y\right\vert =\left(  1+r\right)  z+rw$ and $\left\Vert
y\right\Vert =\left\Vert \left\vert y\right\vert \right\Vert =\left\Vert
z+r\left(  z+w\right)  \right\Vert \leq1$, that is, $y\in S_{1}^{1}$.
Conversely, take $y\in S_{1}^{1}$. Using the spectral decomposition of $y$ in
$\mathcal{M}$, we deduce that $y=y_{+}-y_{-}=\tau\left(  y_{+}\right)
z-\tau\left(  y_{-}\right)  w$ with $z,w\in S_{\tau}$ and $\tau\left(
y_{+}\right)  -\tau\left(  y_{-}\right)  =1$. If $r=\tau\left(  y_{-}\right)
$ then $y=\left(  1+r\right)  z-rw$ with $z,w\in S_{\tau}$, $zw=0$ and
$\left\Vert z+r\left(  z+w\right)  \right\Vert =\left\Vert \left\vert
y\right\vert \right\Vert =\left\Vert y\right\Vert \leq1$. Hence
\[
S\left(  L^{1}\left(  \mathcal{M},\tau\right)  _{+}\right)  =\left\{  \left(
1+r\right)  z-rw:z,w\in S_{\tau},zw=0,\left\Vert z+r\left(  z+w\right)
\right\Vert \leq1\right\}  .
\]
Pick $y\in S\left(  L^{1}\left(  \mathcal{M},\tau\right)  _{+}\right)  $ with
$\left(  1+r\right)  z+rw=\left\vert y\right\vert \leq e$. Then $1+2r=\left(
1+r\right)  \tau\left(  z\right)  +r\tau\left(  w\right)  =\tau\left(
\left\vert y\right\vert \right)  \leq\tau\left(  e\right)  $ or $r\leq\left(
\tau\left(  e\right)  -1\right)  /2$. Taking into account the canonical
inclusions $L^{1}\left(  \mathcal{M},\tau\right)  \subseteq L^{1}\left(
\mathcal{M},\tau\right)  ^{\ast\ast}=\mathcal{M}^{\ast}$ and $S_{\tau
}\subseteq S_{\ast}\left(  \mathcal{M}_{+}\right)  \subseteq S\left(
\mathcal{M}_{+}\right)  \subseteq\mathcal{M}^{\ast}$, we deduce that
\begin{align*}
S\left(  L^{1}\left(  \mathcal{M},\tau\right)  _{+}\right)   &  \subseteq
\left\{  \left(  1+r\right)  z-rw:z,w\in S_{\tau},0\leq r\leq\left(
\tau\left(  e\right)  -1\right)  /2\right\} \\
&  \subseteq\left\{  \left(  1+r\right)  \phi-r\psi:\phi,\psi\in S\left(
\mathcal{M}_{+}\right)  ,0\leq r\leq\left(  \tau\left(  e\right)  -1\right)
/2\right\} \\
&  =S\left(  \mathcal{M}_{\tau\left(  e\right)  }^{+}\right)  \subseteq
S\left(  \mathcal{M}_{\varepsilon}^{+}\right)
\end{align*}
for all $\varepsilon\geq\tau\left(  e\right)  $ by virtue of Lemma
\ref{lemEp1}. In particular,
\begin{align*}
\mathcal{M}_{\varepsilon}^{+}  &  =\left\{  x\in\mathcal{M}_{h}:\left\langle
x,S\left(  \mathcal{M}_{\varepsilon}^{+}\right)  \right\rangle \geq0\right\}
\subseteq\left\{  x\in\mathcal{M}_{h}:\left\langle x,S\left(  L^{1}\left(
\mathcal{M},\tau\right)  _{+}\right)  \right\rangle \geq0\right\} \\
&  =\left\{  x\in\mathcal{M}_{h}:\tau\left(  xS\left(  L^{1}\left(
\mathcal{M},\tau\right)  _{+}\right)  \right)  \geq0\right\}  =L^{1}\left(
\mathcal{M},\tau\right)  _{+}\cap\mathcal{M},
\end{align*}
that is, $\mathcal{M}_{\varepsilon}^{+}\subseteq L^{1}\left(  \mathcal{M}%
,\tau\right)  _{+}\cap\mathcal{M}$ for all $\varepsilon\geq\tau\left(
e\right)  $.
\end{proof}

\begin{remark}
Notice that the spectrum of every state $y=\left(  1+r\right)  z-rw\in
S\left(  L^{1}\left(  \mathcal{M},\tau\right)  _{+}\right)  $ can be described
in term of the density elements $z,w\in S_{\tau}$, $zw=0$ and $\left\Vert
z+r\left(  z+w\right)  \right\Vert \leq1$. Since $wz=\left(  zw\right)
^{\ast}=0=zw$, one can use the joint spectral mapping theorem \cite{Tay}.
Namely, $\sigma\left(  y\right)  =\left\{  \left(  1+r\right)  \lambda
-r\mu:\left(  \lambda,\mu\right)  \in\sigma\left(  z,w\right)  \right\}  $,
where $\sigma\left(  z,w\right)  $ is the joint Taylor spectrum of commuting
elements $z$ and $w$, which is reduced to the joint point spectrum \cite{CTg}
in the case of a matrix algebra $\mathcal{M=}M_{n}$ (see Theorem \ref{thMP1}).
But $\left\{  0\right\}  =\sigma\left(  zw\right)  =\left\{  \lambda
\mu:\left(  \lambda,\mu\right)  \in\sigma\left(  z,w\right)  \right\}  $,
therefore for every nonzero $\lambda\in\sigma\left(  z\right)  $ we have
$\left(  \lambda,0\right)  \in\sigma\left(  z,w\right)  $ and $\left(
1+r\right)  \lambda=\left(  1+r\right)  \lambda-r0\in\sigma\left(  y\right)
$. Similarly, $-r\mu\in\sigma\left(  y\right)  $ for every nonzero $\mu
\in\sigma\left(  w\right)  $. Thus $\sigma\left(  y\right)  =\left(
1+r\right)  \sigma\left(  z\right)  \cup\left(  -r\right)  \sigma\left(
w\right)  $ and $\sigma\left(  \left\vert y\right\vert \right)  =\left(
1+r\right)  \sigma\left(  z\right)  \cup r\sigma\left(  w\right)  $.
\end{remark}

\begin{theorem}
\label{thFVNA1}Let $\mathcal{M}$ be a finite von Neumann algebra with its unit
$e$ and faithful, finite, normal trace $\tau$ such that $\tau\left(  e\right)
>1$. If $1\leq p<l\leq\infty$ then $L^{l}\left(  \mathcal{M},\tau\right)
_{\varepsilon}^{+}\subseteq L^{p}\left(  \mathcal{M},\tau\right)  _{+}$ for
all $\varepsilon\geq\tau\left(  e\right)  ^{\left(  l-p\right)  /lp}$. Thus
\[
\mathcal{M}_{\varepsilon}^{+}\subseteq L^{p}\left(  \mathcal{M},\tau\right)
_{r}^{+}\subseteq L^{1}\left(  \mathcal{M},\tau\right)  _{+}\text{,\quad
}\varepsilon\tau\left(  e\right)  ^{-1/p}\geq r\geq\tau\left(  e\right)
^{1/q}.
\]
In particular, $\mathcal{M}_{\varepsilon}^{+}\subseteq L^{p}\left(
\mathcal{M},\tau\right)  _{+}$ whenever $\varepsilon\geq\tau\left(  e\right)
^{1/p}$.
\end{theorem}

\begin{proof}
Fix $1\leq p<l\leq\infty$ and the related conjugate couples $\left(
p,q\right)  $ and $\left(  l,m\right)  $ with $1\leq m<q\leq\infty$. By
Proposition \ref{propFVNA1}, every $L^{p}\left(  \mathcal{M},\tau\right)  $ is
a function system equipped with the cone $L^{p}\left(  \mathcal{M}%
,\tau\right)  _{+}$ (responded to $\varepsilon=1>1/\left\Vert e\right\Vert
_{p}$), and using (\ref{iball}), we deduce that $\operatorname{ball}%
L^{q}\left(  \mathcal{M},\tau\right)  \subseteq\tau\left(  e\right)  ^{\left(
q-m\right)  /qm}\operatorname{ball}L^{m}\left(  \mathcal{M},\tau\right)  $,
$L^{l}\left(  \mathcal{M},\tau\right)  \subseteq L^{p}\left(  \mathcal{M}%
,\tau\right)  $ and
\[
S_{1}^{p}=L^{p}\left(  \mathcal{M},\tau\right)  _{he}^{\ast}\cap
\operatorname{ball}L^{q}\left(  \mathcal{M},\tau\right)  \subseteq
L^{l}\left(  \mathcal{M},\tau\right)  _{he}^{\ast}\cap\varepsilon
\operatorname{ball}L^{m}\left(  \mathcal{M},\tau\right)  =S_{\varepsilon}^{l}%
\]
for all $\varepsilon\geq\tau\left(  e\right)  ^{\left(  q-m\right)  /qm}%
=\tau\left(  e\right)  ^{\left(  l-p\right)  /lp}\geq1$. It follows that
$L^{l}\left(  \mathcal{M},\tau\right)  _{\varepsilon}^{+}\subseteq
L^{p}\left(  \mathcal{M},\tau\right)  _{+}$ for all $\varepsilon\geq
\tau\left(  e\right)  ^{\left(  l-p\right)  /lp}$. If $l=\infty$ then
$L^{\infty}\left(  \mathcal{M},\tau\right)  _{\varepsilon}^{+}=\mathcal{M}%
_{\varepsilon}^{+}$ by Proposition \ref{propFVNA1}, and $\left(  l-p\right)
/lp=1/p$. In particular, $\mathcal{M}_{\varepsilon}^{+}\subseteq L^{p}\left(
\mathcal{M},\tau\right)  _{+}$ for all $\varepsilon\geq\tau\left(  e\right)
^{1/p}$.

Finally, for every $\varepsilon\geq\tau\left(  e\right)  $ take $r$ with
$\varepsilon\tau\left(  e\right)  ^{-1/p}\geq r\geq\tau\left(  e\right)
^{1/q}$. Using again (\ref{iball}), we derive that $\operatorname{ball}%
L^{\infty}\left(  \mathcal{M},\tau\right)  \subseteq\tau\left(  e\right)
^{1/q}\operatorname{ball}L^{q}\left(  \mathcal{M},\tau\right)  $,
$\operatorname{ball}L^{q}\left(  \mathcal{M},\tau\right)  \subseteq\tau\left(
e\right)  ^{\left(  q-1\right)  /q}\operatorname{ball}L^{1}\left(
\mathcal{M},\tau\right)  $ and $L^{\infty}\left(  \mathcal{M},\tau\right)
\subseteq L^{p}\left(  \mathcal{M},\tau\right)  \subseteq L^{1}\left(
\mathcal{M},\tau\right)  $. It follows that
\begin{align*}
S_{1}^{1}  &  =L^{1}\left(  \mathcal{M},\tau\right)  _{he}^{\ast}%
\cap\operatorname{ball}L^{\infty}\left(  \mathcal{M},\tau\right)  \subseteq
L^{p}\left(  \mathcal{M},\tau\right)  _{he}^{\ast}\cap r\operatorname{ball}%
L^{q}\left(  \mathcal{M},\tau\right)  =S_{r}^{p}\\
&  \subseteq L^{\infty}\left(  \mathcal{M},\tau\right)  _{he}^{\ast}\cap
r\tau\left(  e\right)  ^{1/p}\operatorname{ball}L^{1}\left(  \mathcal{M}%
,\tau\right)  \subseteq\mathcal{M}_{he}^{\ast}\cap\varepsilon
\operatorname{ball}L^{1}\left(  \mathcal{M},\tau\right) \\
&  \subseteq\mathcal{M}_{he}^{\ast}\cap\varepsilon\operatorname{ball}%
\mathcal{M}^{\ast}=S_{\varepsilon}^{\infty},
\end{align*}
which in turn implies that $\mathcal{M}_{\varepsilon}^{+}\subseteq
L^{p}\left(  \mathcal{M},\tau\right)  _{r}^{+}\subseteq L^{1}\left(
\mathcal{M},\tau\right)  _{+}$.
\end{proof}

\subsection{The concluding remarks}

\textbf{1. }As we have seen above in Section \ref{secLP} the projective
$L^{p}$-cones and the original cones of the (commutative) $L^{p}$-spaces are
not comparable apart from the finite dimensional case. The same happens in the
case of a finite von Neumann algebra. Note that the inclusion $L^{1}\left(
\mathcal{M},\tau\right)  _{+}\cap\mathcal{M\subseteq M}_{+}$ may not be true
(see Theorem \ref{thmainComm}). The following concrete example presents an
interest. Suppose that $\mathcal{M=}L^{\infty}\left[  0,2\right]  $ is the
abelian von Neumann algebra of Lebesgue measure $dt$ on the compact interval
$\left[  0,2\right]  $ equipped with the trace $\tau\left(  x\right)
=\int_{0}^{2}x\left(  t\right)  dt$. Notice that $e=1\in L^{1}\left[
0,2\right]  $, $\left\Vert e\right\Vert _{1}=\int dt=2>1$, and%
\[
S_{\varepsilon}^{1}=\left\{  y\in\varepsilon\operatorname{ball}L^{\infty
}\left[  0,2\right]  _{h}:\int ydt=1\right\}
\]
for $\varepsilon\geq1$. For every $0<r\leq-2^{-1}+\sqrt{1/\varepsilon+1/4}$
define $x=-r\left[  0,r\right]  +\left[  r,2\right]  \in L^{1}\left[
0,2\right]  $. If $y\in S_{\varepsilon}^{1}$ then $\left\vert y\left(
t\right)  \right\vert \leq\varepsilon$ for almost all $t$, and
\begin{align*}
\tau\left(  xy\right)   &  =-r\int_{0}^{r}ydt+\int_{r}^{2}ydt=-r\int_{0}%
^{r}ydt+1-\int_{0}^{r}ydt=1-\left(  1+r\right)  \int_{0}^{r}ydt\\
&  \geq1-\left(  1+r\right)  \int_{0}^{r}\varepsilon dt=1-\varepsilon
r-\varepsilon r^{2}\geq0.
\end{align*}
Thus $x\in L^{1}\left[  0,2\right]  _{+}\cap L^{\infty}\left[  0,2\right]  $
whereas $x\notin L^{\infty}\left[  0,2\right]  _{+}$. Thus an essentially
bounded $L^{1}\left[  0,2\right]  _{+}$-positive function could take negative
values on a set of positive measure.

\textbf{2. }The projective positivity allows us to define the operator system
sums and the projective tensor product of the operator systems.

\section{Appendix: Functionals on operator systems\label{SecApp1}}

In this section we support the paper with a supplementary material on operator system.

\subsection{Positive functionals on operator systems}

Let $V$ be an operator system $V$ on a Hilbert space $H$. Put $V_{h}^{2}%
=V_{h}\times V_{h}$ with its subset $V_{+}^{2}=V_{+}\times V_{+}$. For $a\in
V_{h}^{2}$, $b\in V$ and $\varepsilon>0$, we put $a+\varepsilon=\left(
a_{1}+\varepsilon e,a_{2}+\varepsilon e\right)  \in V_{h}^{2}$ and
\[
x\left(  a,b\right)  =\left[
\begin{array}
[c]{cc}%
a_{1} & b\\
b^{\ast} & a_{2}%
\end{array}
\right]  \in M_{2}\left(  V\right)  _{h}\subseteq\mathcal{B}\left(
H^{2}\right)  .
\]
If $a\in V_{+}^{2}$, we use the notation $x_{a,\varepsilon,b}=\left(
a_{1}+\varepsilon e\right)  ^{-1/2}b\left(  a_{2}+\varepsilon e\right)
^{-1/2}\in\mathcal{B}\left(  H\right)  $. The following Paulsen's trick is
well known.

\begin{lemma}
\label{lemTBT1}The matrix $x\left(  a,b\right)  $ is positive iff $a\in
V_{+}^{2}$ and $x_{a,\varepsilon,b}\in\operatorname{ball}\mathcal{B}\left(
H\right)  $ for all $\varepsilon>0$.
\end{lemma}

\begin{proof}
If $x\left(  a,b\right)  \in M_{2}\left(  V\right)  _{+}$ then $a_{i}%
=\epsilon_{i}x\left(  a,b\right)  \epsilon_{i}^{\ast}\in V_{+}$ and
$a+\varepsilon\in V_{+}^{2}$ for all $\varepsilon>0$, where $\epsilon
_{1}=\left[
\begin{array}
[c]{cc}%
1 & 0
\end{array}
\right]  $, $\epsilon_{2}=\left[
\begin{array}
[c]{cc}%
0 & 1
\end{array}
\right]  $. In particular, $x\left(  \left(  a+\varepsilon\right)
^{-1/2},0\right)  =\left(  a_{1}+\varepsilon e\right)  ^{-1/2}\oplus\left(
a_{2}+\varepsilon e\right)  ^{-1/2}\geq0$ and $x\left(  a+\varepsilon
,b\right)  =x\left(  a,b\right)  +\varepsilon\left(  e\oplus e\right)  \geq0$.
It follows that
\[
\left[
\begin{array}
[c]{cc}%
e & x_{a,\varepsilon,b}\\
x_{a,\varepsilon,b}^{\ast} & e
\end{array}
\right]  =x\left(  \left(  a+\varepsilon\right)  ^{-1/2},0\right)  x\left(
a+\varepsilon,b\right)  x\left(  \left(  a+\varepsilon\right)  ^{-1/2}%
,0\right)  \geq0,
\]
which means that $\left\Vert x_{a,\varepsilon,b}\right\Vert \leq1$ (see
\cite[1.3.2]{ER}). Conversely, if $a\in\left(  V_{+}\right)  ^{2}$ and
$x_{a,\varepsilon,b}\in\operatorname{ball}\mathcal{B}\left(  H\right)  $, then
as above
\[
x\left(  a+\varepsilon,b\right)  =x\left(  \left(  a+\varepsilon\right)
^{1/2},0\right)  \left[
\begin{array}
[c]{cc}%
e & x_{a,\varepsilon,b}\\
x_{a,\varepsilon,b}^{\ast} & e
\end{array}
\right]  x\left(  \left(  a+\varepsilon\right)  ^{1/2},0\right)  \geq0
\]
for all $\varepsilon>0$. By the continuity argument, we deduce that $x\left(
a,b\right)  \geq0$.
\end{proof}

\begin{corollary}
\label{corTBT1}Let $\alpha=\left(  \alpha_{1},\alpha_{2}\right)  \in
\mathbb{R}^{2}$ and $\beta\in\mathbb{C}$. Then
\[
x\left(  \alpha,\beta\right)  \in M_{2,+}\Longleftrightarrow\alpha
\in\mathbb{R}_{+}^{2}\text{ and }\left\vert \beta\right\vert \leq\alpha
_{1}^{1/2}\alpha_{2}^{1/2}.
\]

\end{corollary}

\begin{proof}
By Lemma \ref{lemTBT1}, $x\left(  \alpha,\beta\right)  \in M_{2,+}$ iff
$\alpha\in\mathbb{R}_{+}^{2}$ and $\left\vert x_{\alpha,\varepsilon,\beta
}\right\vert \leq1$ for all $\varepsilon>0$. The latter means that $\left\vert
\beta\right\vert \leq\left(  \alpha_{1}+\varepsilon1\right)  ^{1/2}\left(
\alpha_{2}+\varepsilon1\right)  ^{1/2}$ for all $\varepsilon>0$. Thus
$\left\vert \beta\right\vert \leq\alpha_{1}^{1/2}\alpha_{2}^{1/2}$.\bigskip
\end{proof}

Based on Corollary \ref{corTBT1}, we can provide more elegant proof of
Cauchy-Schwarz inequality.

\begin{corollary}
\label{corTBT2}Let $\mathcal{A}$ be a unital $C^{\ast}$-algebra and let
$\mu:\mathcal{A\rightarrow}\mathbb{C}$ be a positive functional. Then
\[
\left\vert \mu\left(  a_{1}^{\ast}a_{2}\right)  \right\vert \leq\mu\left(
a_{1}^{\ast}a_{1}\right)  ^{1/2}\mu\left(  a_{2}^{\ast}a_{2}\right)  ^{1/2}%
\]
for all $a_{1},a_{2}\in\mathcal{A}$.
\end{corollary}

\begin{proof}
Put $\alpha=\left(  \mu\left(  a_{1}^{\ast}a_{1}\right)  ,\mu\left(
a_{2}^{\ast}a_{2}\right)  \right)  \in\mathbb{R}_{+}^{2}$ and $\beta
=\mu\left(  a_{1}^{\ast}a_{2}\right)  $. Then $x\left(  \alpha,\beta\right)
\in M_{2}=\mathcal{B}\left(  \mathbb{C}^{2}\right)  $ and
\[
\left(  x\left(  \alpha,\beta\right)  \zeta,\zeta\right)  =\sum_{i,j}%
\mu\left(  a_{i}^{\ast}a_{j}\right)  \zeta_{i}^{\ast}\zeta_{j}=\mu\left(
\left(  a_{1}\zeta_{1}+a_{2}\zeta_{2}\right)  ^{\ast}\left(  a_{1}\zeta
_{1}+a_{2}\zeta_{2}\right)  \right)  \geq0
\]
for all $\zeta=\left(  \zeta_{1},\zeta_{2}\right)  \in\mathbb{C}^{2}$. Thus
$x\left(  \alpha,\beta\right)  \in M_{2,+}$. Using Corollary \ref{corTBT1}, we
conclude that $\left\vert \mu\left(  a_{1}^{\ast}a_{2}\right)  \right\vert
=\left\vert \beta\right\vert \leq\alpha_{1}^{1/2}\alpha_{2}^{1/2}=\mu\left(
a_{1}^{\ast}a_{1}\right)  ^{1/2}\mu\left(  a_{2}^{\ast}a_{2}\right)  ^{1/2}$.
\end{proof}

Now consider the case of an operator system with a positive functional defined
on it.

\begin{lemma}
\label{lemOS0}Let $\mu:V\rightarrow\mathbb{C}$ be a linear functional. Then
$\mu$ is positive iff $\left\Vert \mu\right\Vert =\left\langle e,\mu
\right\rangle $.
\end{lemma}

\begin{proof}
Suppose $\mu$ is positive. Since $-e\leq\operatorname{ball}V_{h}\leq e$, it
follows that $\left\Vert \mu\right\Vert =\vee\left\vert \left\langle
\operatorname{ball}V_{h},\mu\right\rangle \right\vert \leq\left\langle
e,\mu\right\rangle \leq\left\Vert \mu\right\Vert $ thanks to (\ref{nHf}). Thus
$\left\Vert \mu\right\Vert =\left\langle e,\mu\right\rangle $.

Conversely, suppose $\left\Vert \mu\right\Vert =\left\langle e,\mu
\right\rangle \neq0$. Take $x\in V_{h}$. Taking into account that $\left\vert
\gamma-it\right\vert \leq\sqrt{r^{2}+t^{2}}$, $t\in\mathbb{R}$ whenever
$-r\leq\gamma\leq r$, we conclude that $\left\Vert x+ine\right\Vert ^{2}%
\leq\left\Vert x\right\Vert ^{2}+n^{2}$ (just use the local isomorphism of the
unital $C^{\ast}$-algebra generated by $x$ onto the $C^{\ast}$-algebra
$C\left(  \sigma\left(  x\right)  \right)  $ over the spectrum $\sigma\left(
x\right)  $). Then
\begin{align*}
2n\left\vert \operatorname{Im}\left\langle x,\mu\right\rangle \right\vert  &
=2n\left\Vert \mu\right\Vert \left\vert \operatorname{Im}\left\langle
x,\mu\right\rangle \right\vert \left\Vert \mu\right\Vert ^{-1}\leq\left(
\left\vert \left\langle x,\mu\right\rangle \right\vert ^{2}+2n\left\Vert
\mu\right\Vert \left\vert \operatorname{Im}\left\langle x,\mu\right\rangle
\right\vert \right)  \left\Vert \mu\right\Vert ^{-1}\\
&  =\left(  \left(  \operatorname{Re}\left\langle x,\mu\right\rangle \right)
^{2}+\left(  \operatorname{Im}\left\langle x,\mu\right\rangle +n\left\Vert
\mu\right\Vert \right)  ^{2}-n^{2}\left\Vert \mu\right\Vert ^{2}\right)
\left\Vert \mu\right\Vert ^{-1}\\
&  =\left\vert \left\langle x,\mu\right\rangle +in\left\Vert \mu\right\Vert
\right\vert ^{2}\left\Vert \mu\right\Vert ^{-1}-n^{2}\left\Vert \mu\right\Vert
=\left\vert \left\langle x,\mu\right\rangle +in\left\langle e,\mu\right\rangle
\right\vert ^{2}\left\Vert \mu\right\Vert ^{-1}-n^{2}\left\Vert \mu\right\Vert
\\
&  =\left\vert \left\langle x+ine,\mu\right\rangle \right\vert ^{2}\left\Vert
\mu\right\Vert ^{-1}-n^{2}\left\Vert \mu\right\Vert \\
&  \leq\left\Vert x+ine\right\Vert ^{2}\left\Vert \mu\right\Vert
-n^{2}\left\Vert \mu\right\Vert \leq\left(  \left\Vert x\right\Vert ^{2}%
+n^{2}\right)  \left\Vert \mu\right\Vert -n^{2}\left\Vert \mu\right\Vert
=\left\Vert x\right\Vert ^{2}\left\Vert \mu\right\Vert ,
\end{align*}
which means that $\operatorname{Im}\left\langle x,\mu\right\rangle =0$. Thus
$\left\langle V_{h},\mu\right\rangle \subseteq\mathbb{R}$, that is, $\mu
=\mu^{\ast}$ and $-\left\langle e,\mu\right\rangle \leq\left\langle
x,\mu\right\rangle \leq\left\langle e,\mu\right\rangle $ whenever $x\in V_{h}$
with $-e\leq x\leq e$. Take $y\in V_{+}$, $0\leq y\leq e$. Then $-e\leq x\leq
e$ for $x=2y-e\in V_{h}$, which in turn implies that $\left\langle
y,\mu\right\rangle =2^{-1}\left(  \left\langle x,\mu\right\rangle
+\left\langle e,\mu\right\rangle \right)  \geq0$, that is, $\mu$ is positive.
\end{proof}

\subsection{Orthogonal expansions of a hermitian functional}

In the case of $V=C\left(  \mathcal{T}\right)  $ for a compact Hausdorff
topological space $\mathcal{T}$, we have $V^{\ast}=M\left(  \mathcal{T}%
\right)  $ to be the space of all Radon charges on $\mathcal{T}$. The real
vector space $M\left(  \mathcal{T}\right)  _{h}$ turns out to be a complete
lattice \cite[4.2 Theorem 18]{Sw}. In particular, every $\mu\in M\left(
\mathcal{T}\right)  _{h}$ admits the orthogonal (unique) expansion $\mu
=\mu_{+}-\mu_{-}$ into a difference of Radon measures $\mu_{+},\mu_{-}\in
M\left(  \mathcal{T}\right)  _{+}$ such that $\mu_{+}\wedge\mu_{-}=0$ and
$\left\Vert \mu\right\Vert =\left\Vert \mu_{+}\right\Vert +\left\Vert \mu
_{-}\right\Vert $. In the case of an operator system $V$ and $\mu\in
V_{h}^{\ast}$ we say that $\mu$ admits \textit{an orthogonal expansion} if
$\mu=\mu_{1}-\mu_{2}$ for some $\mu_{1},\mu_{2}\in V_{+}^{\ast}$ such that
$\left\Vert \mu\right\Vert =\left\Vert \mu_{1}\right\Vert +\left\Vert \mu
_{2}\right\Vert $. The presence of such expansions will be proven below.

\begin{lemma}
\label{corOSS1}Let $V$ be an operator system and let $\mu\in V_{h}^{\ast}$,
$\mu=\mu_{1}-\mu_{2}$ with $\mu_{1},\mu_{2}\in V_{+}^{\ast}$. Then $\mu
=\mu_{1}-\mu_{2}$ is an orthogonal expansion iff for every $\varepsilon>0$
there corresponds $y\in\operatorname{ball}V_{+}$ such that $\left\langle
e-y,\mu_{1}\right\rangle <\varepsilon$ and $\left\langle y,\mu_{2}%
\right\rangle <\varepsilon$.
\end{lemma}

\begin{proof}
Suppose $\mu=\mu_{1}-\mu_{2}$ is an orthogonal expansion. Using Lemma
\ref{lemOS0}, we deduce that
\[
\mu_{1}\left(  e\right)  +\mu_{2}\left(  e\right)  -\varepsilon=\left\Vert
\mu_{1}\right\Vert +\left\Vert \mu_{2}\right\Vert -\varepsilon=\left\Vert
\mu\right\Vert -\varepsilon\leq\mu_{1}\left(  x\right)  -\mu_{2}\left(
x\right)
\]
for some $x\in\operatorname{ball}V_{h}$, that is, $\mu_{1}\left(  e-x\right)
+\mu_{2}\left(  e+x\right)  \leq\varepsilon$. But $0\leq e-x\leq2e$ and $0\leq
e+x\leq2e$, therefore $y=2^{-1}\left(  e+x\right)  \in\operatorname{ball}%
V_{+}$, $e-y=2^{-1}\left(  e-x\right)  \in\operatorname{ball}V_{+}$ and
$0\leq\mu_{1}\left(  e-y\right)  +\mu_{2}\left(  y\right)  \leq2^{-1}%
\varepsilon<\varepsilon$.

Conversely, suppose the above stated property holds for the expansion $\mu
=\mu_{1}-\mu_{2}$. Take $\varepsilon>0$ and the related $y\in
\operatorname{ball}V_{+}$. Since $-e\leq2y-e\leq e$, it follows from Lemma
\ref{lemOS0} that
\begin{align*}
\left\Vert \mu_{1}\right\Vert +\left\Vert \mu_{2}\right\Vert  &  =\mu
_{1}\left(  e\right)  +\mu_{2}\left(  e\right)  =\mu_{1}\left(  2y-e\right)
+2\mu_{1}\left(  e-y\right)  +\mu_{2}\left(  e-2y\right)  +2\mu_{2}\left(
y\right) \\
&  <\mu_{1}\left(  2y-e\right)  +\mu_{2}\left(  e-2y\right)  +4\varepsilon
=\mu\left(  2y-e\right)  +4\varepsilon\leq\left\Vert \mu\right\Vert
+4\varepsilon.
\end{align*}
Consequently, $\left\Vert \mu_{1}\right\Vert +\left\Vert \mu_{2}\right\Vert
\leq\left\Vert \mu\right\Vert $, which means that $\mu=\mu_{1}-\mu_{2}$ is an
orthogonal expansion.
\end{proof}

\begin{corollary}
\label{corOSS2}Let $\mathcal{A}$ be a unital $C^{\ast}$-algebra and let
$\mu\in\mathcal{A}_{h}^{\ast}$. The orthogonal expansion of $\mu$ is unique.
\end{corollary}

\begin{proof}
Suppose $\mu=\mu_{1}-\mu_{2}=\tau_{1}-\tau_{2}$ are the orthogonal expansions
of $\mu$. By Lemma \ref{corOSS1}, for $\varepsilon>0$ there is $y\in
\operatorname{ball}\mathcal{A}_{+}$ such that $\mu_{1}\left(  e-y\right)
<\varepsilon$ and $\mu_{2}\left(  y\right)  <\varepsilon$. Then
\begin{align*}
\tau_{1}\left(  y\right)   &  \geq\tau_{1}\left(  y\right)  -\tau_{2}\left(
y\right)  =\mu_{1}\left(  y\right)  -\mu_{2}\left(  y\right)  >\mu_{1}\left(
e\right)  -\varepsilon-\mu_{2}\left(  y\right)  >\mu_{1}\left(  e\right)
-2\varepsilon,\\
\tau_{2}\left(  e-y\right)   &  \geq\tau_{2}\left(  e-y\right)  -\tau
_{1}\left(  e-y\right)  =\mu_{2}\left(  e-y\right)  -\mu_{1}\left(
e-y\right)  >\mu_{2}\left(  e\right)  -2\varepsilon.
\end{align*}
By Lemma \ref{lemOS0}, we deduce that
\begin{align*}
\tau_{1}\left(  y\right)  +\tau_{2}\left(  e-y\right)   &  >\mu_{1}\left(
e\right)  +\mu_{2}\left(  e\right)  -4\varepsilon=\left\Vert \mu
_{1}\right\Vert +\left\Vert \mu_{2}\right\Vert -4\varepsilon=\left\Vert
\tau_{1}\right\Vert +\left\Vert \tau_{2}\right\Vert -4\varepsilon\\
&  =\tau_{1}\left(  e\right)  +\tau_{2}\left(  e\right)  -4\varepsilon,
\end{align*}
that is, $\tau_{1}\left(  e-y\right)  +\tau_{2}\left(  y\right)
<4\varepsilon$. Consequently, $\mu_{1}\left(  e-y\right)  <\varepsilon$,
$\mu_{2}\left(  y\right)  <\varepsilon$, $\tau_{1}\left(  e-y\right)
<4\varepsilon$ and $\tau_{2}\left(  y\right)  <4\varepsilon$. Using Corollary
\ref{corTBT2}, we deduce that
\begin{align*}
\left\vert \mu_{2}\left(  xy\right)  \right\vert  &  \leq\mu_{2}\left(
xx^{\ast}\right)  ^{1/2}\mu_{2}\left(  y^{2}\right)  ^{1/2}\leq\left\Vert
\mu_{2}\right\Vert ^{1/2}\left\Vert x\right\Vert \mu_{2}\left(  y\right)
^{1/2}\leq\left\Vert \mu\right\Vert ^{1/2}\varepsilon^{1/2},\\
\left\vert \tau_{2}\left(  xy\right)  \right\vert  &  \leq\left\Vert \tau
_{2}\right\Vert ^{1/2}\left\Vert x\right\Vert \tau_{2}\left(  y\right)
^{1/2}\leq\left\Vert \mu\right\Vert ^{1/2}2\varepsilon^{1/2},\\
\left\vert \mu_{1}\left(  x\left(  e-y\right)  \right)  \right\vert  &
\leq\left\Vert \mu_{1}\right\Vert ^{1/2}\left\Vert x\right\Vert \mu_{1}\left(
e-y\right)  ^{1/2}\leq\left\Vert \mu\right\Vert ^{1/2}\varepsilon^{1/2},\\
\left\vert \tau_{1}\left(  x\left(  e-y\right)  \right)  \right\vert  &
\leq\left\Vert \tau_{1}\right\Vert ^{1/2}\left\Vert x\right\Vert \tau
_{1}\left(  e-y\right)  ^{1/2}\leq\left\Vert \mu\right\Vert ^{1/2}%
2\varepsilon^{1/2}%
\end{align*}
for all $x\in\operatorname{ball}\mathcal{A}$. But $\mu_{1}-\tau_{1}=\mu
_{2}-\tau_{2}$, therefore
\begin{align*}
\mu_{1}\left(  x\right)  -\tau_{1}\left(  x\right)   &  =\mu_{1}\left(
xy\right)  -\tau_{1}\left(  xy\right)  +\mu_{1}\left(  x\left(  e-y\right)
\right)  -\tau_{1}\left(  x\left(  e-y\right)  \right) \\
&  =\mu_{2}\left(  xy\right)  -\tau_{2}\left(  xy\right)  +\mu_{1}\left(
x\left(  e-y\right)  \right)  -\tau_{1}\left(  x\left(  e-y\right)  \right)  .
\end{align*}
It yields that $\left\vert \mu_{1}\left(  x\right)  -\tau_{1}\left(  x\right)
\right\vert \leq6\left\Vert \mu\right\Vert ^{1/2}\varepsilon^{1/2}$. Whence
$\mu_{1}=\tau_{1}$ and $\mu_{2}=\tau_{2}$.
\end{proof}

Now we can prove an operator system (or noncommutative) version of an
orthogonal expansion result for hermitian functionals on an operator system.

\begin{theorem}
\label{thOS1}Let $V$ be an operator system and let $\mu\in V_{h}^{\ast}$. Then
$\mu$ admits an orthogonal expansion $\mu=\mu_{+}-\mu_{-}$, $\mu_{+},\mu
_{-}\in V_{+}^{\ast}$, $\left\Vert \mu\right\Vert =\left\Vert \mu
_{+}\right\Vert +\left\Vert \mu_{-}\right\Vert $, which may not be unique. If
$V=\mathcal{A}$ is a unital $C^{\ast}$-algebra then every $\mu\in
\mathcal{A}_{h}^{\ast}$ admits a unique orthogonal expansion.
\end{theorem}

\begin{proof}
Let $\mu\in V_{h}^{\ast}$. By Corollary \ref{corOKE31} and Proposition
\ref{propAp1}, $\mu$ is identified with a hermitian functional on $\Phi\left(
V,\left\Vert \cdot\right\Vert _{e}\right)  $ up to an isometry. By Hahn-Banach
extension theorem, there is a functional $\tau\in M\left(  S\left(
V_{+}\right)  \right)  $, $\left\Vert \tau\right\Vert =\left\Vert
\mu\right\Vert $ such that $\tau\Phi=\mu$. But
\[
\left\langle \Phi\left(  x\right)  ,\tau^{\ast}\right\rangle =\left\langle
\Phi\left(  x\right)  ^{\ast},\tau\right\rangle ^{\ast}=\left\langle
\Phi\left(  x^{\ast}\right)  ,\tau\right\rangle ^{\ast}=\left\langle x^{\ast
},\mu\right\rangle ^{\ast}=\left\langle x,\mu^{\ast}\right\rangle
=\left\langle x,\mu\right\rangle =\left\langle \Phi\left(  x\right)
,\tau\right\rangle
\]
for all $x\in V$. It follows that $\left\langle \Phi\left(  x\right)
,\operatorname{Re}\tau\right\rangle =\left\langle \Phi\left(  x\right)
,\tau\right\rangle =\left\langle x,\mu\right\rangle $ for all $x\in V$, and
$\left\Vert \mu\right\Vert \leq\left\Vert \operatorname{Re}\tau\right\Vert
\leq\left\Vert \tau\right\Vert =\left\Vert \mu\right\Vert $ by Corollary
\ref{corOKE31}. Thus we can assume that $\tau\in M\left(  S\left(
V_{+}\right)  \right)  _{h}$. But $\tau=\tau_{+}-\tau_{-}$ with $\tau_{+}%
,\tau_{-}\in M\left(  S\left(  V_{+}\right)  \right)  _{+}$ and $\left\Vert
\tau\right\Vert =\left\Vert \tau_{+}\right\Vert +\left\Vert \tau
_{-}\right\Vert $. Put $\mu_{+}=\tau_{+}\Phi$ and $\mu_{-}=\tau_{-}\Phi$. Then
$\mu_{+},\mu_{-}\in V_{+}^{\ast}$ and $\mu=\tau\Phi=\tau_{+}\Phi-\tau_{-}%
\Phi=\mu_{+}-\mu_{-}$. By Lemma \ref{lemOS0}, we deduce that
\[
\left\Vert \mu_{+}\right\Vert +\left\Vert \mu_{-}\right\Vert =\left\langle
e,\mu_{+}\right\rangle +\left\langle e,\mu_{-}\right\rangle =\left\langle
1,\tau_{+}\right\rangle +\left\langle 1,\tau_{-}\right\rangle =\left\Vert
\tau_{+}\right\Vert +\left\Vert \tau_{-}\right\Vert =\left\Vert \tau
\right\Vert =\left\Vert \mu\right\Vert .
\]
Hence $\mu$ admits an orthogonal expansion $\mu=\mu_{+}-\mu_{-}$, $\left\Vert
\mu\right\Vert =\left\Vert \mu_{+}\right\Vert +\left\Vert \mu_{-}\right\Vert $.

Now let us provide an example of different orthogonal expansions of a
hermitian functional on an operator system. Let $v_{a,b}^{\lambda}=\left[
\begin{array}
[c]{cc}%
\lambda & a\\
b^{\ast} & \lambda
\end{array}
\right]  $ be $2\times2$-matrix over $\mathbb{C}$. If $b=a$ we write
$v_{a}^{\lambda}$ instead of $v_{a,a}^{\lambda}$. Note that $v_{a}^{\lambda
}=x\left(  \lambda,a\right)  $. Put
\[
v_{a,b}^{\lambda,\theta}=v_{a,b}^{\lambda}\oplus v_{a,b}^{\theta}=\left[
\begin{array}
[c]{cccc}%
\lambda & a & 0 & 0\\
b^{\ast} & \lambda & 0 & 0\\
0 & 0 & \theta & a\\
0 & 0 & b^{\ast} & \theta
\end{array}
\right]  \in M_{4}%
\]
and $V=\left\{  v_{a,b}^{\lambda,\theta}:\lambda,\theta,a,b\in\mathbb{C}%
\right\}  $, which is a unital subspace of $M_{4}$. Notice that $\left(
v_{a,b}^{\lambda,\theta}\right)  ^{\ast}=v_{b,a}^{\lambda^{\ast},\theta^{\ast
}}\in V$, that is, $V$ is an operator system on $\mathbb{C}^{4}$. The real
vector space of all hermitian elements from $V$ consists of matrices
$v_{a}^{\lambda,\theta}=v_{a}^{\lambda}\oplus v_{a}^{\theta}$, $\lambda
,\theta\in\mathbb{R}$. Note that $v_{a}^{\lambda,\theta}\geq0$ iff $\left\vert
a\right\vert \leq\lambda$ and $\left\vert a\right\vert \leq\theta$ thanks to
Corollary \ref{corTBT1}. Define the following positive linear functionals
\[
\mu_{1},\mu_{2}:V\rightarrow\mathbb{C},\quad\left\langle v_{a,b}%
^{\lambda,\theta},\mu_{1}\right\rangle =\lambda,\quad\left\langle
v_{a,b}^{\lambda,\theta},\mu_{2}\right\rangle =\theta,
\]
and put $\mu=\mu_{1}-\mu_{2}$. Notice that $\left\Vert \mu_{i}\right\Vert
=\left\langle e,\mu_{i}\right\rangle =\left\langle v_{0}^{1,1},\mu
_{i}\right\rangle =1$ and $\left\Vert \mu\right\Vert \leq\left\Vert \mu
_{1}\right\Vert +\left\Vert \mu_{2}\right\Vert =2$. But $\left\langle
v_{a,b}^{\lambda,\theta},\mu\right\rangle =\lambda-\theta$ and $2=\left\langle
v_{0}^{1,-1},\mu\right\rangle \leq\left\Vert \mu\right\Vert \left\Vert
v_{0}^{1,-1}\right\Vert =\left\Vert \mu\right\Vert $. Hence $\left\Vert
\mu\right\Vert =\left\Vert \mu_{1}\right\Vert +\left\Vert \mu_{2}\right\Vert
$, which means that $\mu=\mu_{1}-\mu_{2}$ is an orthogonal expansion of $\mu$.
Further, define the following linear functionals
\[
\tau_{1},\tau_{2}:V\rightarrow\mathbb{C},\quad\left\langle v_{a,b}%
^{\lambda,\theta},\tau_{1}\right\rangle =\lambda+a,\quad\left\langle
v_{a,b}^{\lambda,\theta},\tau_{2}\right\rangle =\theta+a.
\]
If $v_{a}^{\lambda,\theta}\geq0$ then $\lambda+a\geq0$ and $\theta+a\geq0$,
which means that $\tau_{1}$ and $\tau_{2}$ are positive functionals. In
particular, $\left\Vert \tau_{i}\right\Vert =\left\langle e,\tau
_{i}\right\rangle =\left\langle v_{0}^{1,1},\tau_{i}\right\rangle =1$. But%
\[
\left\langle v_{a,b}^{\lambda,\theta},\tau_{1}-\tau_{2}\right\rangle =\left(
\lambda+a\right)  -\left(  \theta+a\right)  =\lambda-\theta=\left\langle
v_{a,b}^{\lambda,\theta},\mu\right\rangle ,
\]
that is, $\mu=\tau_{1}-\tau_{2}$ is another orthogonal expansion.

Finally, let $\mathcal{A}$ be a unital $C^{\ast}$-algebra with $\mu
\in\mathcal{A}_{h}^{\ast}$. As we have just seen above $\mu$ admits an
orthogonal expansion $\mu=\mu_{+}-\mu_{-}$, $\mu_{+},\mu_{-}\in\mathcal{A}%
_{+}^{\ast}$, $\left\Vert \mu\right\Vert =\left\Vert \mu_{+}\right\Vert
+\left\Vert \mu_{-}\right\Vert $. The uniqueness of such expansion follows
from Corollary \ref{corOSS2}.
\end{proof}

\subsection{The extreme points\label{subsecEP}}

Recall that an extreme point of a convex subset $\mathcal{C\subseteq}X$ of a
vector space $X$ is a point of $\mathcal{C}$ that can not be expressed as a
nontrivial convex combination of elements from $\mathcal{C}$. The set of all
extreme points of $\mathcal{C}$ is called \textit{the extremal boundary of}
$\mathcal{C}$ denoted by $\partial\mathcal{C}$. The well known (see
\cite[2.5.4]{Ped}) Krein-Milman theorem asserts that the convex hull
$\operatorname{co}\left(  \partial\mathcal{C}\right)  $ of $\partial
\mathcal{C}$ is weakly dense in $\mathcal{C}$ for a convex, compact subset
$\mathcal{C}$ of a vector space $X$ equipped with the weak topology
$\sigma\left(  X,Y\right)  $ of a dual pair $\left(  X,Y\right)  $, that is,
$\mathcal{C=}\operatorname{co}\left(  \partial\mathcal{C}\right)  ^{-w}$.

As above let $\mathcal{T}$ be a compact Hausdorff topological space. It is not
hard to see that the extremal boundary of $\operatorname{ball}C\left(
\mathcal{T}\right)  $ consists of unitary functions, that is,
\begin{equation}
\partial\operatorname{ball}C\left(  \mathcal{T}\right)  =\left\{  f\in
C_{0}\left(  \mathcal{T}\right)  :\left\vert f\right\vert =1\right\}  .
\label{ex1}%
\end{equation}
Indeed, take $f\in\operatorname{ball}C\left(  \mathcal{T}\right)  $. If
$\left\vert f\left(  t_{0}\right)  \right\vert <1$ for some $t_{0}%
\in\mathcal{T}$ then $\alpha=\sup\left\vert f\left(  U\right)  \right\vert <1$
for a compact neighborhood $U$ of $t_{0}$, and we can choose $h\in C\left(
\mathcal{T}\right)  $ such that $0\leq h\leq1$, $h\left(  t_{0}\right)  =1$
and $\operatorname{supp}h\subseteq U$ (see \cite[1.7.5]{Ped}). Put $g=\left(
1-\alpha\right)  h$. Then $f\pm g\in\operatorname{ball}C\left(  \mathcal{T}%
\right)  $ and $f=\frac{1}{2}\left(  \left(  f+g\right)  +\left(  f-g\right)
\right)  $, that is, $f\notin\partial\operatorname{ball}C\left(
\mathcal{T}\right)  $. Conversely, if $f$ is unitary and $f=\frac{1}{2}\left(
g+h\right)  $ for some $g,h\in\operatorname{ball}C\left(  \mathcal{T}\right)
$, then $f\left(  t\right)  $ is an extreme point of the unit circle in
$\mathbb{C}$, therefore $g\left(  t\right)  =h\left(  t\right)  $, that is,
$f\in\partial\operatorname{ball}C\left(  \mathcal{T}\right)  $.

Now consider the cone $C\left(  \mathcal{T}\right)  _{+}$ of positive
functions. Then $\partial\operatorname{ball}C\left(  \mathcal{T}\right)  _{+}$
consists of projections in $C\left(  \mathcal{T}\right)  $, that is,
\begin{equation}
\partial\operatorname{ball}C\left(  \mathcal{T}\right)  _{+}=\left\{  \left[
S\right]  :S\subseteq\mathcal{T}\text{ is an open compact}\right\}  ,
\label{ex2}%
\end{equation}
where $\left[  S\right]  $ is the characteristic function of a subset $S$.
Indeed, if $\left[  S\right]  =\frac{1}{2}\left(  g+h\right)  $ for some
$g,h\in\operatorname{ball}C\left(  \mathcal{T}\right)  _{+}$, then $g\left(
t\right)  =h\left(  t\right)  =1$, $t\in S$ and $0\leq g\left(  t\right)
=h\left(  t\right)  =0$ for all $t\in\mathcal{T}-S$, that is, $g=h=\left[
S\right]  $ and $\left[  S\right]  \in\partial\operatorname{ball}C\left(
\mathcal{T}\right)  _{+}$. Conversely, take $f\in\partial\operatorname{ball}%
C\left(  \mathcal{T}\right)  _{+}$. If $0<f\left(  t_{0}\right)  <1$ for some
$t_{0}\in\mathcal{T}$ then as above, choose the function $h\in
\operatorname{ball}C\left(  \mathcal{T}\right)  _{+}$ and small $\varepsilon
>0$ such that $f\left(  1\pm\varepsilon h\right)  \in\operatorname{ball}%
C\left(  \mathcal{T}\right)  _{+}$. But $f=\frac{1}{2}\left(  f\left(
1+\varepsilon h\right)  +f\left(  1-\varepsilon h\right)  \right)  $, which is
impossible. Thus $f$ can only take the values $0$ or $1$, that is, $f=\left[
S\right]  $ for a clopen subset $S\subseteq\mathcal{T}$.

In the case of a $C^{\ast}$-algebra $\mathcal{A}$ we have the following
nontrivial result that $\partial\operatorname{ball}\mathcal{A}\neq\varnothing$
iff $\mathcal{A}$ is unital. We skip the details.

Now let $P\left(  \mathcal{T}\right)  $ be a convex subset in $M\left(
\mathcal{T}\right)  $ of probability measures on $\mathcal{T}$. Recall that a
Radon measure $\mu\in M\left(  \mathcal{T}\right)  $ on $\mathcal{T}$ is
called a probability measure if $\left\Vert \mu\right\Vert \leq1$ and
$\left\langle 1,\mu\right\rangle =1$. By Lemma \ref{lemOS0}, we deduce that%
\[
P\left(  \mathcal{T}\right)  =\left\{  \mu\in\operatorname{ball}M\left(
\mathcal{T}\right)  :\left\Vert \mu\right\Vert =\left\langle 1,\mu
\right\rangle =1\right\}  =S\left(  C\left(  \mathcal{T}\right)  _{+}\right)
\]
is the state space of the cone $C\left(  \mathcal{T}\right)  _{+}$. In
particular, $P\left(  \mathcal{T}\right)  $ is a convex, $w^{\ast}$-compact
set, which contains all atomic measures $\mu=\sum_{t\in S}c_{t}\delta_{t}$
with $c_{t}\geq0$ and $\sum_{t\in S}c_{t}=1$. Notice each such $\mu$ is a
norm-limit of positive atomic measures with finite support on $\mathcal{T}$.

\begin{proposition}
\label{propBPM}The equality $\partial P\left(  \mathcal{T}\right)  =\left\{
\delta_{t}:t\in\mathcal{T}\right\}  $ holds. In particular, $P\left(
\mathcal{T}\right)  $ consists of those $\mu\in M\left(  \mathcal{T}\right)
_{+}$ which are $w^{\ast}$-limit of $\sum_{t\in F}c_{t}\delta_{t}$ with
$c_{t}\geq0$, $\sum_{t\in F}c_{t}=1$ and $F\subseteq\mathcal{T}$ is a finite
subset. In particular,
\[
P\left(  \mathcal{T}\right)  =\operatorname{co}\left(  \left\{  \delta
_{t}:t\in\mathcal{T}\right\}  \right)  ^{-w^{\ast}}\text{\quad and \quad
}\operatorname{ball}M\left(  \mathcal{T}\right)  =\left(  \operatorname{abs}%
\left(  \partial P\left(  \mathcal{T}\right)  \right)  \right)  ^{-w^{\ast}}.
\]

\end{proposition}

\begin{proof}
First prove that $\delta_{t}\in\partial P\left(  \mathcal{T}\right)  $ for
every $t$. If $\delta_{t}=r\mu+\left(  1-r\right)  \tau$ for some $\mu,\tau\in
P\left(  \mathcal{T}\right)  $ and $0<r<1$, then $r\left\vert \left\langle
f,\mu\right\rangle \right\vert \leq r\left\langle \left\vert f\right\vert
,\mu\right\rangle =\left\langle \left\vert f\right\vert ,r\mu\right\rangle
\leq\left\langle \left\vert f\right\vert ,\delta_{t}\right\rangle =\left\vert
f\left(  t\right)  \right\vert $. It follows that $\ker\delta_{t}\subseteq
\ker\mu$ or $\mu=\lambda\delta_{t}$ for some $\lambda\in\mathbb{C}$. But
$\lambda=\lambda\left\langle 1,\delta_{t}\right\rangle =\left\langle
1,\mu\right\rangle =1$, that is, $\mu=\delta_{t}$ and $\tau=\delta_{t}$. Thus
$\delta_{t}$ is an extremal point of $P\left(  \mathcal{T}\right)  $.

Now take $\mu\in\partial P\left(  \mathcal{T}\right)  $. Prove that $\mu$ is a
multiplicative functional on the $C^{\ast}$-algebra $C\left(  \mathcal{T}%
\right)  $ (or a character), that is, $\left\langle fg,\mu\right\rangle
=\left\langle f,\mu\right\rangle \left\langle g,\mu\right\rangle $ for all
$f,g\in C\left(  \mathcal{T}\right)  $. Since $C\left(  \mathcal{T}\right)  $
is a linear span of $C\left(  \mathcal{T}\right)  _{+}$, we can assume that
$0\leq f\leq1$. Put $r=\left\langle f,\mu\right\rangle $, thereby $0\leq
r\leq\left\langle 1,\mu\right\rangle =1$. Notice that $\left\vert \left\langle
fg,\mu\right\rangle \right\vert \leq\left\langle \left\vert fg\right\vert
,\mu\right\rangle =\left\langle f\left\vert g\right\vert ,\mu\right\rangle
\leq\left\Vert g\right\Vert _{\infty}\left\langle f,\mu\right\rangle
=r\left\Vert g\right\Vert _{\infty}$ and $\left\vert \left\langle \left(
1-f\right)  g,\mu\right\rangle \right\vert \leq\left\langle \left(
1-f\right)  \left\vert g\right\vert ,\mu\right\rangle \leq\left\Vert
g\right\Vert _{\infty}\left\langle \left(  1-f\right)  ,\mu\right\rangle
=\left\Vert g\right\Vert _{\infty}\left(  1-r\right)  $. It follows that
$\tau_{1}\left(  g\right)  =r^{-1}\left\langle fg,\mu\right\rangle $ and
$\tau_{2}\left(  g\right)  =\left(  1-r\right)  ^{-1}\left\langle \left(
1-f\right)  g,\mu\right\rangle $ are elements of $\operatorname{ball}M\left(
\mathcal{T}\right)  $ whenever $0<r<1$. Actually, they are functionals from
$P\left(  \mathcal{T}\right)  $ and $\mu=r\tau_{1}+\left(  1-r\right)
\tau_{2}$. Taking into account that $\mu\in\partial P\left(  \mathcal{T}%
\right)  $, we deduce that $\mu=\tau_{1}=\tau_{2}$. Hence $\left\langle
fg,\mu\right\rangle =\left\langle f,\mu\right\rangle \left\langle
g,\mu\right\rangle $ whenever $0<r<1$. But the equality holds if $r=0$, for
$\left\vert \left\langle fg,\mu\right\rangle \right\vert \leq r\left\Vert
g\right\Vert _{\infty}$. Since $\left\vert \left\langle \left(  1-f\right)
g,\mu\right\rangle \right\vert \leq\left(  1-r\right)  \left\Vert g\right\Vert
_{\infty}$, the equality holds for $r=1$ too. Whence $\mu$ is a multiplicative functional.

It is well known (see \cite[4.2.5]{Ped}) that the Gelfand transform of the
$C^{\ast}$-algebra $C\left(  \mathcal{T}\right)  $ is the identity mapping.
Therefore $\mu=\delta_{t}$ for some $t\in\mathcal{T}$. To end up the proof we
need to apply Krein-Milman theorem to the convex, $w^{\ast}$-compact set
$P\left(  \mathcal{T}\right)  $ within the dual pair $\left(  M\left(
\mathcal{T}\right)  ,C\left(  \mathcal{T}\right)  \right)  $.

Finally, for a possible $\mu\in\operatorname{ball}M\left(  \mathcal{T}\right)
\backslash\left(  \operatorname{abs}\left(  \partial P\left(  \mathcal{T}%
\right)  \right)  \right)  ^{-w^{\ast}}$ one can find a $w^{\ast}$-continuous
linear functional $f:M\left(  \mathcal{T}\right)  \rightarrow\mathbb{C}$ such
that $\operatorname{Re}\left\langle \operatorname{abs}\left(  \partial
P\left(  \mathcal{T}\right)  \right)  ,f\right\rangle \leq r<\operatorname{Re}%
\left\langle \mu,f\right\rangle $ for some real $r\geq0$ thanks to the
Separation Theorem \cite[2.4.7]{Ped} But $f\in C\left(  \mathcal{T}\right)  $
and $\operatorname{Re}\left\langle \mu,f\right\rangle =\operatorname{Re}%
\left\langle f,\mu\right\rangle \leq\left\vert \left\langle f,\mu\right\rangle
\right\vert \leq\left\Vert f\right\Vert _{\infty}\left\Vert \mu\right\Vert
\leq\left\Vert f\right\Vert _{\infty}$. Moreover, if $f\left(  t\right)
\neq0$ for some $t\in\mathcal{T}$, then using the equality $\partial P\left(
\mathcal{T}\right)  =\left\{  \delta_{t}:t\in\mathcal{T}\right\}  $, we deduce
that $\left\vert f\left(  t\right)  \right\vert =\theta f\left(  t\right)
=\left\langle \theta f,\delta_{t}\right\rangle =\operatorname{Re}\left\langle
\theta f,\delta_{t}\right\rangle =\operatorname{Re}\left\langle \theta
\delta_{t},f\right\rangle \leq r$ for a certain $\theta\in\mathbb{C}$,
$\left\vert \theta\right\vert =1$. Thus $\left\Vert f\right\Vert _{\infty
}=\vee\left\vert f\left(  \mathcal{T}\right)  \right\vert \leq r$ whereas
$r<\left\Vert f\right\Vert _{\infty}$, a contradiction.
\end{proof}

In particular,
\begin{align*}
\operatorname{ball}C\left(  \mathcal{T}\right)   &  =\left(
\operatorname{ball}M\left(  \mathcal{T}\right)  \right)  ^{\circ}=\left(
\left(  \operatorname{abs}\left(  P\left(  \mathcal{T}\right)  \right)
\right)  ^{-w^{\ast}}\right)  ^{\circ}=\left(  \operatorname{abs}\left(
P\left(  \mathcal{T}\right)  \right)  \right)  ^{\circ}=\left(  P\left(
\mathcal{T}\right)  \right)  ^{\circ}\\
&  =S\left(  C\left(  \mathcal{T}\right)  _{+}\right)  ^{\circ},
\end{align*}
which means that $\left\Vert f\right\Vert _{\infty}=\vee\left\vert
\left\langle f,\operatorname{ball}M\left(  \mathcal{T}\right)  \right\rangle
\right\vert =\vee\left\vert \left\langle f,S\left(  C\left(  \mathcal{T}%
\right)  _{+}\right)  \right\rangle \right\vert =\left\Vert f\right\Vert _{e}$
for every $f\in C\left(  \mathcal{T}\right)  $. Thus $\left\Vert
\cdot\right\Vert _{\infty}$ is the state norm $\left\Vert \cdot\right\Vert
_{e}$. Note also that $\left\Vert f\right\Vert _{\infty}=\vee\left\vert
\left\langle f,\partial P\left(  \mathcal{T}\right)  \right\rangle \right\vert
$, $f\in C\left(  \mathcal{T}\right)  $.

\textbf{Acknowledgements. }The author thanks for an opportunity to share the
results of the paper in the seminar Algebras in Analysis by A. Ya. Helemskii
and A. Yu. Pirkovskii. The paper has no funding and there are no relevant
financial or non-financial competing interests to report.

\end{document}